\documentclass{article}

\usepackage{algorithm2e}
\usepackage{amsmath}
\usepackage{amssymb}
\usepackage{amsthm}
\usepackage{changepage}
\usepackage{comment}
\usepackage{enumitem}
\usepackage{float}
\usepackage{geometry}
\usepackage[round]{natbib}
\usepackage{scalerel}
\usepackage{subcaption}
\usepackage{tikz}

\DeclareFontFamily{U} {MnSymbolC}{}

\DeclareFontShape{U}{MnSymbolC}{m}{n}{
	<-6> MnSymbolC5
	<6-7> MnSymbolC6
	<7-8> MnSymbolC7
	<8-9> MnSymbolC8
	<9-10> MnSymbolC9
	<10-12> MnSymbolC10
	<12-> MnSymbolC12}{}
\DeclareFontShape{U}{MnSymbolC}{b}{n}{
	<-6> MnSymbolC-Bold5
	<6-7> MnSymbolC-Bold6
	<7-8> MnSymbolC-Bold7
	<8-9> MnSymbolC-Bold8
	<9-10> MnSymbolC-Bold9
	<10-12> MnSymbolC-Bold10
	<12-> MnSymbolC-Bold12}{}

\DeclareSymbolFont{MnSyC} {U} {MnSymbolC}{m}{n}

\DeclareMathSymbol{\filleddiamond}{\mathbin}{MnSyC}{109}

\usepackage{tikz-3dplot}
\usetikzlibrary{arrows.meta,shapes,calc,positioning,intersections,automata,fit,quotes}
\tdplotsetmaincoords{80}{120}
\tikzset{+ /.tip = {Turned Square}}

\newtheorem{defn}{Definition}[section]
\newtheorem{exmp}[defn]{Example}
\newtheorem{lem}[defn]{Lemma}
\newtheorem{prop}[defn]{Proposition}
\newtheorem{thm}[defn]{Theorem}
\newtheorem{cor}[defn]{Corollary}
\newtheorem{dec}[defn]{Decision problem}

\newcommand\restr[2]{{
		\left.\kern-\nulldelimiterspace 
		#1 
		\vphantom{\big|} 
		\right|_{#2} 
	}}
	
\newcommand{\an}{\mathrm{an}}
\newcommand{\con}{\mathrm{con}}
\newcommand{\disjU}{\mathbin{\dot{\cup}}}
\newcommand{\dtr}{\mathrm{dtr}}
\newcommand{\id}{\mathrm{in}}
\newcommand{\od}{\mathrm{out}}

\newcommand{\starrightarrow}{\ *\!\!\rightarrow}

\newcommand{\musep}[3]{#1 \perp_\mu #2 \mid #3}
\newcommand{\musepG}[4]{#1 \perp_\mu #2 \mid #3\ [#4]}

\DeclareRobustCommand{\allEdges}{%
	\mathrel{%
		\text{%
			\ooalign{$\filleddiamond\!\!\!\!\;-\!\!\!\!\;\filleddiamond$}%
		}%
	}%
}

\title{Weak equivalence of local independence graphs}
\author{Søren Wengel Mogensen}
\date{\vspace*{-.3cm} \normalsize \it Department of Automatic Control, Lund 
University}

\begin{document}
	\maketitle
	
	\begin{abstract}
		Classical graphical modeling of multivariate random vectors uses graphs 
		to encode conditional independence. In graphical modeling of 
		multivariate stochastic processes, graphs may encode so-called local 
		independence 
		analogously. If some coordinate processes of the multivariate 
		stochastic process are unobserved, the 
		local independence graph of the observed coordinate processes is a 
		directed mixed graph (DMG). Two DMGs may 
		encode the same local independences in which case we say that they are 
		Markov 
		equivalent.
		
		Markov equivalence is a central notion in graphical modeling. We show 
		that deciding Markov equivalence of DMGs is coNP-complete, even under a 
		sparsity assumption. As a remedy, we introduce a collection of 
		equivalence relations on DMGs that are all less granular 
		than Markov 
		equivalence and we say that they are weak equivalence relations. This 
		leads to feasible algorithms for naturally occurring computational 
		problems related to weak equivalence of DMGs. The equivalence classes 
		of a weak equivalence 
		relation have attractive properties. In particular, each equivalence 
		class has a greatest element which leads to a concise representation of 
		an equivalence class. Moreover, these equivalence relations define a 
		hierarchy of granularity in the graphical modeling which leads to 
		simple and interpretable connections between equivalence relations 
		corresponding to different levels of granularity.
	\end{abstract}

\section{Introduction}

The distribution of a multivariate random vector, $(X^\alpha)_{\alpha\in V}$, 
induces an \emph{independence model}, $\mathcal{I}$, which is simply the 
collection of triples, 
$(A,B,C)$, such that $X^A$ and $X^B$ are conditionally independent given $X^C$. 
Graphs are often used as convenient representations of such independence models 
\citep{lauritzen1996, maathuis2018}. The graphical theory reflects the fact 
that conditional independence is \emph{symmetric} in $A$ and $B$, i.e., 
$(A,B,C) \in \mathcal{I}$ if and only if $(B,A,C)\in\mathcal{I}$. In graphical 
modeling of 
multivariate stochastic 
processes, it is 
useful to apply a notion of independence that distinguishes between past and 
present and for this purpose several authors have used \emph{local 
independence},
analogously to how conditional independence is used in classical graphical 
modeling. However, local independence is not symmetric in the above sense and 
its graphical representation therefore requires a specialized framework. Local 
independence was first introduced by \cite{schweder1970} in 
composable Markov processes and 
later studied by \cite{aalen1987} in a broader class of stochastic processes. 
\cite{didelez2000,didelez2008} described 
graphical modeling of marked point processes based on local independence and 
\cite{mogensenUAI2018} extended this theory to It\^o processes.

Graphs are said to be \emph{Markov equivalent} if they represent the same 
independences, i.e., if they are indistinguishable when observing only the 
induced independences. Several characterizations of Markov equivalence are 
available in different classes of graphs representing classical conditional 
independence \citep{frydenberg1990, verma1990equi, 
spirtes1992equivalence,
andersson1997characterization,andersson1997markov,richardson1997,
andersson2001alternative,zhao2005,zhang2007characterization, ali2009}. 
\cite{Mogensen2020a} used 
\emph{directed mixed graphs} as representations of local independences 
in partially observed stochastic processes and they characterized 
Markov 
equivalence in this class of graphs by proving that each equivalence class 
contains a greatest element. Their equivalence result also 
provided a simple approach to visualizing and understanding an entire 
equivalence class. \cite{Mogensen2020b} characterized Markov equivalence of 
\emph{directed correlation graphs} representing local independence in the 
presence of correlated noise processes. Recent work 
studied local independence testing 
in point
processes \citep{thams2021local} and \cite{christgau2022nonparametric} 
described 
nonparametric tests of local independence. It is 
worth noting that local independence is a continuous-time version of 
discrete-time 
Granger causality which has been used in graphical models of time series
\citep{eichler2007, eichler2010, eichler2012, eichler2013}. The graphical 
theory of directed mixed graphs and the results
in this paper may be applied in both continuous-time and discrete-time 
stochastic processes \citep[supplementary 
material]{Mogensen2020a}. 

In graphs representing classical 
conditional independence, several characterizations of Markov equivalence lead 
to 
polynomial-time 
algorithms for deciding Markov equivalence \citep[e.g.,][]{richardson1997, 
ali2009}. In the local independence framework, \cite{Mogensen2020b} proved 
that
deciding Markov equivalence of two directed correlation graphs is coNP-complete 
which 
means that we should not expect to find a polynomial-time algorithm in this 
case. 
In this paper, we show that deciding Markov equivalence of directed mixed 
graphs is also coNP-complete. We 
further show that assuming sparsity of the directed mixed graphs does not 
generally remedy this. Our results imply that several computational problems 
that occur naturally when using directed mixed graphs are also computationally 
hard. For this reason, Markov 
equivalence in partially observed local independence graphs may not always be a 
practical notion. 
Instead, we introduce a class of \emph{weak equivalence relations} between 
local independence graphs. We characterize the corresponding equivalence 
classes 
and show that they too contain a greatest element. \cite{Mogensen2020a} argued 
that the existence of a greatest element leads to a straightforward Markov 
equivalence theory. We extend this theory to the more general weak 
equivalences studied in this paper. This allows a simple
representation 
of weak equivalence classes. A subset of the weak 
equivalence relations may be understood as creating a \emph{hierarchy} of 
equivalence 
relations in which a parameter, $k$, creates a trade-off between the size of 
the equivalence classes and the computational complexity, 
leading to a graphical theory which is both useful and practical. This 
hierarchy also illustrates interpretable connections between equivalence 
classes across different values of $k$.

The paper is structured in the following way. In Section \ref{sec:graphs}, we 
introduce necessary terminology and notation. We also describe \emph{global 
Markov properties} that connect so-called \emph{$\mu$-separation} in graphs to 
local 
independence and provide justification for using graphs as representations of 
local 
independence. Moreover, we give an example to illustrate the framework 
and 
purpose of the paper. In Section \ref{sec:hard}, we prove that deciding Markov 
equivalence of directed mixed graphs is computationally hard, even under 
sparsity restrictions, and we discuss the implications of this result. In 
Section 
\ref{sec:we}, we introduce the notion of \emph{weak equivalence} of graphs. We 
describe its properties and compare it with Markov equivalence. Section 
\ref{sec:greatElemWEqui} proves that, under a regularity condition, every weak 
equivalence class has a greatest element. Using the main result from the 
previous section, Section \ref{sec:representkWeak} first describes a graph 
which concisely represents an entire equivalence class. It then describes a 
\emph{hierarchy} of certain weak equivalence classes and how they represent 
different levels of granularity in their description of the underlying graphs. 
Section \ref{sec:algo} discusses algorithmic aspects of weak equivalence, and 
in Section \ref{sec:learn} we briefly outline how results from the previous 
sections relate to graphical structure learning. Section \ref{sec:discuss} 
provides a discussion of the results.

\section{Local independence and graphs}
\label{sec:graphs}

The interest in $\mu$-separation arises from its connection to local 
independence as formalized through various \emph{global Markov properties}. We 
start by defining local independence following the exposition in 
\cite{christgau2022nonparametric}. We will give the definition for counting 
processes, though, it can be extended to other classes of stochastic processes 
\citep{didelez2008, mogensenUAI2018, Mogensen2020b}. 

We consider a multivariate counting processes, $N_t = (N_t^1,\ldots,N_t^n)$, on 
a probability space, $(\Omega, \mathbb{F}, P)$, and we assume that $N_t$ is 
observed over some interval $[0,T]$. We let $V$ denote the set 
$\{1,2,\ldots,n\}$. We use $\mathcal{F}_t^D$ to 
denote the right-continuous and complete filtration generated by $N_t^D = 
(N_t^\alpha: \alpha \in D)$. One can think of $\mathcal{F}_t^D$ as consisting 
of the information in the coordinate processes in $D\subseteq V$ up until time 
point $t$. 
For $\beta\in V$ and $C\subseteq V$, we assume that 
$N_t^\beta$ has a $\mathcal{F}_t^C$-intensity, $\lambda_t^{\beta,C}$. The 
stochastic process
$\lambda_t^{\beta,C}$ is $\mathcal{F}_t^C$-predictable and 
$N_t^\beta - \int_{0}^{t} \lambda_s^{\beta,C} \mathrm{d}s$ is a local 
$\mathcal{F}_t^C$-martingale.

\begin{defn}[Local independence]
	Let $\alpha,\beta\in V$ and let $C\subseteq V$. We say that $N_t^\beta$ is 
	\emph{locally independent of $N_t^\alpha$ given $N_t^C$} (or simply, that 
	\emph{$\beta$ 
	is locally independent of $\alpha$ given $C$}) if the local 
	$\mathcal{F}_t^C$-martingale as defined above is also a local 
	$\mathcal{F}_t^{C\cup\{ 
	\alpha\} }$-martingale. For $A,B,C\subseteq V$, we say that $B$ is locally 
	independent of $A$ given $C$ if $\beta$ is locally independent of $\alpha$ 
	given $C$ for all $\alpha\in A$ and $\beta\in B$, and we denote this by 
	$A\not\rightarrow B\mid C$.
	\label{def:li}
\end{defn}

\cite{christgau2022nonparametric} use the term \emph{conditional local 
independence} instead of local independence which highlights the fact that 
Definition \ref{def:li} is analogous to classical conditional independence of 
random variables. Intuitively, when $\beta$ is locally independent of $\alpha$ 
given $C$, observation of the $\alpha$-process over the interval $[0,t]$ does 
not provide additional information other than that contained in 
$\mathcal{F}_{t-}^C$ when trying to predict if there will be an event in 
process $\beta$ in the interval $[t,t+\mathrm{d}t)$.

Local independence was first used by \cite{schweder1970} in composable Markov 
processes and later studied by \cite{aalen1987}. \cite{didelez2000,didelez2008} 
described graphical modeling based on local independence. Other work on local 
independence Markov properties go into more detail 
\citep{didelez2000,didelez2008,mogensenUAI2018,Mogensen2020b}.

\begin{defn}[Local independence graph]
	We consider a multivariate counting process, $N_t = (N_t^1,\ldots,N_t^n)$, 
	as above, $V=\{1,\ldots,n\}$. Its \emph{local independence graph} is the 
	directed graph, $\mathcal{D}$, on nodes $V$ such that
	
	$$
	\alpha\not\rightarrow\beta \text{ in }\mathcal{D} \Leftrightarrow 
	\alpha\not\rightarrow\beta\mid V\setminus \{\alpha\}
	$$
	
	\noindent for $\alpha,\beta\in V$ where $\alpha\not\rightarrow\beta$ 
	indicates the absence of the directed edge from $\alpha$ to $\beta$.
	\label{def:liGraph}
\end{defn}

The statement $\{\alpha\}\not\rightarrow\{\beta\}\mid V\setminus \{\alpha\}$ 
denotes that 
$\beta$ is locally independent of $\alpha$ given $V\setminus \{\alpha\}$, and 
above we have simply written the singletons $\{\alpha\}$ and $\{\beta\}$ as 
$\alpha$ and $\beta$, respectively. The 
implication from left to right in Definition \ref{def:liGraph} is known as the 
\emph{pairwise Markov 
property}. When this property holds, we see that the absence of an edge implies 
a local independence. The \emph{global Markov property} allows one to read off 
more general local independences from a local independence graph using 
$\delta$- or $\mu$-separation 
(Definition \ref{def:muSep}). This is similar to other classes of graphical 
models \citep{maathuis2018}. Several results state conditions 
for the equivalence of pairwise and global Markov properties 
\citep{didelez2008,mogensenUAI2018}. 

Local independence is a continuous-time analogue of Granger causality in 
discrete-time stochastic processes. The results of this paper also applies to 
Granger-causal graphs, see, e.g., the supplementary material of 
\cite{Mogensen2020a} and \cite{eichlerGranger2007}.


\subsection{Alarm network}
\label{ssec:alarm1}

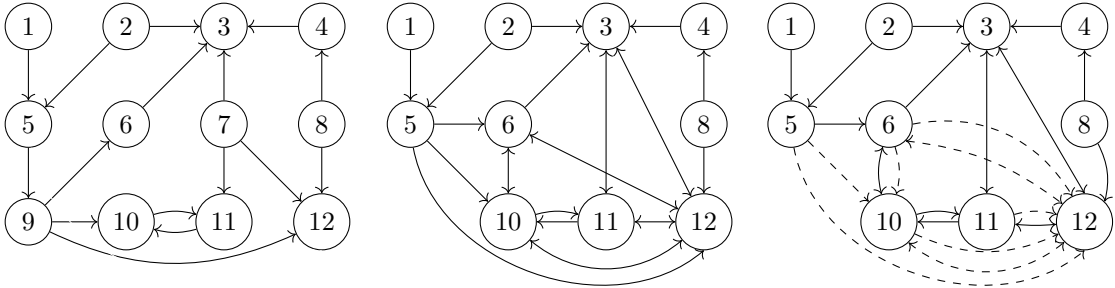
\begin{figure}[h]
	\begin{subfigure}{.32\textwidth}
		\centering
		\begin{tikzpicture}[scale = .65]
		\node[shape=circle,draw=black] (a) at (0,0) {$1$};
		\node[shape=circle,draw=black] (b) at (2,0) 
		{$2$};
		\node[shape=circle,draw=black] (c) at (4,0) {$3$};
		\node[shape=circle,draw=black] (d) at (6,0) {$4$};
		\node[shape=circle,draw=black] (e) at (0,-2) {$5$};
		\node[shape=circle,draw=black] (f) at (2,-2) 
		{$6$};
		\node[shape=circle,draw=black] (g) at (4,-2) {$7$};
		\node[shape=circle,draw=black] (h) at (6,-2) {$8$};
		\node[shape=circle,draw=black] (i) at (0,-4) {$9$};
		\node[shape=circle,draw=black] (j) at (2,-4) 
		{$10$};
		\node[shape=circle,draw=black] (k) at (4,-4) {$11$};
		\node[shape=circle,draw=black] (l) at (6,-4) {$12$};

		\path 
		[->](a) edge [bend left = 0] node {} (e)
		[->](b) edge [bend left = 0] node {} (e)
		[->](b) edge [bend left = 0] node {} (c)
		[->](d) edge [bend left = 0] node {} (c)
		[->](e) edge [bend left = 0] node {} (i)
		[->](f) edge [bend left = 0] node {} (c)
		[->](g) edge [bend left = 0] node {} (c)
		[->](g) edge [bend left = 0] node {} (k)
		[->](g) edge [bend left = 0] node {} (l)
		[->](h) edge [bend left = 0] node {} (d)
		[->](h) edge [bend left = 0] node {} (l)
		[->](i) edge [bend left = 0] node {} (f)
		[->](i) edge [bend left = 0] node {} (j)
		[->](i) edge [bend left = -25] node {} (l)
		[->](j) edge [bend left = 15] node {} (k)
		[->](k) edge [bend left = 15] node {} (j)
		[->](e) edge [bend left = -60, white] node {} (l.south);
		\end{tikzpicture}
	\end{subfigure}\hfill
	\begin{subfigure}{.32\textwidth}
		\centering
		\begin{tikzpicture}[scale = .65]
		\node[shape=circle,draw=black] (a) at (0,0) {$1$};
		\node[shape=circle,draw=black] (b) at (2,0) 
		{$2$};
		\node[shape=circle,draw=black] (c) at (4,0) {$3$};
		\node[shape=circle,draw=black] (d) at (6,0) {$4$};
		\node[shape=circle,draw=black] (e) at (0,-2) {$5$};
		\node[shape=circle,draw=black] (f) at (2,-2) 
		{$6$};
		\node[shape=circle,draw=black] (h) at (6,-2) {$8$};
		\node[shape=circle,draw=black] (j) at (2,-4) 
		{$10$};
		\node[shape=circle,draw=black] (k) at (4,-4) {$11$};
		\node[shape=circle,draw=black] (l) at (6,-4) {$12$};

		\path 
		[->](a) edge [bend left = 0] node {} (e)
		[->](b) edge [bend left = 0] node {} (e)
		[->](b) edge [bend left = 0] node {} (c)
		[->](d) edge [bend left = 0] node {} (c)
		[->](e) edge [bend left = 0] node {} (f)
		[->](e) edge [bend left = 0] node {} (j)
		[->](e) edge [bend left = -60] node {} (l.south)
		[->](f) edge [bend left = 0] node {} (c)
		[->](h) edge [bend left = 0] node {} (d)
		[->](h) edge [bend left = 0] node {} (l)
		[->](j) edge [bend left = 15] node {} (k)
		[->](k) edge [bend left = 0] node {} (j);
		
		\path
		[<->](c) edge [bend left = 0] node {} (k)
		[<->](c) edge [bend left = 0] node {} (l)
		[<->](f) edge [bend left = 0] node {} (j)
		[<->](f) edge [bend left = 0] node {} (l)
		[<->](j) edge [bend left = -40] node {} (l)
		[<->](k) edge [bend left = 0] node {} (l);
		\end{tikzpicture}
	\end{subfigure}\hfill
	\begin{subfigure}{.32\textwidth}
		\centering
		\begin{tikzpicture}[scale = .65]
		\node[shape=circle,draw=black] (a) at (0,0) {$1$};
		\node[shape=circle,draw=black] (b) at (2,0) 
		{$2$};
		\node[shape=circle,draw=black] (c) at (4,0) {$3$};
		\node[shape=circle,draw=black] (d) at (6,0) {$4$};
		\node[shape=circle,draw=black] (e) at (0,-2) {$5$};
		\node[shape=circle,draw=black] (f) at (2,-2) 
		{$6$};
		\node[shape=circle,draw=black] (h) at (6,-2) {$8$};
		\node[shape=circle,draw=black] (j) at (2,-4) 
		{$10$};
		\node[shape=circle,draw=black] (k) at (4,-4) {$11$};
		\node[shape=circle,draw=black] (l) at (6,-4) {$12$};

		\path 
		[->](a) edge [bend left = 0] node {} (e)
		[->](b) edge [bend left = 0] node {} (e)
		[->](b) edge [bend left = 0] node {} (c)
		[->](d) edge [bend left = 0] node {} (c)
		[->](e) edge [bend left = 0] node {} (f)
		[->](e) edge [bend left = 0, dashed] node {} (j)
		[->](e) edge [bend left = -60, dashed] node {} (l.south)
		[->](f) edge [bend left = 0] node {} (c)
		[->](h) edge [bend left = 0] node {} (d)
		[->](h) edge [bend left = 20] node {} (l.north east)
		[->](j) edge [bend left = 15] node {} (k)
		[->](k) edge [bend left = 0] node {} (j)
		[->](f.east) edge [bend left = 35, dashed] node {} (l)
		[->](f) edge [bend left = 15, dashed] node {} (j)
		[->](j) edge [bend left = -25, dashed] node {} (l)
		[->](k) edge [bend left = 15, dashed] node {} (l);
		
		\path
		[<->](c) edge [bend left = 0] node {} (k)
		[<->](c) edge [bend left = 0] node {} (l.north)
		[<->](f) edge [bend left = -15] node {} (j)
		[<->](f.south east) edge [bend left = 15, dashed] node {} (l)
		[<->](j) edge [bend left = -40, dashed] node {} (l.south west)
		[<->](k) edge [bend left = -5] node {} (l);
		\end{tikzpicture}
	\end{subfigure}
	\vspace*{-.7cm}
	\caption{Graphs from the example in
		Subsection \ref{ssec:alarm1}. Loops are omitted from the 
		visualization. Left: underlying local independence graph 
		(directed graph, $\mathcal{D}$,) representing the alarm 
		network. Middle: latent projection, $\mathcal{G}$, of the 
		graph $\mathcal{D}$ when processes $7$ and $9$ are 
		unobserved. The graph $\mathcal{G}$ is a \emph{directed mixed 
			graph} and it represents the partially observed alarm network. 
		The bidirected edges, e.g., $6\leftrightarrow 12$, 
		represent correlation mediated by unobserved nodes, e.g., 
		$6\leftarrow 9 \rightarrow 12$. Unobserved directed paths 
		may create new edges, e.g., $5\rightarrow 6$ in 
		$\mathcal{G}$ corresponds to $5\rightarrow 9 \rightarrow 6$ 
		in $\mathcal{D}$. Right: the directed mixed equivalence 
		graph of the Markov equivalence class containing $\mathcal{G}$.}
	\label{fig:alarm1}
\end{figure}

We describe an example application based on modeling how alarms propagate 
through a complex industrial system. Example data is in Figure 
\ref{fig:localIndep}. In this industrial system, a number of 
\emph{process variables} 
(e.g., temperatures and pressures) are measured repeatedly. Each process 
variable corresponds to an \emph{alarm process}, and if a measured process is 
outside the normal 
range of operations an event occurs in the corresponding alarm process. The 
stochastic system is described by a 
$12$-dimensional counting process, $N_t^V$,

\begin{align*}
	V = 
	\{\mathrm{A1},\mathrm{A2},\mathrm{A3},\mathrm{A4},\mathrm{A5},\mathrm{A6},
	\mathrm{A7},\mathrm{A8},\mathrm{A9},\mathrm{A10},\mathrm{H},\mathrm{E}\},
\end{align*} 

\noindent observed over the interval $[0,1]$. The 
coordinate processes in $V\setminus \{ E\}$ are alarm 
	processes. Process $\mathrm{E}$ represents exogenous 
events 
that feed into the system, e.g., changes in operating conditions, and this 
process is unobserved. Process $\mathrm{H}$ is an alarm process, but 
unavailable for 
some reason, and the observed processes are those in $V\setminus \{E,H\}$. 	
We assume that $\mathcal{D}$ is a local independence graph in the sense 
of Definition \ref{def:liGraph}. Under some regularity conditions, this 
implies that the global Markov property is satisfied in this graph
\citep{didelez2008} and therefore $\mu$-separation (Definition \ref{def:muSep}) 
in the graph implies local independence.

The graph $\mathcal{G}$ in Figure \ref{fig:alarm1} (the \emph{latent 
	projection} of 
$\mathcal{D}$, see Section \ref{app:margin}) represents the observable 
local 
independences in the sense that for $A,B,C \subseteq V\setminus \{E,H\}$ it 
holds that $B$ is $\mu$-separated from $A$ given $C$ in $\mathcal{D}$ if 
and only if $B$ is $\mu$-separated from $A$ given $C$ in $\mathcal{G}$. The 
underlying graph of the full system, $\mathcal{D}$, is a \emph{directed graph} 
while the 
latent projection is a \emph{directed mixed graph}. In general, this larger 
class 
of graphs is needed to represent the local independences of partially 
observed multivariate stochastic processes.

Local independence asks the following question. If we are to predict if 
processes $B$ will have an event in the immediate future and we have the 
information in the past of processes $C$ will the information in the past of 
proesses $A$ add anything? This is illustrated visually in Figure 
\ref{fig:localIndep} with $A = \{\mathrm{A1}\}$, $B = \{ \mathrm{A3}\}$, and $C 
= \{\mathrm{A2},\mathrm{A6}\}$. In this specific example, $\{ \mathrm{A3}\}$ is 
$\mu$-separated from $\{ \mathrm{A1}\}$ given $ \{\mathrm{A2},\mathrm{A6}\}$ in 
$\mathcal{G}$ and under the global Markov property this implies that the 
corresponding local independence holds. Therefore, the information in the past 
of 
process $ \{\mathrm{A1}\}$ is superfluous when already accounting for the 
information in the past of processes $ \{\mathrm{A2},\mathrm{A6}\}$.

Several directed mixed graphs may induce the same $\mu$-separations which means 
that they represent the same local independences. In this case, we say that 
they 
are \emph{Markov equivalent}. The graph on the right in Figure \ref{fig:alarm1} 
is the \emph{directed mixed equivalence graph} of $\mathcal{G}$. It represents 
the entire Markov equivalence class by indicating if an edge is in every Markov 
equivalent graph (solid), in no Markov equivalent graph (absent), or in only 
some Markov equivalent graphs (dashed). This is a useful representation, but 
it may not be a practical one for all applications as it leads to 
computationally hard problems. In this paper, we trade away some of the 
expressive power of Markov equivalence to obtain a more feasible notion of 
equivalence and we show that weaker notions of equivalence remain easily 
interpretable.

\begin{figure}
	\includegraphics[width=1\textwidth]{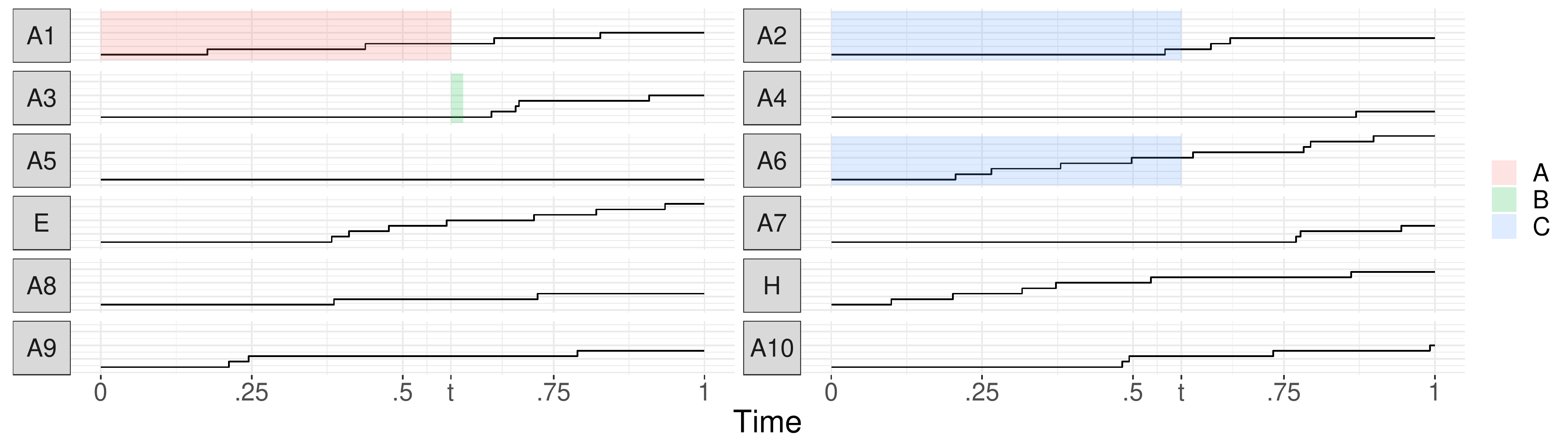}
	\caption{Example data from the alarm network in Subsection 
	\ref{ssec:alarm1}. Each jump of a step 
	function corresponds to an event in the corresponding counting process. 
	Under the global Markov property, the 
	local independence graph $\mathcal{D}$ in Figure \ref{fig:alarm1} implies 
	that $\mathrm{A3}$ ($B$) is locally independent of $\mathrm{A1}$ ($A$) 
	given $\{\mathrm{A2},\mathrm{A6}\}$ ($C$). 
	The meaning of this is illustrated in the above plot. We are trying to 
	predict if there is an event in process $\mathrm{A3}$ in the 
	infinitesimal 
	interval $[t, t+\mathrm{d}t)$ ($B$, green), and we are asking if the 
	information 
	in the past of process $\mathrm{A1}$ ($A$, red) adds anything when 
	accounting 
	for the past of processes $\{\mathrm{A2},\mathrm{A6}\}$ ($C$, blue).}
	\label{fig:localIndep}
\end{figure}


\subsection{Graphs}
\label{ssec:graphs}
A \emph{graph} is a pair $(V,E)$ where $V$ is a finite \emph{node set} and $E$ 
is an \emph{edge set}. The 
edge set $E$ is a disjoint union, $E = E_d \disjU E_b$, where $E_d$ is a set of 
ordered pairs, corresponding to directed edges, $\rightarrow$, and $E_b$ is a 
set of unordered 
pairs, corresponding to bidirected edges, $\leftrightarrow$. We 
use 
$\alpha\leftrightarrow_\mathcal{G} \beta$ to 
denote that there is a bidirected edge between 
$\alpha$ and $\beta$ in the 
graph $\mathcal{G}$, or just $\alpha\leftrightarrow \beta$ when it is clear 
from the context to which graph the statement refers, and we use 
$\alpha\rightarrow_\mathcal{G}$ and $\alpha\rightarrow\beta$ analogously. The 
definition of the node set implies that we 
allow multiple edges between a pair of nodes, however, the edges 
between two nodes $\alpha$ and $\beta$ is always a subset of 
$\{\alpha\rightarrow\beta, \alpha\leftarrow\beta, \alpha\leftrightarrow\beta 
\}$. Moreover, $\alpha\leftrightarrow\beta$ and $\beta\leftrightarrow\alpha$ 
are equivalent while $\alpha\rightarrow\beta$ and $\alpha\leftarrow\beta$ are 
different edges. We emphasize that the edge $\alpha\leftrightarrow\beta$ is not 
shorthand for the two edges $\alpha\rightarrow\beta$ and 
$\alpha\leftrightarrow\beta$, and the meaning of the bidirected edge is 
different from that of the two directed edges. This will be clear from 
subsequent definitions. 

We use $\alpha\sim\beta$ to 
denote a generic edge of either type between $\alpha$ and $\beta$, and we say 
that $\alpha$ and $\beta$ are \emph{adjacent} in $\mathcal{G}$ when there 
exists an edge between them, $\alpha\sim\beta$. When there 
are multiple 
nodes on each side of the edge, 
$\alpha_1,\ldots,\alpha_k \sim 
\beta_1,\ldots,\beta_l$, this means that $\alpha_i \sim \beta_j$ for all 
$i=1,\ldots,k$ and $j=1,\ldots,l$. We separate such statements by semicolons, 
$\alpha_1,\ldots,\alpha_k \sim \beta_1,\ldots,\beta_l$; 
$\gamma_1,\ldots,\gamma_r \sim \delta_1,\ldots,\delta_s$. We use 
$\alpha\starrightarrow\beta$ 
to mean 
that $\alpha\rightarrow\beta$ or $\alpha\leftrightarrow\beta$. We say that 
edges $\alpha\rightarrow\beta$ and $\alpha\leftrightarrow\beta$ have a 
\emph{head} at $\beta$, and that $\alpha\rightarrow\beta$ has a \emph{tail} at 
$\alpha$. If an edge $e$ 
is between $\alpha$ and 
$\beta$ and $\alpha=\beta$, we say that $e$ is a \emph{loop}.

We use $V$ as a generic node set and let $n$ denote the 
cardinality of $V$, $n = \vert V 
\vert$. The graphs described above are \emph{directed mixed graphs} as 
formalized in the next definition.

\begin{defn}[Directed mixed graph (DMG)]
	We say that $\mathcal{G}=(V,E)$ is a \emph{directed mixed 
	graph} if its edge set, $E$, consists of directed and bidirected edges.
\label{def:DMG}
\end{defn}

We say that a DMG is a \emph{directed graph} (DG) if it has no bidirected 
edges. A \emph{walk} between $\gamma_1$ and $\gamma_{l+1}$ is an alternating 
sequence 
of nodes, $\gamma_1,\ldots,\gamma_{l+1}$ and edges $\sim_1,\ldots,\sim_{l}$

$$
\gamma_1 \sim_1 \gamma_2 \sim_2 \ldots \sim_{l} \gamma_{l+1}
$$

\noindent such that for each $i = 1,\ldots,l$, $\sim_i$ is between $\gamma_i$ 
and $\gamma_{i+1}$. Let $e_i$ denote the edge $\sim_i$ above. We will sometimes 
write a walk as $(\gamma_1,e_1,\gamma_2,\ldots,e_l,\gamma_{l+1})$. A walk also 
specifies an \emph{orientation} for each edge 
as one can otherwise not distingush between $\alpha \leftarrow \alpha$ and 
$\alpha\rightarrow\alpha$. We say that $\gamma_i$, $1<i<l+1$, is a 
\emph{collider} if $\sim_{i-1}$ and $\sim_{i}$ both have head at $\gamma_i$. 
Otherwise, we say that it is a 
\emph{noncollider}. A node may be repeated on a walk, $\gamma_i = \gamma_j$, 
$i\neq j$, and may therefore occur both as a collider and as a noncollider on 
the same walk. Thus, the property of being a collider/noncollider pertains to 
the specific \emph{instance} of the node on the walk. We say 
that $\gamma_1$ and $\gamma_{l+1}$ are \emph{endpoints} of the walk. Note that 
endpoints of a walk are neither colliders nor noncolliders. We say that a walk 
is \emph{nontrivial} if it has at least one edge. A walk on 
which no node is repeated is a \emph{path}.

Let $\mathcal{G} = (V,E)$. When $e$ is an edge we use $\mathcal{G} + e$ to 
denote the graph $(V, E\cup \{e \})$, and we use $\mathcal{G} - e$ to denote 
the graph $(V,E\setminus \{ e\})$. We say that $\mathcal{G}$ is 
\emph{complete} if it contains 
$\alpha\rightarrow\beta$; $\alpha\leftarrow\beta$, and 
$\alpha\leftrightarrow\beta$ for all $\alpha,\beta\in V$, and we say that 
$\mathcal{G}$ is \emph{empty} if $E=\emptyset$. We say that a walk 
between $\alpha$ and $\beta$ is \emph{directed} from $\alpha$ to $\beta$ if 
every edge on the walk is directed and points towards (the last) $\beta$, 
$\alpha\rightarrow\ldots\rightarrow\beta$. We say that $\alpha$ 
is an \emph{ancestor} of $\beta$ in $\mathcal{G}$ if there exists a directed 
walk from $\alpha$ to $\beta$, 
and we allow this walk to be trivial (no edges) meaning that a node is always 
an 
ancestor of itself. We define $\an_\mathcal{G}(\alpha)$, or simply 
$\an(\alpha)$, to be the set of 
ancestors of $\alpha$, and for $C\subseteq V$ we define $\an_\mathcal{G}(C) = 
\cup_{\alpha\in C} \an_\mathcal{G}(\alpha)$. Note that $C\subseteq 
\an_\mathcal{G}(C)$.

\begin{defn}[$\mu$-connecting walk]
	We say that a nontrivial walk in a DMG, $\mathcal{G}$,
	
	$$
	\alpha \sim_1 \gamma_1 \sim_2 \ldots \sim_l 
	\beta
	$$
	
	\noindent is \emph{$\mu$-connecting} from $\alpha$ to $\beta$ given $C$ if 
	$\alpha\notin C$, the edge $\sim_l$ has a head at $\beta$, every collider 
	is in $\an(C)$ and no noncollider is in $C$. 
\end{defn}

The $\mu$-connecting walks are used in the definition of
\emph{$\mu$-separation} below which will help us 
connect DMGs to local independence. \cite{mogensenUAI2018} and 
\cite{Mogensen2020a} defined $\mu$-separation as an 
extension to 
\emph{$\delta$-separation} \citep{didelez2000,didelez2008}. One can think of 
$\delta$- 
and 
$\mu$-separation as analogous to $d$- and $m$-separation in 
DAG-based graphical models \citep{pearl2009,richardson2002,richardson2003}.

\begin{defn}[$\mu$-separation]
	Let $\mathcal{G} = (V,E)$ and let $A,B,C\subseteq V$. We say that $B$ is 
	\emph{$\mu$-separated from $A$ given $C$ in $\mathcal{G}$} if there is no 
	$\mu$-connecting walk from any $\alpha\in A$ to any $\beta\in B$ given $C$. 
	We write this as $\musepG{A}{B}{C}{\mathcal{G}}$, or simply 
	$\musep{A}{B}{C}$. We say that $C$ is a \emph{conditioning set}.
	\label{def:muSep}
\end{defn}

By definition, $B$ is $\mu$-separated from $A$ given $C$ if 
$A\subseteq C$. One should also note that $\mu$-separation is not symmetric in 
$A$ and $B$ in that $\musepG{A}{B}{C}{\mathcal{G}}$ does not 
imply $\musepG{A}{B}{C}{\mathcal{G}}$, and neither is local 
independence. This lack of 
symmetry sets the graphical modeling of local independence apart from the 
classical graphical modeling of
conditional independence \citep{lauritzen1996}. In contrast to $m$-separation, 
$\mu$-separation cannot be characterized using 
only paths \citep{Mogensen2020a}. It is, however, possible to obtain a 
characterization using only \emph{routes} which are a finite subset of all 
possible walks (see Definition \ref{def:route} in 
Appendix 
\ref{app:proofs} or \cite{Mogensen2020a}). The next example illustrates the 
concept 
of $\mu$-connecting walks and $\mu$-separation in a DMG.

\begin{exmp}
	We consider the DMG, $\mathcal{G}$, in Figure \ref{fig:DMGindep}. The walk 
	$1\leftrightarrow 2\rightarrow 3$ is 
	$\mu$-connecting from $1$ to $3$ given $\emptyset$. It is not 
	$\mu$-connecting from $1$ to $3$ given $\{ 2\}$ as $2$ is a noncollider. On 
	the walk $1 \leftrightarrow 2 \leftarrow 2 \rightarrow 3$ the node $2$ is a 
	collider in its first instance and a noncollider in its second. The walk $3 
	\rightarrow 2 \leftrightarrow 1$ is $\mu$-connecting from $3$ to $1$ given 
	$\{2\}$, however, the \emph{reverse} walk, $1\leftrightarrow 2 
	\leftarrow 3$ is not $\mu$-connecting from $1$ to $3$ given $\{2\}$.
	
	We see that $3$ is $\mu$-separated from $1$ given $\{2,3\}$ in 
	$\mathcal{G}$. On the other 
	hand, $3$ is not $\mu$-separated from $1$ given $\{2\}$ as the walk $1 
	\leftrightarrow 2 \leftarrow 3 \rightarrow 3$ is $\mu$-connecting.
	\label{exmp:DMGsep}
	
	\begin{figure}[h]
		\centering
		\begin{tikzpicture}
		\node[shape=circle,draw=black] (a) at (0,0) {$1$};
		\node[shape=circle,draw=black] (b) at (2,0) 
		{$2$};
		\node[shape=circle,draw=black] (c) at (4,0) {$3$};
		
		\path 
		[<->](a) edge [bend left = 0] node {} (b)
		[->](b) edge [bend left = 20] node {} (c)
		[->](c) edge [bend left = 20] node {} (b)
		[->](a) edge [loop above] node {} (a)
		[<->](b) edge [loop above] node {} (b)
		[->](c) edge [loop above] node {} (c);
		\end{tikzpicture}
		\caption{The graph $\mathcal{G}$ in Examples \ref{exmp:DMGsep} and
			\ref{exmp:DMGindep}.}
		\label{fig:DMGindep}
	\end{figure}
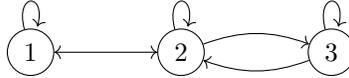
\end{exmp}

\subsection{Independence models and Markov equivalence}
\label{ssec:indep}

For a fixed stochastic process, $X_t = (X_t^1,\ldots,X_t^n)^T$, and a DMG, 
$\mathcal{G}=(V,E)$, both local independence and $\mu$-separation can be 
thought as 
\emph{ternary 
relations} on a finite set $\mathbb{P}(V) \times \mathbb{P}(V)\times 
\mathbb{P}(V)$ 
where $V = \{1,2,\ldots,n\}$ and $\mathbb{P}(\cdot)$ denotes power set. We use 
$\mathcal{P}$ to denote $\mathbb{P}(V) \times \mathbb{P}(V)\times 
\mathbb{P}(V) = \{(A,B,C): A,B,C \subseteq V \}$ and we define an abstract 
independence model, $\mathcal{I}$, to 
be a subset of $\mathcal{P}$. Thus, $\mathcal{I}$ is a collection of 
triples $(A,B,C)$ such that $A,B,C\subseteq V$. We say that $\mathcal{I}$ is an 
\emph{independence model over} $V$. When $A,B$, or $C$ are 
singletons, we will often omit the set notation and write, e.g., 
$(\alpha,\beta,C)$ instead of $(\{\alpha\},\{\beta\},C)$.

We use $\mathcal{I}(\mathcal{G})$ to denote the independence model induced by 
$\mathcal{G}$, that is, the set of 
$\mu$-separations that are true in $\mathcal{G}$, $\mathcal{I}(\mathcal{G}) = 
\{(A,B,C) \in \mathcal{P}: \musepG{A}{B}{C}{\mathcal{G}} \}$. Similarly, an 
independence model can be defined as the set of local independences that hold 
in the distribution of a multivariate stochastic process. We say that an 
independence model, $\mathcal{I}$, is \emph{graphical}, if there 
exist a DMG, $\mathcal{G}$, such that $\mathcal{I} = \mathcal{I}(\mathcal{G})$.

\begin{defn}[Markov equivalence]
	Let $\mathcal{G}_1 = (V,E_1)$ and $\mathcal{G}_2 = (V,E_2)$ be DMGs. We say 
	that $\mathcal{G}_1$ and $\mathcal{G}_2$ are \emph{Markov equivalent} if 
	for all $A,B,C\subseteq V$ it holds that $B$ is $\mu$-separated from $A$ 
	given $C$ in $\mathcal{G}_1$ if and only if $B$ is $\mu$-separated from $A$ 
	given $C$ in $\mathcal{G}_1$. Equivalently, $\mathcal{G}_1$ and 
	$\mathcal{G}_2$ are Markov equivalent if $\mathcal{I}(\mathcal{G}_1) = 
	\mathcal{I}(\mathcal{G}_2)$. We use $[\mathcal{G}_1]$ to denote the Markov 
	equivalence class of $\mathcal{G}_1$.
	\end{defn}

\begin{exmp}	
	We return to the graph, $\mathcal{G}$, in Figure \ref{fig:DMGindep}. 
	By definition, its independence model, $\mathcal{I}(\mathcal{G})$, consists 
	of all triples 
	$(A,B,C)$ such that $B$ is $\mu$-separated from $A$ given $C$ in 
	$\mathcal{G}$. It is enough to consider $(A,B,C)$ such that $A$ and $B$  
	are 
	singletons and $A\not\subseteq C$ as these characterize 
	$\mathcal{I}(\mathcal{G})$ (Proposition 
	\ref{prop:singletonGraphIndep}). We see that $3$ is 
	$\mu$-separated from $1$ given $\{2,3\}$, and this is the only 
	$\mu$-separation of this type in the graph.
	
	\label{exmp:DMGindep}
\end{exmp}

\subsubsection{Extremal elements of sets of DMGs}

Let $\mathbb{G} = \{\mathcal{G}_1 = (V,E_1),\ldots,\mathcal{G}_l=(V,E_l) \}$ be 
a set of DMGs on a common node set, $V$. If $E_i\subseteq E_j$, we write 
$\mathcal{G}_i\subseteq \mathcal{G}_j$, and we say that $\mathcal{G}_i$ is a 
\emph{subgraph} of $\mathcal{G}_j$, and that $\mathcal{G}_j$ is a 
\emph{supergraph} of $\mathcal{G}_i$. We write $\mathcal{G}_i\subsetneq 
\mathcal{G}_j$ when $E_i\subseteq E_j$ and $E_i\neq E_j$. The following 
definitions are common 
set-theoretic notions when considering the set $\mathbb{G}$ with the partial 
order, $\subseteq$.

\begin{defn}[Maximal element, DMG]
	We say that $\mathcal{G} \in \mathbb{G}$ is a \emph{maximal element} of 
	$\mathbb{G}$ if there is no $\bar{\mathcal{G}} \in \mathbb{G}$, 
	$\bar{\mathcal{G}}\neq {\mathcal{G}}$, such that $\mathcal{G}\subseteq 
	\bar{\mathcal{G}}$.
	\label{def:maximalDMG}
\end{defn}

\begin{defn}[Greatest element, DMG]
	We say that $\mathcal{G} \in \mathbb{G}$ is a \emph{greatest element} of 
	$\mathbb{G}$ if
	$\bar{\mathcal{G}}\subseteq 
{\mathcal{G}}$ for all $\bar{\mathcal{G}} \in \mathbb{G}$.
\label{def:greatDMG}
\end{defn}

When a greatest element exists, it is unique. It is also maximal, and it is the 
only maximal 
element. In this paper, we are mostly concerned with maximal and greatest 
elements, however, we also define \emph{minimal} and \emph{least} elements of 
sets of DMGs. We say that $\mathcal{G} \in \mathbb{G}$ is a \emph{minimal 
element} of $\mathbb{G}$ if there is no $\bar{\mathcal{G}} \in \mathbb{G}$, 
	$\bar{\mathcal{G}}\neq {\mathcal{G}}$, such that 
	$\bar{\mathcal{G}}\subseteq 
	\mathcal{G}$. We say that $\mathcal{G} \in \mathbb{G}$ is a \emph{least 
	element} of 
	$\mathbb{G}$ if
	${\mathcal{G}}\subseteq 
	\bar{\mathcal{G}}$ for all $\bar{\mathcal{G}} \in \mathbb{G}$. 
	The set $\mathbb{G}$ will most often be an equivalence class in our usage 
	of 
	the above 
	terms, and we sometimes simply say that $\mathcal{G}$ is a 
	maximal/minimal/greatest/least element when the 
	equivalence class is understood from the context.

\begin{exmp}
If we consider the set of graphs $\mathbb{G} = \{\mathbf{A}, \mathbf{B}, 
\mathbf{C}, \mathbf{D}\}$ in Figure \ref{fig:greatME}, we see that graph 
$\mathbf{D}$ is the greatest 
element of $\mathbb{G}$ as every graph in $\mathbb{G}$ is a subgraph of 
$\mathbf{D}$, and therefore $\mathbf{D}$ is also the unique maximal element of 
$\mathbb{G}$. The smaller set $\bar{\mathbb{G}} = 
\{\mathbf{A},\mathbf{B},\mathbf{C} \}$ does not have a greatest element and 
graphs $\mathbf{B}$ and $\mathbf{C}$ are maximal elements of $\bar{\mathbb{G}}$.
\label{exmp:greatestMaximal}
\end{exmp}

\subsubsection{Representation of Markov equivalence classes}

We introduce a central result from \cite{Mogensen2020a}. They show 
that every Markov equivalence class has a greatest element. Section 
\ref{sec:greatElemWEqui} extends this theorem to weak equivalence relations.

\begin{thm}[Greatest element of a Markov equivalence class, 
\citep{Mogensen2020a}]
	Let $\mathcal{G}$ be a DMG, and let $[\mathcal{G}]$ be its Markov 
	equivalence class. There exists $\mathcal{N}\in [\mathcal{G}]$ such 
	that for all $\bar{\mathcal{G}} \in [\mathcal{G}]$ the edge set of 
	$\bar{\mathcal{G}}$ is a 
	subset of the edge set of $\mathcal{N}$. 
	\label{thm:greatME}
\end{thm}

The next example illustrates the utility of this theorem.

\begin{exmp}
	Graphs $\mathbf{A}$-$\mathbf{D}$ in Figure \ref{fig:greatME} constitute a 
	Markov 
	equivalence class, $[\mathcal{G}]$ (for simplicity, we assume that 
	all loops are present, 
	and do not 
	consider Markov equivalent graphs obtained by removing loops). Graph 
	$\mathbf{D}$ is the \emph{greatest 
	element}
	of $[\mathcal{G}]$ in the sense that all Markov equivalent graphs are 
	subgraphs of graph $\mathbf{D}$. In other words, if a graph in the 
	Markov equivalence class contains the edge $e$, then $e$ is also in 
	the graph $\mathbf{D}$. This means that we can represent the entire 
	Markov equivalence class using graph \textbf{E}. 
	The edges are the same as in the greatest element. Edges are solid in graph 
	$\mathbf{E}$ if they are in every 
	Markov equivalent graph and they are dashed if they are in some Markov 
	equivalent 
	graphs, but not in others. Absent edges are not in any graph in the Markov 
	equivalence class. Therefore, graph $\mathbf{E}$ represents a summary of 
	the information the Markov equivalence class provides on each edge. 
	Moreover, Theorem \ref{thm:greatME} implies that every Markov equivalence 
	class 
	contains a greatest element, and therefore this is a general approach to 
	representing and understanding Markov equivalence classes 
	\citep{Mogensen2020a}.
	
		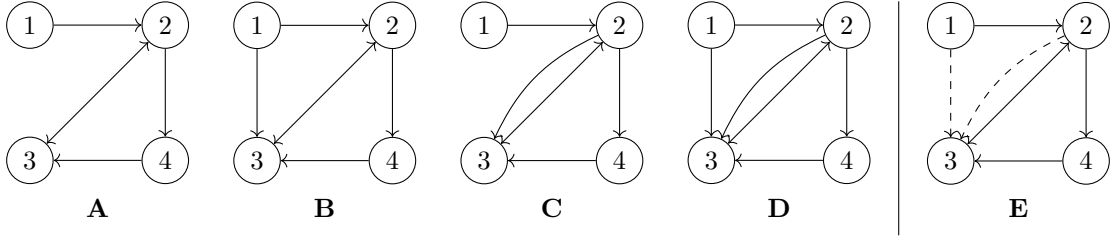
\begin{figure}[h]
			\begin{subfigure}{.19\textwidth}
				\centering
				\begin{tikzpicture}[scale = .9]
				\node[shape=circle,draw=black] (a) at (0,0) {$1$};
				\node[shape=circle,draw=black] (b) at (2,0) 
				{$2$};
				\node[shape=circle,draw=black] (c) at (0,-2) {$3$};
				\node[shape=circle,draw=black] (d) at (2,-2) {$4$};
				
				\node[shape=circle,draw=white] (e) at (1,-2.7) {\textbf{A}};
				
				\path 
				[->](a) edge [bend left = 0] node {} (b)
				[->](b) edge [bend left = 0] node {} (d)
				[->](d) edge [bend left = 0] node {} (c)
				[<->](c) edge [bend left = 0] node {} (b);
				\end{tikzpicture}
			\end{subfigure}\hfill
			\begin{subfigure}{.19\textwidth}
				\centering
				\begin{tikzpicture}[scale = .9]
				\node[shape=circle,draw=black] (a) at (0,0) {$1$};
				\node[shape=circle,draw=black] (b) at (2,0) 
				{$2$};
				\node[shape=circle,draw=black] (c) at (0,-2) {$3$};
				\node[shape=circle,draw=black] (d) at (2,-2) {$4$};
				
				\node[shape=circle,draw=white] (e) at (1,-2.7) {\textbf{B}};
				
				\path 
				[->](a) edge [bend left = 0] node {} (b)
				[->](a) edge [bend left = 0] node {} (c)
				[->](b) edge [bend left = 0] node {} (d)
				[->](d) edge [bend left = 0] node {} (c)
				[<->](c) edge [bend left = 0] node {} (b);
				\end{tikzpicture}
			\end{subfigure}\hfill
			\begin{subfigure}{.19\textwidth}
				\centering
				\begin{tikzpicture}[scale = .9]
				\node[shape=circle,draw=black] (a) at (0,0) {$1$};
				\node[shape=circle,draw=black] (b) at (2,0) 
				{$2$};
				\node[shape=circle,draw=black] (c) at (0,-2) {$3$};
				\node[shape=circle,draw=black] (d) at (2,-2) {$4$};
				
				\node[shape=circle,draw=white] (e) at (1,-2.7) {\textbf{C}};
				
				\path 
				[->](a) edge [bend left = 0] node {} (b)
				[->](b) edge [bend left = -20] node {} (c)
				[->](b) edge [bend left = 0] node {} (d)
				[->](d) edge [bend left = 0] node {} (c)
				[<->](c) edge [bend left = 0] node {} (b);
				\end{tikzpicture}
			\end{subfigure}\hfill
			\begin{subfigure}{.19\textwidth}
				\centering
				\begin{tikzpicture}[scale = .9]
				\node[shape=circle,draw=black] (a) at (0,0) {$1$};
				\node[shape=circle,draw=black] (b) at (2,0) 
				{$2$};
				\node[shape=circle,draw=black] (c) at (0,-2) {$3$};
				\node[shape=circle,draw=black] (d) at (2,-2) {$4$};
				
				\node[shape=circle,draw=white] (e) at (1,-2.7) {\textbf{D}};
				
				\path 
				[->](a) edge [bend left = 0] node {} (b)
				[->](b) edge [bend left = -20] node {} (c)
				[->](a) edge [bend left = 0] node {} (c)
				[->](b) edge [bend left = 0] node {} (d)
				[->](d) edge [bend left = 0] node {} (c)
				[<->](c) edge [bend left = 0] node {} (b);
				\end{tikzpicture}
			\end{subfigure}\hfill\vline\hfill
			\begin{subfigure}{.19\textwidth}
				\centering
				\begin{tikzpicture}[scale = .9]
				\node[shape=circle,draw=black] (a) at (0,0) {$1$};
				\node[shape=circle,draw=black] (b) at (2,0) 
				{$2$};
				\node[shape=circle,draw=black] (c) at (0,-2) {$3$};
				\node[shape=circle,draw=black] (d) at (2,-2) {$4$};
				
				\node[shape=circle,draw=white] (e) at (1,-2.7) {\textbf{E}};
				
				\path 
				[->](a) edge [bend left = 0] node {} (b)
				[->](b) edge [bend left = -20, dashed] node {} (c)
				[->](a) edge [bend left = 0, dashed] node {} (c)
				[->](b) edge [bend left = 0] node {} (d)
				[->](d) edge [bend left = 0] node {} (c)
				[<->](c) edge [bend left = 0] node {} (b);
				\end{tikzpicture}
			\end{subfigure}
			\caption{Graphs from 
				Examples \ref{exmp:greatestMaximal} and \ref{exmp:greatME}. 
				All loops are present in the graphs but omitted from the 
				visualization.}
			\label{fig:greatME}
		\end{figure}

\label{exmp:greatME}
\end{exmp}

\section{Hardness of marginalized local independence graphs}
\label{sec:hard}

In this section, we argue that certain computational problems in relation to 
DMGs and Markov equivalence are hard. For this purpose, we give a very short 
introduction to the concepts from complexity theory that we will need. A 
decision problem is in \textbf{coNP} if 
no-instances have certificates which can be evaluated in polynomial time. For 
instance, if $\mathcal{G}_1$ and $\mathcal{G}_2$ are not Markov equivalent 
(they are a 
\emph{no-instance} when deciding Markov equivalence) a triple $(A,B,C)$ such 
that $B$ is $\mu$-separated from $A$ given $C$ in $\mathcal{G}_1$, but not in 
$\mathcal{G}_2$, may function as a certificate as one can check this specific 
separation in both graphs and conclude that they are not Markov equivalent. A 
decision problem is in \textbf{P} if it can be solved by a deterministic 
Turing machine in polynomial time. A decision problem is \emph{coNP-hard} if it 
is at least as hard as any problem in \textbf{coNP}, and it is 
\emph{coNP-complete} if it is coNP-hard and in \textbf{coNP}. It is generally 
believed that \textbf{P} $\neq$ \textbf{coNP} in which case there are no 
polynomial-time algorithm which can solve a coNP-hard problem. The 
\emph{complement} of a decision problem arises from interchanging 
\emph{yes} and \emph{no}. A decision problem is in \textbf{coNP} if and only if 
its complement is in \textbf{NP}. We now introduce some decision problems 
relating to DMGs.

\begin{dec}[Markov equivalence in DMGs]
	Let $\mathcal{G}_1 = (V,E)_1$ and $\mathcal{G}_2 = (V,E_2)$ be DMGs. Are 
	$\mathcal{G}_1$ and $\mathcal{G}_2$ Markov equivalent?
	\label{dc:ME}
\end{dec}

The development in this paper is partly motivated by the fact that the above 
decision problem is hard (Corollary \ref{cor:MEhard}). We can formulate a 
restricted version of the problem in which the pair of graphs for which to 
decide Markov 
equivalence only differ by a single (bidirected or directed) edge, as 
formalized in Decision problems \ref{dc:add1bidir} (bidirected) and 
\ref{dc:add1dir} (directed). These 
problems are also hard 
and we prove this in Theorem \ref{thm:MEhard}. Corollary \ref{cor:MEhard} 
follows immediately from this theorem.

\begin{thm}
	Let $\mathcal{G}$ be a DMG and let $e$ denote an edge. Deciding Markov 
	equivalence of $\mathcal{G}$ and $\mathcal{G} + e$ is 
	coNP-complete (Decision problems \ref{dc:add1bidir} and \ref{dc:add1dir}).
	\label{thm:MEhard}
\end{thm}

\begin{cor}
	Deciding Markov equivalence of DMGs is 
	coNP-complete (Decision problem \ref{dc:ME}).
	\label{cor:MEhard}
\end{cor}

Decision problem \ref{dc:add1bidir} has been proven to be coNP-complete (PhD 
thesis, \cite{mogensenThesis2020}) and this was used to obtain the result in 
Corollary \ref{cor:MEhard}. We will 
give a slightly different proof to make the generalization to the proof in the 
sparse setting 
more transparent and to also prove that Decision problem \ref{dc:add1dir} is 
coNP-complete. The graphs $\mathcal{G}$, $\mathcal{G}_1$, and $\mathcal{G}_2$ 
used in 
the proof of Theorem \ref{thm:MEhard} are clearly not sparse, that is, for the 
size of the node set going to 
infinity there are nodes with unbounded connectivity (formal definitions of 
node connectivity are in Subsection \ref{ssec:sparseDMGs} and Section 
\ref{app:nodeConn}). In the 
next section, we will show that the hardness results remain true under certain 
sparsity assumptions. We include the proof 
of the non-sparse result in Theorem \ref{thm:MEhard} to illustrate the 
technique as the more 
general result can be proved using a similar approach, even if some 
additional ideas are needed.

\cite{Mogensen2020b} showed that deciding $\mu$-separation Markov equivalence of
so-called \emph{directed correlation graphs} (cDGs) is coNP-complete, though 
only in the non-sparse case. Their proof of coNP-hardness uses a reduction from 
3DNF tautology as does the proof of Theorem \ref{thm:MEhard}. 
However, their proof is specific to cDGs as it uses a characterization of 
Markov 
equivalence which holds in cDGs, but not in DMGs \citep{Mogensen2020b}. While a 
DMG represents the local independences of a partially observed multivariate 
stochastic process, i.e., some coordinate processes are unobserved, a cDG 
represents a multivariate stochastic process driven by correlated noise. 
\cite{Mogensen2020b} compared DMGs and cDGs further and showed that a Markov 
equivalence 
class of cDGs need not have a greatest element.


\begin{proof}
We consider $n$ Boolean variables, $x_1,\ldots,x_n$, and a Boolean formula, $H$,

$$
(z_1^1 \wedge z_2^1 \wedge z_3^1) \vee (z_1^2 \wedge z_2^2 \wedge z_3^2) \vee 
\ldots \vee (z_1^N \wedge z_2^N \wedge z_3^N)
$$

\noindent such that $z_i^k$ is a \emph{literal} of a variable, that is, either 
$x_l$ (a \emph{positive} literal) or $\neg x_l$ (a \emph{negative} literal). We 
assume $H$ to be in 3DNF form (each conjunction has at most three literals).
$N$ is the number of conjunctions in the formula and $n$ is the number of 
variables. We define $n_j$ to be the 
number of factors in the $j$'th conjunction. Deciding whether $H$ is a 
\emph{tautology} (evaluates to true for all inputs) is known to be 
coNP-complete \cite{Garey1979} and we will use a reduction from this problem to 
show coNP-hardness of Decision problems \ref{dc:add1bidir} and \ref{dc:add1dir}.

We construct three graphs, 
$\mathcal{G} = (V,E)$, $\mathcal{G}_1 = (V,E_1)$, and $\mathcal{G}_2 = (V,E_2)$ 
from $H$ such that $\mathcal{G}_1 = \mathcal{G} + e_b$ and $\mathcal{G}_2 = 
\mathcal{G} + e_d$ where $e_b$ is a bidirected edge and $e_d$ is a directed 
edge. We then show that $\mathcal{G}$ and 
$\mathcal{G}_1$ are Markov equivalent if and only if $H$ is a tautology and 
that $\mathcal{G}_1$ and $\mathcal{G}_2$ are Markov equivalent if and only if 
$H$ is a tautology.

First, we define the set $V^-$.

\begin{align*}
	V^- &= \{\gamma,\bar{\gamma},\delta,\bar{\delta} \} \\ &\cup 
	\{\phi_i^k\}_{i = 
	1,\ldots,n_k, k = 1,\ldots,N} \\ &\cup  \{\bar{\phi}_i^k\}_{i = 
	1,\ldots,n_k, k = 1,\ldots,N} \\ &\cup \{\chi_i, \lambda_i\}_{i = 
	1,\ldots,n}.
\end{align*}

We define the node set $V =\{\alpha,\beta,\varepsilon,\phi\} \cup  V^- \cup 
\{\nu_\beta^\rho, \nu_\varepsilon^\rho \}_{\rho \in V^-}$ and 
$V$ is 
the node set of all three graphs $\mathcal{G}=(V,E)$, $\mathcal{G}_1=(V,E_1)$, 
and 
$\mathcal{G}_2=(V,E_2)$. Note that each literal, $z_i^k$, corresponds to two 
nodes, $\phi_i^k$ and $\bar{\phi}_i^k$.
	
We now define the edge set $E$. We add $\gamma \rightarrow \bar{\gamma}$; 
$\gamma \leftarrow \bar{\gamma}$ ; 
$\delta \rightarrow \bar{\delta}$ ; $\delta \leftarrow \bar{\delta}$. For each 
node $\rho \in V^-$, 
we add 
edges $\rho \rightarrow \nu_\varepsilon^\rho, \nu_\beta^\rho$ and $\rho 
\leftarrow 
\nu_\varepsilon^\rho, \nu_\beta^\rho$. We also add edges $\varepsilon 
\leftrightarrow \nu_\varepsilon^\rho$; $\beta \leftrightarrow \nu_\beta^\rho$. 
We add edges 
$\nu_\varepsilon^\rho \rightarrow \nu_\beta^\rho$ and $\nu_\varepsilon^\rho 
\rightarrow \nu_\beta^\rho$ for each $\rho\in V^-$. We also add all directed 
and bidirected loops, $\rho \sim \rho$, for all $\rho\in V$. We add edges 
$\alpha\leftrightarrow\gamma,\bar{\gamma}$; 
$\varepsilon\leftrightarrow\bar{\delta}$; $\beta\leftrightarrow\delta$, and 
$\varepsilon \rightarrow\beta$; $\varepsilon\leftarrow\beta$ as well as $\phi 
\leftrightarrow \varepsilon,\beta$. For each $k = 
1,\ldots, N$, we add $\gamma \leftrightarrow \phi_1^k \leftrightarrow \ldots 
\leftrightarrow \phi_{n_k}^k \leftrightarrow \delta$ and $\bar{\gamma} 
\leftrightarrow \bar{\phi}_1^k \leftrightarrow \ldots \leftrightarrow 
\bar{\phi}_{n_k}^k \leftrightarrow \bar{\delta}$. We add $\bar{\gamma} 
\leftrightarrow \chi_1,\lambda_1$ and $\bar{\delta}
\leftrightarrow\chi_n,\lambda_n$. For each $i = 1,\ldots,n-1$, we add 
$\chi_i,\lambda_i \leftrightarrow \chi_{i+1},\lambda_{i+1}$. Finally, we add 
for 
each $l=1,\ldots,n$ a directed cycle containing $\chi_l$ as well as every 
$\phi_i^k$ and  $\bar{\phi}_i^k$
corresponding to a positive literal of the variable $x_l$, and we add a 
directed cycle containing $\lambda_l$ as well as every $\phi_i^k$ and 
$\bar{\phi}_i^k$ 
corresponding to a 
negative literal of the variable $x_l$. This defines the edge set $E$, 
$\mathcal{G} = (V,E)$. We obtain $\mathcal{G}_1 = (V,E_1)$ from $\mathcal{G}$ 
by adding the edge $\varepsilon\leftrightarrow\beta$, that is, $E_1 = E \cup 
\{\varepsilon\leftrightarrow \beta \}$. Note that $\rho_1$ is an ancestor of 
$\rho_2$ in $\mathcal{G}$ if and only if $\rho_1$ is an ancestor of $\rho_2$ in 
$\mathcal{G}_1$. We obtain $\mathcal{G}_2 = (V,E_2)$ from $\mathcal{G}$ by 
adding the edge $\phi \rightarrow \varepsilon$, $E_2 = E \cup \{\phi 
\rightarrow \varepsilon \}$.

We will first argue that $\mathcal{G}$ and $\mathcal{G}_1$ are Markov 
equivalent if and only if $H$ is a tautology. Assume first that $H$ is a 
tautology and consider a $\mu$-connecting walk in 
$\mathcal{G}_1$,

$$
\rho_1 \sim \ldots \varepsilon \leftrightarrow \beta \ldots \sim \rho_m
$$

\noindent Using the fact that all loops are included, we can always find a 
$\mu$-connecting walk such that the edge $\varepsilon\leftrightarrow\beta$ 
occurs at most once and we assume that this is the case. We can assume that 
$\rho_1$ only occurs once on the walk. If $\rho_1 \neq \alpha$, there is a 
$\mu$-connecting walk from 
$\rho_1$ to $\beta$ with a head at $\beta$: If $\rho_1 \in V^-$, or  $\rho_1 = 
\nu_\varepsilon^\rho$ for some $\rho\in V^-$, either $\rho_1 \rightarrow 
\nu_\beta^\rho \leftrightarrow \beta$ or 
$\rho_1 \leftarrow\nu_\beta^\rho \leftrightarrow \beta$ is connecting and can 
be composed with the subwalk from $\beta$ to $\rho_m$ to obtain a connecting 
walk in $\mathcal{G}$. If 
$\rho_1 = \varepsilon,\beta,\phi$ or $\rho_1 = \nu_\beta^\rho$ for some $\rho 
\in 
V^-$, then $\rho_1\starrightarrow \beta$ is in $\mathcal{G}$. Assume instead 
that $\rho_1 = \alpha$,

$$
\alpha \sim \ldots \varepsilon \leftrightarrow \beta \ldots \sim \rho_m
$$

\noindent and consider the 
subwalk from $\alpha$ to $\varepsilon$, $\omega_1$. 
If 
there is a noncollider on $\omega_1$, say $\psi$, 
then 
$\psi \notin C$ and $\psi \in \an(C)$. We use this to argue that we can always 
find a walk from $\psi$ to $\beta$ such that when concatenated with the subwalk 
from $\alpha$ to $\psi$ we obtain a $\mu$-connecting walk from $\alpha$ to 
$\beta$. If $\psi \in V^-$, we can find a 
connecting walk from 
$\alpha$ to $\beta$ with a head at $\beta$ by concatenating the subwalk from 
$\alpha$ to $\psi$ with $\psi \rightarrow \nu_\beta^\psi \leftrightarrow \beta$ 
if 
$\nu_\beta^\psi \in C$ and $\psi \leftarrow \nu_\beta^\psi \leftrightarrow 
\beta$ if 
$\nu_\beta^\psi \notin C$.  If 
$\psi=\nu_\varepsilon^\rho$ for some $\rho$, we can concatenate with  $\psi 
\rightarrow \nu_\beta^\rho \leftrightarrow \beta$ or $\psi \leftarrow 
\nu_\beta^\rho \leftrightarrow 
\beta$. If $\psi = \nu_\beta^\rho$ for some $\rho$, we can concatenate with 
$\psi \leftrightarrow\beta$. If $\psi = \varepsilon$, then we can replace 
$\varepsilon \leftrightarrow \beta$ with 
$\varepsilon \rightarrow \beta$ to obtain a connecting walk in $\mathcal{G}$. 
If $\psi = \beta$, we can concatenate with $\psi\rightarrow\beta$. If $\psi= 
\phi$, we can concatenate with $\psi\leftrightarrow\beta$. Finally, 
$\psi=\alpha$ is not possible as $\rho_1 = \alpha$ only occurs once on the 
original walk.

Assume now that $\omega_1$ is 
a collider walk. If it goes through a $\bar{\phi}$-segment, then the 
corresponding 
$\phi$-segment is open (note that $\gamma$ and $\bar{\gamma}$ are in a directed 
cycle and so are $\delta$ and $\bar{\delta}$). If it goes through the 
$\chi$-$\lambda$-segment, then 
for each $l = 1,\ldots,n$ either $\chi_l\in \an(C)$ or $\lambda_l\in \an(C)$. 
Let $x_l = 1$ if $\chi_l \in \an(C)$ and $x_l = 0$ otherwise. The formula $H$  
is a tautology and therefore it evaluates to $1$ under this assignment of truth
values. Thus, there exists $k$ such that $z_i^k = 1$ for $i = 
1,\ldots,n_k$. Assume first that $z_i^k$ is a positive literal corresponding to 
the variable $x_l$. In this case, $x_l = 1$ and $\chi_l\in \an(C)$, and 
therefore $\phi_i^k \in \an(C)$. Assume instead that $z_i^k$ is a negative 
literal corresponding to the variable $x_l$. In this case, $x_l = 0$ and 
$\chi_l\notin \an(C)$ which means that $\lambda_l \in \an(C)$ and 
$\phi_i^k\in\an(C)$.
This means that the walk $\alpha \leftrightarrow \gamma \leftrightarrow 
\phi_1^k \leftrightarrow \ldots \leftrightarrow \phi_{n_k}^k \leftrightarrow 
\delta \leftrightarrow \beta$ is open for some $k = 1,\ldots, N$ and this gives 
us a $\mu$-connecting walk from $\alpha$ to $\rho_m$ in $\mathcal{G}$ also in 
this case.

If instead 

$$
\rho_1 \sim \ldots \beta  \leftrightarrow \varepsilon \ldots \sim \rho_m
$$

\noindent then the same arguments hold.

On the other hand, say that $H$ is not a tautology, and consider an assignment, 
$A$, of truth values such that $H$ evaluates to false. Define the set

$$
C = \an\Bigl(\{\chi_i : x_i = 1 \text{ in } A \} \cup \{\lambda_i : x_i = 0
\text{ in } A \} \cup \{\gamma,\delta,\varepsilon,\beta \}\Bigr).
$$

\noindent In $\mathcal{G}_1$, there is an open, bidirected walk from $\alpha$ 
to $\beta$ through the $\chi$-$\lambda$ segment, and we see that $\beta$ is not 
$\mu$-separated from $\alpha$ given $C$. On the other hand, consider a walk 
between $\alpha$ and $\beta$ in $\mathcal{G}$. The first and last edges on a 
connecting walk from $\alpha$ to $\beta$ given $C$ must be bidirected and as $C 
= \an(C)$, this means that 
the walk must be a collider walk to be 
$\mu$-connecting from $\alpha$ to $\beta$ given $C$, and it must 
go through $\delta$. If $\phi_i^k$ corresponds to a positive literal and it is 
open (i.e., in $\an(C)$) then 
the correspond variable is $1$ in $A$ and $z_i^k = 1$. If it corresponds to a 
negative literal 
and it is open, then the corresponding variable is $0$ in $A$ and $z_i^k = 
1$. This means 
that each $\phi_i^k$ segment must be closed in at least one node as the 
assignment $A$ evaluates to $0$. Therefore, $\beta$ is $\mu$-separated from 
$\alpha$ given $C$ in $\mathcal{G}$, and we conclude that $\mathcal{G}$ and 
$\mathcal{G}_1$ are Markov equivalent if and only if $H$ is a tautology.

We now show that $\mathcal{G}_1$ and $\mathcal{G}_2$ are Markov equivalent. 
Take any 
$\mu$-connecting walk in $\mathcal{G}_1$. Any occurrence of 
$\varepsilon\leftrightarrow\beta$ can be replaced by either $\beta 
\leftrightarrow \phi \rightarrow \varepsilon$ or $\beta \leftrightarrow \phi 
\leftrightarrow \varepsilon$, depending on whether $\phi\in C$. The resulting 
walk is present and connecting in $\mathcal{G}_2$. On the other 
hand, consider a $\mu$-connecting walk from $\rho_1$ to $\rho_m$ given $C$ in 
$\mathcal{G}_2$. We start by 
removing 
all non-endpoint occurrences of $\phi$. Say 

$$
\rho_1 \sim \ldots \sim \rho_i \sim \phi \rightarrow \varepsilon \sim \ldots 
\sim \rho_m.
$$

\noindent If $\rho_i = \beta$, then $\rho_i \leftrightarrow \phi \rightarrow 
\varepsilon$ 
can be replaced by $\rho_i\leftrightarrow \varepsilon$. If $\rho_i = \phi$ or 
if $\rho_i = \varepsilon$, we 
can remove the cycle ($\varepsilon = \rho_m$ we may need to concatenate with 
$\varepsilon\rightarrow\varepsilon$ to obtain a $\mu$-connecting walk after 
removing a cycle). If 
instead

$$
\rho_1 \sim \ldots \sim \rho_i \sim \varepsilon \leftarrow \phi \sim \rho_j 
\sim \ldots 
\sim \rho_m
$$

\noindent we do the same depending on $\rho_j$ (if $\phi= \rho_m$ then we 
concatenate the subwalk from $\rho_1$ to $\varepsilon$ with 
$\varepsilon\leftrightarrow\phi$). This gives us a 
$\mu$-connecting walk in $\mathcal{G}_2$ such that $\phi$ is not a 
non-endpoint 
node. Finally, if $\phi \rightarrow \varepsilon$ is still on the walk $\phi$, 
we must have $\rho_1=\psi$ and this edge can be substituted by 
$\phi\leftrightarrow\varepsilon$. The resulting walk is present in 
$\mathcal{G}_1$. Every collider is different from $\phi$ and this means that it 
is in $\an_{\mathcal{G}_1}(C)$ as well. Therefore, this walk is 
$\mu$-connecting in $\mathcal{G}_1$. It follows that $\mathcal{G}_1$ and 
$\mathcal{G}_2$ are Markov equivalent (regardless of whether $H$ is a 
tautology). Therefore, $H$ is a tautology 
if and only if $\mathcal{G}$ are ${\mathcal{G}}_2$ Markov equivalent.

The reduction from 3DNF tautology to Markov equivalence of $\mathcal{G}$ and 
$\mathcal{G}_1$ (or of $\mathcal{G}$ and $\mathcal{G}_2$) is done in polynomial 
time in the number of conjunctions and it follows that Decision problems 
\ref{dc:add1bidir} and \ref{dc:add1dir} are coNP-hard. Given a triple 
$(A,B,C)$, one can decide $\mu$-separation in polynomial time. If two graphs 
are not Markov equivalent, then there exists a triple $(A,B,C)$ such that 
$\mu$-separation holds in one and not in the other. This is a 
polynomially-sized 
certificate, and this means 
that these problems are in \textbf{coNP}, thus, coNP-complete.
\end{proof}

	\begin{figure}[h]
		\centering
		\begin{tikzpicture}
		\node[shape=circle,draw=black] (a) at (-1,0) {$\alpha$};
		\node[shape=circle,draw=black] (x1) at (2,1) {$\chi_1$};
		\node[shape=circle,draw=black] (y1) at (2,-1) {$\lambda_1$};
		\node[shape=circle,draw=black] (x2) at (4,1) {$\chi_2$};
		\node[shape=circle,draw=black] (y2) at (4,-1) {$\lambda_2$};
		\node[shape=circle,draw=black] (x3) at (6,1) {$\chi_3$};
		\node[shape=circle,draw=black] (y3) at (6,-1) {$\lambda_3$};
		\node[draw=none] (xnm1) at (8,1) {};
		\node[draw=none] (ynm1) at (8,-1) {};
		\node[shape=circle,draw=black] (xn) at (10,1) {$\chi_n$};
		\node[shape=circle,draw=black] (yn) at (10,-1) {$\lambda_n$};
		\node[shape=circle,draw=black] (e) at (13,0) {$\varepsilon$};
		\node[shape=circle,draw=black] (b) at (15,0) {$\beta$};
		\node[shape=circle,draw=black] (f) at (14,-1) {$\phi$};
		\node[shape=circle,draw=black] (c) at (.3,2) {$\gamma$};
		\node[shape=circle,draw=black] (d) at (11.5,2.5) {$\delta$};
		\node[shape=circle,draw=black] (cb) at (.3,-2) {$\bar{\gamma}$};
		\node[shape=circle,draw=black] (db) at (11.5,-2.5) {$\bar{\delta}$};
		
		\path (x3) -- node[auto=false]{\ldots} (xnm1);
		\path (x3) -- node[auto=false, sloped, pos = .5]{\ldots\ \ \ \ \ 
			\ldots} 
		(ynm1);
		\path (y3) -- node[auto=false, sloped, pos = .5]{\ldots\ \ \ \ \ 
			\ldots} 
		(xnm1);
		\path (y3) -- node[auto=false]{\ldots} (ynm1);
		
		\path 
		[<->](cb) edge [bend left = 0] node {} (x1)
		[<->](cb) edge [bend left = 0] node {} (y1)
		[<->](x1) edge [bend left = 0] node {} (x2)
		[<->](y1) edge [bend left = 0] node {} (x2)
		[<->](x1) edge [bend left = 0] node {} (y2)
		[<->](y1) edge [bend left = 0] node {} (y2)
		[<->](x2) edge [bend left = 0] node {} (x3)
		[<->](y2) edge [bend left = 0] node {} (x3)
		[<->](x2) edge [bend left = 0] node {} (y3)
		[<->](y2) edge [bend left = 0] node {} (y3)
		[<->](xnm1) edge [bend left = 0] node {} (xn)
		[<->](ynm1) edge [bend left = 0] node {} (xn)
		[<->](xnm1) edge [bend left = 0] node {} (yn)
		[<->](ynm1) edge [bend left = 0] node {} (yn)
		[<->](xn) edge [bend left = 0] node {} (db)
		[<->](yn) edge [bend left = 0] node {} (db)
		[->](b) edge [bend right = 20] node {} (e)
		[<-](b) edge [bend left = 20] node {} (e)
		[<->](b) edge [bend left = 0] node {} (f)
		[<->](f) edge [bend left = 0] node {} (e);
		
		\node[shape=circle,draw=black] (p1) at (2,3) {$\phi_1^1$};
		\node[shape=circle,draw=black] (p2) at (6,3) {$\phi_2^1$};
		\node[shape=circle,draw=black] (p3) at (10,3) {$\phi_3^1$};
		
		\path 
		[<->](c) edge node  {} (p1)
		[<->] (p1) edge node {} 
		(p2)
		[<->] (p2) edge node {} 
		(p3)
		[<->] (p3) edge node {} 
		(d);
		
		\node[shape=circle,draw=black] (p1b) at (2,-3) 
		{$\bar{\phi}_1^1$};
		\node[shape=circle,draw=black] (p2b) at (6,-3) 
		{$\bar{\phi}_2^1$};
		\node[shape=circle,draw=black] (p3b) at (10,-3) 
		{$\bar{\phi}_3^1$};
		
		\path 
		[<->](cb) edge node  {} (p1b)
		[<->] (p1b) edge node {} 
		(p2b)
		[<->] (p2b) edge node {} 
		(p3b)
		[<->] (p3b) edge node {} 
		(db);
		
		\node[shape=circle,draw=black] (p1bN) at (2,-5) 
		{$\bar{\phi}_1^N$};
		\node[shape=circle,draw=black] (p2bN) at (6,-5) 
		{$\bar{\phi}_2^N$};
		\node[shape=circle,draw=black] (p3bN) at (10,-5) 
		{$\bar{\phi}_3^N$};
		
		\path 
		[<->](cb) edge node  {} (p1bN)
		[<->] (p1bN) edge node {} 
		(p2bN)
		[<->] (p2bN) edge node {} 
		(p3bN)
		[<->] (p3bN) edge node {} 
		(db);
		
		\path (p1b) -- node[auto=false, sloped, pos = .5]{\ldots} (p1bN);
		\path (p2b) -- node[auto=false, sloped, pos = .5]{\ldots} (p2bN);
		\path (p3b) -- node[auto=false, sloped, pos = .5]{\ldots} (p3bN);

		\node[shape=circle,draw=black] (p1N) at (2,5) 
		{${\phi}_1^N$};
		\node[shape=circle,draw=black] (p2N) at (6,5) 
		{${\phi}_2^N$};
		\node[shape=circle,draw=black] (p3N) at (10,5) 
		{${\phi}_3^N$};
		
		\path 
		[<->](c) edge node  {} (p1N)
		[<->] (p1N) edge node {} 
		(p2N)
		[<->] (p2N) edge node {} 
		(p3N)
		[<->] (p3N) edge node {} 
		(d);
		
		\path (p1) -- node[auto=false, sloped, pos = .5]{\ldots} (p1N);
		\path (p2) -- node[auto=false, sloped, pos = .5]{\ldots} (p2N);
		\path (p3) -- node[auto=false, sloped, pos = .5]{\ldots} (p3N);
		
		\path 
		[<->](a) edge node  {} (c)
		[<->] (a) edge node {} 
		(cb)
		[<->] (db) edge node {} 
		(e)
		[<->] (d) edge node {} 
		(b);
		\end{tikzpicture}
		\caption{A subgraph of $\mathcal{G}$ in the proof of Theorem 
			\ref{thm:MEhard}.}
		\label{fig:MEhard}
	\end{figure}
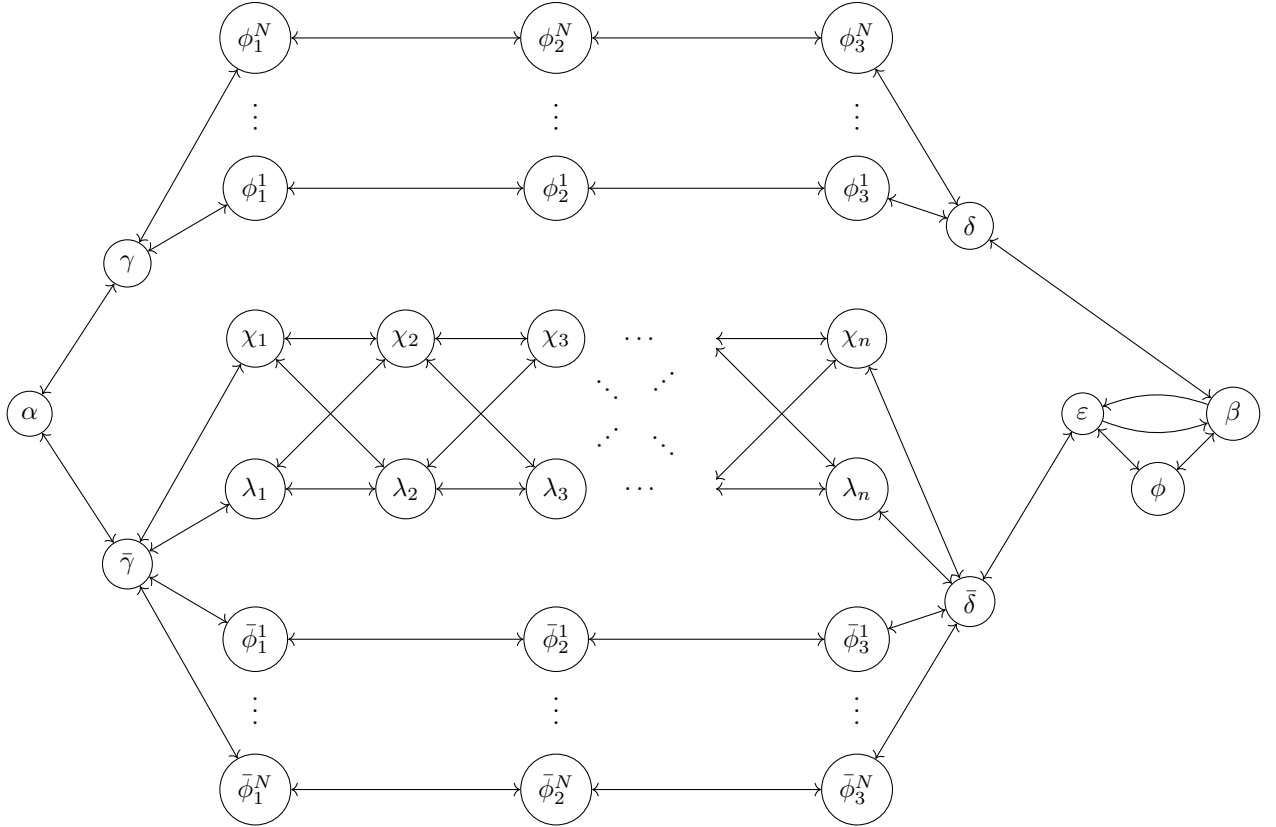
	
		\begin{figure}[h]
			\centering
			\begin{tikzpicture}
			\node[shape=circle,draw=black] (v) at (0,0) {$\rho$};
			\node[shape=circle,draw=black] (ne) at (0,2) 
			{$\nu_\varepsilon^\rho$};
			\node[shape=circle,draw=black] (nb) at (-2,0) {$\nu_\beta^\rho$};
			\node[shape=circle,draw=black] (e) at (-4,0) {$\varepsilon$};
			\node[shape=circle,draw=black] (b) at (0,4) {$\beta$};
			
			\path 
			[->](v) edge [bend left = 20] node {} (nb)
			[->](nb) edge [bend left = 0] node {} (v)
			[->](v) edge [bend left = 0] node {} (ne)
			[->](ne) edge [bend left = 20] node {} (v)
			[->](nb) edge [bend left = 20] node {} (ne)
			[->](ne) edge [bend left = 0] node {} (nb)
			[<->](ne) edge [bend left = 0] node {} (b)
			[<->](e) edge [bend left = 0] node {} (nb);
			\end{tikzpicture}
			\caption{A subgraph of $\mathcal{G}$ in the proof of Theorem 
				\ref{thm:MEhard}.}
			\label{fig:nuMEhard}
		\end{figure}
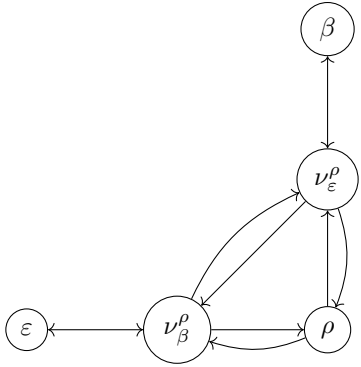
		

Theorem \ref{thm:MEhard} shows that deciding Markov equivalence is not 
computationally 
feasible for large graphs which hurts the practical applicability of 
$\mu$-separation DMGs. We discuss the implications further in Subsection 
\ref{ssec:implic}. We now consider the analogous decision problems in a 
\emph{sparse} setting.

\subsection{Sparse DMGs}
\label{ssec:sparseDMGs}

We may ask if the hardness results still apply if we fix the maximal 
connectivity of each node and let the size of the node set grow. As a 
formalization of this, we first define a notion of node connectivity based on 
\emph{inseparability}. We say that $\beta$ is \emph{inseparable} from
$\alpha$ in 
$\mathcal{I}(\mathcal{G})$ if there is no $C\subseteq V\setminus \{\alpha\}$ 
such that $\beta$ is $\mu$-separated from $\alpha$ given $C$ in $\mathcal{G}$ 
\citep{mogensenUAI2018}. We let 
$\overset{{\scaleto{\rightarrow}{1.5pt}}}{u}(\beta,\mathcal{I}(\mathcal{G}))$ 
denote the 
set of nodes $\alpha$ such that $\beta$ is inseparable from $\alpha$ in 
$\mathcal{G}$, and we let 
$\overset{{\scaleto{\leftarrow}{1.5pt}}}{u}(\beta,\mathcal{I}(\mathcal{G}))$ 
denote the set of nodes $\alpha$ such that $\alpha$ is inseparable from 
$\beta$.

\begin{defn}[Node connectivity in DMG]
	We define $\con_\mathcal{G}^{\scaleto{\rightarrow}{1.8pt}}(\beta)$ as the 
	cardinality of the set
	$\overset{{\scaleto{\rightarrow}{1.5pt}}}{u}(\beta,\mathcal{I}(\mathcal{G}))$
	 and 
	 we define $\con_\mathcal{G}^{\scaleto{\leftarrow}{1.8pt}}(\beta)$ as the 
	 cardinality of the 
	 set 
	 $\overset{{\scaleto{\leftarrow}{1.5pt}}}{u}(\beta,\mathcal{I}(\mathcal{G}))$.
We define $\con_\mathcal{G}(\beta)$ as the maximum of 
	$\con_\mathcal{G}^{\scaleto{\rightarrow}{1.8pt}}(\beta)$ and 
	$\con_\mathcal{G}^{\scaleto{\leftarrow}{1.8pt}}(\beta)$.
\end{defn}

We see that the above definitions are invariant under Markov equivalence, i.e., 
$\con_{\mathcal{G}_1}(\beta) = 
\con_{\mathcal{G}_2}(\beta)$, 
$\con_{\mathcal{G}_1}^{\scaleto{\rightarrow}{1.8pt}}(\beta) = 
\con_{\mathcal{G}_2}^{\scaleto{\rightarrow}{1.8pt}}(\beta)$, and 
$\con_{\mathcal{G}_1}^{\scaleto{\leftarrow}{1.8pt}}(\beta) = 
\con_{\mathcal{G}_2}^{\scaleto{\leftarrow}{1.8pt}}(\beta)$ when $\mathcal{G}_1$ 
and $\mathcal{G}_2$ are 
Markov equivalent. One can define other notions of node 
connectivity in a 
DMG, in particular based on the edges directly, instead of using separability. 
However, a DMG in which every node is adjacent with only a 
small number of nodes may be Markov equivalent with the complete DMG (see 
Figure \ref{fig:completeME}). Even 
in a maximal DMG, the lack of an 
edge between a pair of nodes does not generally imply separability (Appendix 
\ref{app:nodeConn}), and therefore connectivity based on separability 
appears to be a more useful notion of connectivity. Moreover, the 
graphs are intended as representations of stochastic systems, thus
functional sparsity (i.e., sparsity in the implied dependence structure) seems 
more useful 
than representational sparsity (sparsity in node adjacency). Appendix 
\ref{app:nodeConn} provides more details and examples.

		\begin{figure}[h]
			\centering
			\begin{tikzpicture}
			\node[shape=circle,draw=black,inner sep=0pt,minimum size=28pt] (a) 
			at (0,0) {$1$};
			\node[shape=circle,draw=black,inner sep=0pt,minimum size=28pt] (b) 
			at (2.25,0) 
			{$2$};
			\node[shape=circle,draw=black,inner sep=0pt,minimum size=28pt] (c) 
			at (4.5,0) {$3$};
			\node[shape=circle,draw=white,inner sep=0pt,minimum size=28pt] (d) 
			at (6.75,0) {};
			\node[shape=circle,draw=white,inner sep=0pt,minimum size=28pt] (e) 
			at (7.75,0) {\ldots};
			\node[shape=circle,draw=white,inner sep=0pt,minimum size=28pt] (f) 
			at (8.75,0) {};
			\node[shape=circle,draw=black,inner sep=0pt,minimum size=28pt] (g) 
			at (11,0) {$n-1$};
			\node[shape=circle,draw=black,inner sep=0pt,minimum size=28pt] (h) 
			at (13.25,0) {$n$};
			
			\path 
			[->](a) edge [bend left = -20] node {} (b)
			[->](b) edge [bend left = -20] node {} (c)
			[->](c) edge [bend left = -20] node {} (d)
			[->](f) edge [bend left = -20] node {} (g)
			[->](g) edge [bend left = -20] node {} (h);
			\path 
			[<->](a) edge [bend left = 0] node {} (b)
			[<->](b) edge [bend left = 0] node {} (c)
			[<->](c) edge [bend left = 0] node {} (d)
			[<->](f) edge [bend left = 0] node {} (g)
			[<->](g) edge [bend left = 0] node {} (h);
			\path
			[->](a) edge [loop above] node {} (a)
			[->](b) edge [loop above] node {} (b)
			[->](c) edge [loop above] node {} (c)
			[->](g) edge [loop above] node {} (g)
			[->](h) edge [loop above] node {} (h);
			\end{tikzpicture}
			\caption{Loops are omitted from the 
				visualization. This graph is Markov equivalent with the 
				complete DMG on nodes $\{1,2,\ldots,n \}$.}
			\label{fig:completeME}
		\end{figure}
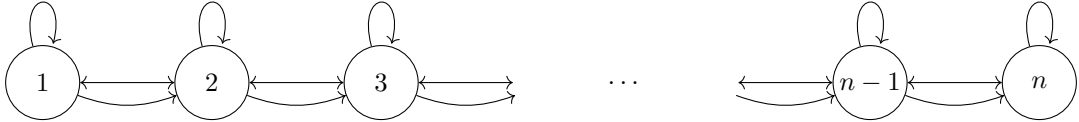

\begin{defn}[$m$-sparsity]
	Let $\mathcal{G}$ be a DMG. The \emph{maximal 
		connectivity} of $\mathcal{G}$ is defined as $\max_{\alpha\in 
		V}(\con_\mathcal{G}(\alpha))$. We say that $\mathcal{G}=(V,E)$ is 
	\emph{$m$-sparse} 
	if $\max_{\alpha\in V}(\con_\mathcal{G}(\alpha)) \leq m$. 
	\label{def:mSparse}
\end{defn}

We now state a sparse version of Decision problem \ref{dc:ME}.

\begin{dec}[Markov equivalence in $m$-sparse DMGs]
	Let $m$ be a nonnegative integer and let $\mathcal{G}_1 = (V,E_1)$ and 
	$\mathcal{G}_2 = (V,E_2)$ be $m$-sparse 
	DMGs. Are $\mathcal{G}_1$ and $\mathcal{G}_2$ Markov equivalent?
	\label{dc:MEsparse}
\end{dec}

The following are sparse versions of Theorem \ref{thm:MEhard} and Corollary 
\ref{cor:MEhard}.

\begin{thm}
	Let $m \geq 16$, let $\mathcal{G} = (V,E)$ be an $m$-sparse graph, and let 
	$e$ denote an edge. 
	Deciding Markov equivalence of $\mathcal{G}$ and $\mathcal{G}+e$ is 
	coNP-complete (Decision problems \ref{dc:add1bidirSparse} and 
	\ref{dc:add1dirSparse}).
		\label{thm:sparseMEhard}
\end{thm}

Theorem \ref{thm:sparseMEhard} is a stronger version of Theorem 
\ref{thm:MEhard} as it shows that the 
problem of deciding Markov equivalence of DMGs remains coNP-complete when 
restricting to sparse DMGs. We discuss the implications in Subsection 
\ref{ssec:implic}.

\begin{cor}
	Let $m \geq 16$. Deciding Markov equivalence of $m$-sparse DMGs is 
	coNP-complete.
	\label{cor:sparseMEhard}
\end{cor}

The value $m = 16$ may not be what we expect from `sparse' graphical models and 
two comments are in order. First, the adjacency sparsity (see Section 
\ref{app:nodeConn}) of the graphs in the 
proof are only $8$, also in the maximal Markov equivalent graphs of the graphs 
used in the proof. Second, the upshot of the corollary is that there exists a 
finite number such that deciding Markov equivalence of $m$-sparse DMGs is 
coNP-complete. This means that fixing the value of $m$ does not generally lead 
to computational problems that scale as polynomials in the size of the graph. 
On the other hand, the so-called $k$-weak equivalences that are introduced in 
this paper 
provide polynomial-time algorithms for each fixed $k$ (Section \ref{sec:algo}). 
Note that results analogous to those of Theorems \ref{thm:MEhard} and 
\ref{thm:sparseMEhard} do not hold for ADMGs with 
$m$-separation. For those, polynomial-time algorithms for Markov equivalence 
are known, without making sparsity 
assumptions \citep{hu2020faster}.

	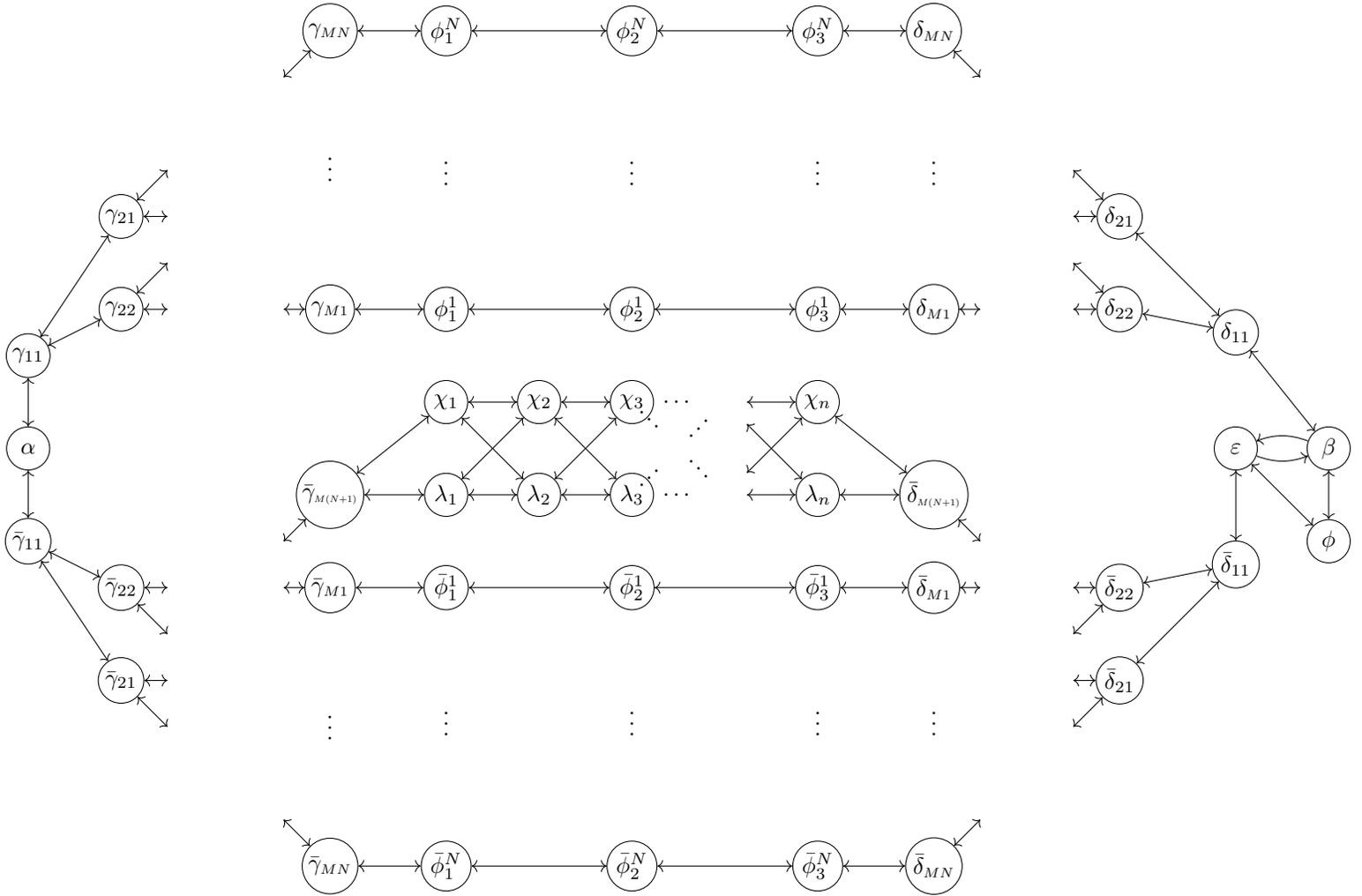
\begin{figure}
		\hspace*{-2.5cm}%
		\centering
		\begin{minipage}{\dimexpr\textwidth+3.5cm}
		\begin{tikzpicture}[scale = .7]
		\tikzset{minimum size= .65cm, inner sep = .05cm}
		\node[shape=circle,draw=black] (a) at (-7,0) {$\alpha$};
		\node[shape=circle,draw=black] (x1) at (2,1) {$\chi_1$};
		\node[shape=circle,draw=black] (y1) at (2,-1) {$\lambda_1$};
		\node[shape=circle,draw=black] (x2) at (4,1) {$\chi_2$};
		\node[shape=circle,draw=black] (y2) at (4,-1) {$\lambda_2$};
		\node[shape=circle,draw=black] (x3) at (6,1) {$\chi_3$};
		\node[shape=circle,draw=black] (y3) at (6,-1) {$\lambda_3$};
		\node[draw=none] (xnm1) at (8,1) {};
		\node[draw=none] (ynm1) at (8,-1) {};
		\node[shape=circle,draw=black] (xn) at (10,1) {$\chi_n$};
		\node[shape=circle,draw=black] (yn) at (10,-1) {$\lambda_n$};
		\node[shape=circle,draw=black] (e) at (19,0) {$\varepsilon$};
		\node[shape=circle,draw=black] (b) at (21,0) {$\beta$};
		\node[shape=circle,draw=black] (f) at (21,-2) {$\phi$};
		\node[shape=circle,draw=black] (c11) at (-7,2) {$\gamma_{11}$};
		\node[shape=circle,draw=black] (d11) at (19,2.5) {$\delta_{11}$};
		\node[shape=circle,draw=black] (cb11) at (-7,-2) {$\bar{\gamma}_{11}$};
		\node[shape=circle,draw=black] (db11) at (19,-2.5) 
		{$\bar{\delta}_{11}$};
		\node[shape=circle,draw=black] (cM1) at (-.5,3) 
		{$\gamma_{\scaleto{M1}{4pt}}$};
		\node[shape=circle,draw=black] (dM1) at (12.5,3) 
		{$\delta_{\scaleto{M1}{4pt}}$};
		\node[shape=circle,draw=black] (cbM1) at (-.5,-3) 
		{$\bar{\gamma}_{\scaleto{M1}{4pt}}$};
		\node[shape=circle,draw=black] (dbM1) at (12.5,-3) 
		{$\bar{\delta}_{\scaleto{M1}{4pt}}$};
		\node[shape=circle,draw=black] (cMN) at (-.5,9) 
		{$\gamma_{\scaleto{MN}{4pt}}$};
		\node[shape=circle,draw=black] (dMN) at (12.5,9) 
		{$\delta_{\scaleto{MN}{4pt}}$};
		\node[shape=circle,draw=black] (cbMN) at (-.5,-9) 
		{$\bar{\gamma}_{\scaleto{MN}{4pt}}$};
		\node[shape=circle,draw=black] (dbMN) at (12.5,-9) 
		{$\bar{\delta}_{\scaleto{MN}{4pt}}$};
		\node[shape=circle,draw=black] (cbMN1) at (-.5,-1) 
		{$\bar{\gamma}_{\scaleto{M(N+1)}{4pt}}$};
		\node[shape=circle,draw=black] (dbMN1) at (12.5,-1) 
		{$\bar{\delta}_{\scaleto{M(N+1)}{4pt}}$};
		\node[shape=circle,draw=black] (cb21) at (-5,-5) {$\bar{\gamma}_{21}$};
		\node[shape=circle,draw=black] (db21) at (16.5,-5) 
		{$\bar{\delta}_{21}$};
		\node[shape=circle,draw=black] (cb22) at (-5,-3) 
		{$\bar{\gamma}_{22}$};
		\node[shape=circle,draw=black] (db22) at (16.5,-3) 
		{$\bar{\delta}_{22}$};
		\node[shape=circle,draw=black] (c21) at (-5,5) {${\gamma}_{21}$};
		\node[shape=circle,draw=black] (d21) at (16.5,5) 
		{${\delta}_{21}$};
		\node[shape=circle,draw=black] (c22) at (-5,3) 
		{${\gamma}_{22}$};
		\node[shape=circle,draw=black] (d22) at (16.5,3) 
		{${\delta}_{22}$};
		
		\path (x3) -- node[auto=false]{\ldots} (xnm1);
		\path (x3) -- node[auto=false, sloped, pos = .5]{\ldots\ \ \ \ \ 
			\ldots} 
		(ynm1);
		\path (y3) -- node[auto=false, sloped, pos = .5]{\ldots\ \ \ \ \ 
			\ldots} 
		(xnm1);
		\path (y3) -- node[auto=false]{\ldots} (ynm1);
		
		\path 
		[<->](cbMN1) edge [bend left = 0] node {} (x1)
		[<->](cbMN1) edge [bend left = 0] node {} (y1)
		[<->](x1) edge [bend left = 0] node {} (x2)
		[<->](y1) edge [bend left = 0] node {} (x2)
		[<->](x1) edge [bend left = 0] node {} (y2)
		[<->](y1) edge [bend left = 0] node {} (y2)
		[<->](x2) edge [bend left = 0] node {} (x3)
		[<->](y2) edge [bend left = 0] node {} (x3)
		[<->](x2) edge [bend left = 0] node {} (y3)
		[<->](y2) edge [bend left = 0] node {} (y3)
		[<->](xnm1) edge [bend left = 0] node {} (xn)
		[<->](ynm1) edge [bend left = 0] node {} (xn)
		[<->](xnm1) edge [bend left = 0] node {} (yn)
		[<->](ynm1) edge [bend left = 0] node {} (yn)
		[<->](xn) edge [bend left = 0] node {} (dbMN1)
		[<->](yn) edge [bend left = 0] node {} (dbMN1)
		[->](b) edge [bend right = 20] node {} (e)
		[<-](b) edge [bend left = 20] node {} (e)
		[<->](f) edge [bend right = 0] node {} (e)
		[<->](b) edge [bend left = 0] node {} (f);
		
		\node[shape=circle,draw=black] (p1) at (2,3) {$\phi_1^1$};
		\node[shape=circle,draw=black] (p2) at (6,3) {$\phi_2^1$};
		\node[shape=circle,draw=black] (p3) at (10,3) {$\phi_3^1$};
		
		\path 
		[<->](cM1) edge node  {} (p1)
		[<->] (p1) edge node {} 
		(p2)
		[<->] (p2) edge node {} 
		(p3)
		[<->] (p3) edge node {} 
		(dM1);
		
		\node[shape=circle,draw=black] (p1b) at (2,-3) 
		{$\bar{\phi}_1^1$};
		\node[shape=circle,draw=black] (p2b) at (6,-3) 
		{$\bar{\phi}_2^1$};
		\node[shape=circle,draw=black] (p3b) at (10,-3) 
		{$\bar{\phi}_3^1$};
		
		\path 
		[<->](cb11) edge node  {} (cb21)
		[<->](cb11) edge node  {} (cb22)
		[<->] (p1b) edge node {} 
		(p2b)
		[<->] (p2b) edge node {} 
		(p3b)
		[<->] (p3b) edge node {} 
		(dbM1)
		[<->] (db11) edge node {} 
		(db21)
		[<->] (db11) edge node {} 
		(db22);
		
		\node[shape=circle,draw=black] (p1bN) at (2,-9) 
		{$\bar{\phi}_1^N$};
		\node[shape=circle,draw=black] (p2bN) at (6,-9) 
		{$\bar{\phi}_2^N$};
		\node[shape=circle,draw=black] (p3bN) at (10,-9) 
		{$\bar{\phi}_3^N$};
		
		\path 
		[<->](cbMN) edge node  {} (p1bN)
		[<->](cbM1) edge node  {} (p1b)
		[<->] (p1bN) edge node {} 
		(p2bN)
		[<->] (p2bN) edge node {} 
		(p3bN)
		[<->] (p3bN) edge node {} 
		(dbMN);
		
		\path (p1b) -- node[auto=false, sloped, pos = .5]{\ldots} (p1bN);
		\path (p2b) -- node[auto=false, sloped, pos = .5]{\ldots} (p2bN);
		\path (p3b) -- node[auto=false, sloped, pos = .5]{\ldots} (p3bN);

		\path (cMN) -- node[auto=false, sloped, pos = .5]{\ldots} (cM1);		
		\path (cbMN) -- node[auto=false, sloped, pos = .5]{\ldots} (cbM1);
		\path (dMN) -- node[auto=false, sloped, pos = .5]{\ldots} (dM1);
		\path (dbMN) -- node[auto=false, sloped, pos = .5]{\ldots} (dbM1);
				
		\node[shape=circle,draw=black] (p1N) at (2,9) 
		{${\phi}_1^N$};
		\node[shape=circle,draw=black] (p2N) at (6,9) 
		{${\phi}_2^N$};
		\node[shape=circle,draw=black] (p3N) at (10,9) 
		{${\phi}_3^N$};
		
		\path 
		[<->](c11) edge node  {} (c21)
		[<->](c11) edge node  {} (c22)
		[<->] (cMN) edge node {} 
		(p1N)
		[<->] (p1N) edge node {} 
		(p2N)
		[<->] (p2N) edge node {} 
		(p3N)
		[<->] (p3N) edge node {} 
		(dMN)
		[<->] (d11) edge node {} 
		(d21)
		[<->] (d11) edge node {} 
		(d22);

\draw[<->] (c21) --++(1,0);
\draw[<->] (c21) --++(1,1);
\draw[<->] (c22) --++(1,0);
\draw[<->] (c22) --++(1,1);
\draw[<->] (cMN) --++(-1,-1);
\draw[<->] (cM1) --++(-1,0);

\draw[<->] (cb21) --++(1,0);
\draw[<->] (cb21) --++(1,-1);
\draw[<->] (cb22) --++(1,0);
\draw[<->] (cb22) --++(1,-1);
\draw[<->] (cbMN) --++(-1,1);
\draw[<->] (cbM1) --++(-1,0);

\draw[<->] (d21) --++(-1,0);
\draw[<->] (d21) --++(-1,1);
\draw[<->] (d22) --++(-1,0);
\draw[<->] (d22) --++(-1,1);
\draw[<->] (dMN) --++(1,-1);
\draw[<->] (dM1) --++(1,0);

\draw[<->] (db21) --++(-1,0);
\draw[<->] (db21) --++(-1,-1);
\draw[<->] (db22) --++(-1,0);
\draw[<->] (db22) --++(-1,-1);
\draw[<->] (dbMN) --++(1,1);
\draw[<->] (dbM1) --++(1,0);

\draw[<->] (cbMN1) --++(-1,-1);
\draw[<->] (dbMN1) --++(1,-1);
		
		\path (p1) -- node[auto=false, sloped, pos = .5]{\ldots} (p1N);
		\path (p2) -- node[auto=false, sloped, pos = .5]{\ldots} (p2N);
		\path (p3) -- node[auto=false, sloped, pos = .5]{\ldots} (p3N);
		
		\path 
		[<->](a) edge node  {} (c11)
		[<->] (a) edge node {} 
		(cb11)
		[<->] (db11) edge node {} 
		(e)
		[<->] (d11) edge node {} 
		(b);
		\end{tikzpicture}
		\caption{A subgraph of $\mathcal{G}$ in the proof of Theorem 
			\ref{thm:sparseMEhard}.}
		\label{fig:sparseMEhard}
\end{minipage}%
\hspace*{-1cm}
	\end{figure}
 

\begin{proof}
	We consider a Boolean formula in 3DNF form as in the proof of Theorem 
	\ref{thm:MEhard} (see that proof for related notation and terminology). We 
	will define 
	three $m$-sparse graphs 
	$\mathcal{G} = (V,E)$, $\mathcal{G}_1=(V,E_1)$, and $\mathcal{G}_2 = 
	(V,E_2)$ and show that $\mathcal{G}$ and $\mathcal{G}_1$ are Markov 
	equivalent if and only if $H$ is a tautology while $\mathcal{G}_1$ and 
	$\mathcal{G}_2$ are always Markov equivalent.

We define 
$M$ to be the smallest integer such that $2^{M-1} \geq N+1$. We first define a 
number of sets that will be subsets of the node set $V$. Note 
that these 
sets are all pairwise disjoint. 

	\begin{align*}
	\Gamma &= \{\gamma_{ij}, i = 1,\ldots,M, j = 1\ldots, 2^{i-1} \} \\
	\bar{\Gamma} &= \{\bar{\gamma}_{ij}, i = 1,\ldots,M, j = 1\ldots, 2^{i-1} 
	\} 
	\\
	\Delta &= \{\delta_{ij}, i = 1,\ldots,M, j = 1\ldots, 2^{i-1} \} \\
	\bar{\Delta} &= \{\bar{\delta}_{ij}, i = 1,\ldots,M, j = 1\ldots, 2^{i-1} 
	\} 
	\\	
	\Phi &= \{\phi_i^j, j = 1,\ldots,N, i=1,\ldots,n_j \}	\\
	\bar{\Phi} &= \{\bar{\phi}_i^j,  j = 1,\ldots,N, i=1,\ldots,n_j \}	\\
	\mathrm{X} &= \{{\chi}_l,   l=1,\ldots,n \} \\
	\Lambda &= \{{\lambda}_l,   l=1,\ldots,n \} 
	\end{align*}
	
	\noindent The node 
	$\chi_l$ corresponds to the Boolean variable $x_l$ and the node $\lambda_l$ 
	corresponds to the negation of $x_l$. Nodes $\phi_i^j$ and $\bar{\phi}_i^j$ 
	both correspond to the literal $z_i^j$ (see also the proof of Theorem 
	\ref{thm:MEhard} for additional explanation). We define
	
	\begin{align*}
		V^- &= \Gamma  \cup \Delta 
		\cup  \Phi   \cup \mathrm{X} \cup 
		\Lambda \\
\bar{V}^- &=   \bar{\Gamma} 
\cup \bar{\Delta}  \cup \bar{\Phi} \\
		\mathrm{N}_\varepsilon &= \{\nu_\varepsilon^\phi, 
				\}_{\phi \in V^-} \\
		\mathrm{N}_\beta	&= \{\nu_\beta^\phi \}_{\phi \in V^-} \\
		\bar{\mathrm{N}}_\varepsilon &= \{\bar{\nu}_\varepsilon^\phi, 
		\}_{\phi \in \bar{V}^-} \\
		\bar{\mathrm{N}}_\beta	&= \{\bar{\nu}_\beta^\phi \}_{\phi \in 
		\bar{V}^-}.
	\end{align*}

	\noindent We now define the node 
	set $V$ as a disjoint union,
	
	$$
	V = \{\alpha,\beta,\varepsilon,\phi\} \cup V^- \cup \bar{V}^- \cup 
	\mathrm{N}_\varepsilon \cup \mathrm{N}_\beta \cup 
	\bar{\mathrm{N}}_\varepsilon \cup \bar{\mathrm{N}}_\beta.
	$$
  
  We add some intuition on the construction of the graph. The 
  $\Gamma$- and $\Delta$-nodes (and their barred versions) are `triangular' in 
  shape and help connect a 
  single node to many more in a sparse manner (see Figure 
  \ref{fig:sparseMEhard}). The $\Phi$- and $\bar{\Phi}$-nodes correspond to 
  literals in the conjunctions of the Boolean formula, $H$. The elements of 
  $\mathrm{X}$ correspond to 
  variables in $H$, and the elements of $\Lambda$ to their 
  negation. The $\nu_\varepsilon$- and 
  $\nu_\beta$-components will help connect every node to $\varepsilon$ and to 
  $\beta$ and are copies of the $V^{-}$ and $\bar{V}^-$ sets in the 
  sense that $\rho \mapsto \nu_\varepsilon^\rho$ is a bijection from $V^-$ to 
  $\mathrm{N}_\varepsilon$,  $\rho \mapsto \nu_\beta^\rho$ is a bijection from 
  $V^-$ to $\mathrm{N}_\beta$,  ${\rho} \mapsto 
  \bar{\nu}_\varepsilon^{{\rho}}$ is a bijection from $\bar{V}^-$ to 
  $\bar{\mathrm{N}}_\varepsilon$, and  ${\rho} \mapsto 
  \bar{\nu}_\beta^{{\rho}}$ is a bijection from $\bar{V}^-$ to 
  $\bar{\mathrm{N}}_\beta$, though the edges are not exact copies as explained 
  below.
  
  We now define the edge set of $\mathcal{G}$. We add 
  bidirected edges $\gamma_{ij} \leftrightarrow \gamma_{(i+1)(2j)}, 
  \gamma_{(i+1)(2j-1)}$ for $i = 1,\ldots,M-1$, and analogously for 
  $\bar{\Gamma}$, 
  $\Delta$, and $\bar{\Delta}$ (see Figure \ref{fig:sparseMEhard}). Moreover, 
  we add 
  $\gamma_{Mj} \leftrightarrow \phi_1^j$; $\bar{\gamma}_{Mj} \leftrightarrow 
  \bar{\phi}_1^j$; $\delta_{Mj} \leftrightarrow \delta_{n_j}^j$; 
  $\bar{\delta}_{Mj} \leftrightarrow \bar{\delta}_{n_j}^j$ for $j \leq N$. We 
  also add $\bar{\gamma}_{M2^{M-1}} 
  \leftrightarrow {\chi}_1,{\lambda}_1$; $\bar{\delta}_{M2^{M-1}} 
  \leftrightarrow 
  \chi_n,\lambda_n$. We add $\alpha 
  \leftrightarrow \gamma_{11}, \bar{\gamma}_{11}$. 
  We also add $\varepsilon \leftrightarrow \bar{\delta}_{11}$ and 
  $\beta \leftrightarrow {\delta}_{11}$. We add 
  $\varepsilon\rightarrow\beta$ and $\beta\rightarrow \varepsilon$ as well as 
  $\phi\leftrightarrow\varepsilon,\beta$. We add for each $j = 1,\ldots, N$, 
  $\phi_i^j \leftrightarrow \phi_{i+1}^j$ 
  and  $\bar{\phi}_i^j \leftrightarrow \bar{\phi}_{i+1}^j$ 
  for $1\leq i \leq n_j-1$.
  
  For $\phi_1,\phi_2\in V^-$ such that $\phi_1\notin \Phi$ or 
  $\phi_2\notin \Phi$, we add $\nu_\varepsilon^{\phi_1} 
  \leftrightarrow 
  \nu_\varepsilon^{\phi_2}$ 
  and $\nu_\beta^{\phi_1} \leftrightarrow \nu_\beta^{\phi_2}$ if 
  and only if $\phi_1\leftrightarrow\phi_2$ was added above. For each $j$, we 
  also add
  ${\nu}_\beta^{\gamma_{Mj}} \leftrightarrow 
  {\nu}_\beta^{\phi_i^j}\leftrightarrow {\nu}_\beta^{\delta_{Mj}}$ and  
  $\bar{\nu}_\varepsilon^{\gamma_{Mj}} \leftrightarrow 
  \bar{\nu}_\varepsilon^{\phi_i^j}\leftrightarrow 
  \bar{\nu}_\varepsilon^{\delta_{Mj}}$ for 
  each $i = 1,\ldots,n_j$. We also add 
  $\nu_\varepsilon^{\delta_{11}} \leftrightarrow \varepsilon$ ; 
  $\nu_\beta^{\delta_{11}} \leftrightarrow \beta$ ; 
  $\bar{\nu}_\varepsilon^{\bar{\delta}_{11}} \leftrightarrow \varepsilon$ and 
  $\bar{\nu}_\beta^{\bar{\delta}_{11}} \leftrightarrow \beta$. Note that 
  $\nu_\varepsilon^{\gamma_{11}}$, $\nu_\beta^{\gamma_{11}}$, 
  $\bar{\nu}_\varepsilon^{\bar{\gamma}_{11}}$, 
  $\bar{\nu}_\beta^{\bar{\gamma}_{11}}$ are 
  not 
  adjacent with $\alpha$. For $\phi_1,\phi_2\in \bar{V}^-$ such that 
  $\phi_1\notin \bar{\Phi}$ or $\phi_2\notin \bar{\Phi}$, we add 
  $\bar{\nu}_\varepsilon^{\phi_1} \leftrightarrow 
  \bar{\nu}_\varepsilon^{\phi_2}$ 
  and $\bar{\nu}_\beta^{\phi_1} \leftrightarrow \bar{\nu}_\beta^{\phi_2}$ if 
  and only if $\phi_1\leftrightarrow\phi_2$ was added above. For each $j$, we 
  also add
  $\bar{\nu}_\beta^{\bar{\gamma}_{Mj}} \leftrightarrow 
  \bar{\nu}_\beta^{\phi_i^j}\leftrightarrow 
  \bar{\nu}_\beta^{\bar{\delta}_{Mj}}$ and  
  $\bar{\nu}_\varepsilon^{\bar{\gamma}_{Mj}} \leftrightarrow 
  \bar{\nu}_\varepsilon^{\phi_i^j}\leftrightarrow 
  \bar{\nu}_\varepsilon^{\bar{\delta}_{Mj}}$ for 
  each $i = 1,\ldots,n_j$.
   
   In this proof, we will say that sets $
   V^-, 
   	\bar{V}^-, 
   	N_\varepsilon, 
   	N_\beta, 
   	\bar{N}_{{\varepsilon}}$, and 
   $\bar{N}_{{\beta}}$  are \emph{line segments}. We define
   
   \begin{alignat*}{3}
   V^i &= \{\gamma_{ij}, 
   \nu_\varepsilon^{\gamma_{ij}}, \nu_\beta^{\gamma_{ij}}, \bar{\gamma}_{ij}, 
   \bar{\nu}_\varepsilon^{\bar{\gamma}_{ij}}, 
   \bar{\nu}_\beta^{\bar{\gamma}_{ij}}, j = 
   1,\ldots,2^{i-1}  
   \}, &&i = -M,\ldots,-1,  \\
   V^i &= \{\delta_{ij}, 
   \nu_\varepsilon^{\delta_{ij}}, \nu_\beta^{\delta_{ij}}, \bar{\delta}_{ij}, 
   \bar{\nu}_\varepsilon^{\bar{\delta}_{ij}}, 
   \bar{\nu}_\beta^{\bar{\delta}_{ij}}, j = 
   1,\ldots,2^{i-1}  
   \}, &&i = 1,\ldots,M, \\ 
   V^0 &= \{\phi_{i}^j, 
   \nu_\varepsilon^{\phi_{i}^j}, \nu_\beta^{\phi_{i}^j}, \bar{\phi}_{i}^j, 
   \bar{\nu}_\varepsilon^{\bar{\phi}_{i}^j}, 
   \bar{\nu}_\beta^{\bar{\phi}_{i}^j}, j = 
   1,\ldots,N,i=1,\ldots,n_j \}\ \cup && \\
    &\qquad \{\chi_i,\lambda_i, \nu_\varepsilon^{\chi_i}, 
   \nu_\varepsilon^{\lambda_i}, 
   \nu_\beta^{\chi_i}, \nu_\beta^{\lambda_i} \},&& \\
   V^{-(M+1)} &= \{\alpha \} && \\
   V^{M+1} &= \{\beta,\varepsilon,\phi \} && \\
   \end{alignat*}
   
   \noindent and we say that $V^i$ is a 
   \emph{vertical segment} for $i = -(M+1),M,\ldots,-1,0,1,\ldots, M,M+1$. 
   `Vertical' refers to the specific visualization of $\mathcal{G}$ used in 
   Figure \ref{fig:sparseMEhard}. The 
   sets, $V_j^i$, defined 
   above are disjoint 
   and $\bigcup_{i = -(M+1)}^{M+1} V^i = V$.
     
      We now add a number of directed edges. For every node $\phi\in V^-$, we 
      add 
      $\phi, \nu_\varepsilon^\phi, \nu_\beta^\phi
      \rightarrow \phi, \nu_\varepsilon^\phi, \nu_\beta^\phi$. For every node 
      $\phi\in \bar{V}^-$, we 
      add 
      $\phi, \bar{\nu}_\varepsilon^\phi, \bar{\nu}_\beta^\phi
      \rightarrow \phi, \bar{\nu}_\varepsilon^\phi, \bar{\nu}_\beta^\phi$. For 
      each $i = \pm 1,\ldots,\pm M$, we connect the nodes in the vertical 
      segment 
      $V^i$ by a directed cycle (any will work). We add directed 
    cycles containing 
    $\chi_k$ 
    and 
    all $\phi_i^j$ and $\bar{\phi}_i^j$ such that $z_i^j$ is a 
    positive literal of the variable $x_k$.  We add directed cycles containing 
    $\lambda_k$ and all $\phi_i^j$ and $\bar{\phi}_i^j$ such 
    that $z_i^j$ is a negative literal of the variable $x_k$.
    
     Finally, we add all directed and bidirected loops. The above defines 
     the edge set $E$ and we let
     $\mathcal{G}=(V,E)$. Note that the nodes in a vertical 
     segment are connected by a directed 
 cyclic walk for $i\neq -(M-1),0,M+1$.  We also define $\mathcal{G}_1 = 
 (V,E_1)$ where $E_1 = E \cup \{\beta\leftrightarrow\varepsilon \}$ and 
 $\mathcal{G}_2 = (V,E_2)$ where $E_2 = E \cup \{\phi\rightarrow\varepsilon 
 \}$. Note that in all three graphs, if $\rho_1\sim_e\rho_2$ and $\rho_1$ and 
 $\rho_2$ are in different 
  vertical segments, $V^{i_1}$ and $V^{i_2}$, respectively, then $e$ is 
  bidirected and $i_1 - i_2 = \pm 1$.
   
  We will first show that $\mathcal{G}$ and $\mathcal{G}_1$ are Markov 
  equivalent if and only if $H$ is a tautology. Assume first that $H$ 
  is a 
  tautology and consider a $\mu$-connecting walk from 
  $\rho_1$ to $\rho_m$ in 
  $\mathcal{G}_1$,
  
  $$
  \rho_1 \sim \ldots \sim \rho_m.
  $$
  
  Every node has a self-loop, so it suffices to consider walks where $e_1$ (the 
  edge $\varepsilon \leftrightarrow \beta$) only occurs once. If it does not 
  occur 
  at all the walk is present in $\mathcal{G}$ as well and connecting 
  (ancestry is the same in $\mathcal{G}$ and $\mathcal{G}_1$). Say

  $$
  \underbrace{\rho_1 \sim \ldots \sim \varepsilon}_{\omega_1} \leftrightarrow 
  \underbrace{\beta \sim \ldots \sim \rho_m}_{\omega_2}.
  $$	
  
If $\rho \in V^i$, then we say that $i$ is the \emph{order} of $\rho$.

\begin{lem}
	Let $\rho\in V$ be of order $j$. If there is 
	an open walk from 
	$\rho$ to $\beta$ given $C$ in $\mathcal{G}$ or in $\mathcal{G}_1$ then 
	the $k$'th vertical segment 
	, $j < k < M+1$, contains at least one node in $C$.
	\label{lem:verticalSegmentC}
\end{lem}

\begin{proof}
	If $j = M,M+1$ this is vacuously true as no vertical segment satisfies the 
	condition, and we can assume that $\rho\neq \varepsilon,\beta,\phi$. Note 
	that 
	this
	walk must necessarily pass through a collider in each vertical segment 
	$V^k$ such 
	that $k > j$ which gives the result. To see this, note that removing any 
	vertical segment such that $k>j$ gives us a disconnected graph with $\rho$ 
	in one component and $\beta$ in the other as a vertical segment, 
	$k$, is only adjacent to vertical segments $k-1$ and $k+1$. When a walk 
	contains a subwalk $\rho_1 \sim \rho_2$ such that $\rho_1$ is in $V^{k-1}$ 
	and 
	$\rho_2$ 
	is in $V^k$, then the connecting edge must be bidirected. If $\rho_2$ is a 
	collider, we must have $\rho_2\in \an_\mathcal{G}(C)$ and $\rho_2$ is only 
	an 
	ancestor of nodes in $V^k$. Otherwise, $\rho_2$ is an ancestor of a 
	collider in $V^k$ and the same argument applies.
\end{proof}

   \begin{lem}
   	Let $\rho\neq \alpha$ be a node in $\mathcal{G}$. If there exists an open 
   	walk from $\rho$ to $\beta$ in $\mathcal{G}_1$ with a head at $\beta$, then 
   	there exists an open walk $\rho \sim \nu_\beta^{\rho} \sim \ldots \sim 
   	\beta$ in $\mathcal{G}$ with a head at $\beta$ such that every nonendpoint 
   	node equals $\nu_\beta^\rho$ for $\rho \in V^-$ or $\bar{\nu}_\beta^\rho$ 
   	for $\rho \in \bar{V}^-$.
   	\label{lem:connInNu}
   \end{lem}
   
   \begin{proof}
   	If $\rho = \beta,\varepsilon,\phi$, this is immediate.
   	Assume instead that $\rho \in V^- \cup \bar{V}^-$. Choose first the edge 
   	$\rho \leftarrow \nu_\beta^\rho$ if 
   	$\nu_\beta^\rho\in C$, and otherwise $\rho \rightarrow \nu_\beta^\rho$. 
   	We concatenate this with the open bidirected path to $\beta$. Such a path 
   	exists as 
   	$\nu_\beta^{\gamma_{Mj}}\leftrightarrow\nu_\beta^{\delta_{Mj}}$ and 
   	$\bar{\nu}_\beta^{\bar{\gamma}_{Mj}}\leftrightarrow\bar{\nu}_\beta^{\bar{\delta}_{Mj}}$.
   	 This is open 
   	since all vertical segments between $\rho$ and $\beta$ must contain at 
   	least one node which is in $C$ by Lemma \ref{lem:verticalSegmentC}.
   	
   	If instead $\rho \in \mathrm{N}_\varepsilon \cup 
   	\bar{\mathrm{N}}_\varepsilon$ we can do as above as $\rho \leftarrow 
   	\nu_\beta^{\rho}$ and  $\rho \rightarrow \nu_\beta^{\rho}$ are in the 
   	graph. If $\rho \in 
   	\mathrm{N}_\beta \cup \bar{\mathrm{N}}_\beta$, then there is an open 
   	bidirected path with a head at $\beta$ between $\rho$ and $\beta$. If 
   	$\rho= \varepsilon$ or $\rho = \beta$ it follows directly. 
   \end{proof}
  
  We split into cases depending on whether $\rho_1 = \alpha$.
  
  \paragraph{$\rho_1 \neq \alpha$:} 
  There is an open walk (given $C$) from $\rho_1$ with a head at $\beta$ (Lemma 
  \ref{lem:connInNu}) that 
  we can 
  concatenate with $\omega_2$ to obtain a connecting walk in $\mathcal{G}$.
  
  If instead
  
  $$
  \rho_1 \sim \ldots \sim \beta \leftrightarrow \varepsilon \sim \ldots \sim 
  \rho_m
  $$
  
  \noindent the same argument holds.
  
  \paragraph{$\rho_1 = \alpha$:} If we have a subwalk between $\alpha$ and 
  $\beta$ with a noncollider, then we 
  can find a connecting path in the following way. Say we have
  
  $$
  \rho_1 \sim \ldots \sim \psi_0 \sim \psi_1 \sim \psi_2 \sim \ldots  
  \varepsilon
  \leftrightarrow \beta \sim \ldots \sim \rho_m
  $$
  
  \noindent such that $\psi_1$ 
  is a noncollider (note that, ignoring $\alpha\rightarrow\alpha$, 
  $\alpha$ only has bidirected edges at it, so $\psi_1\neq \alpha$ if we remove 
  $\alpha$-loops). There 
  is 
  necessarily a tail at $\psi_1$ on one of the adjacent edges, $\psi_1\notin 
  C$, and $\psi_1\in \an(C)$. We concatenate the subwalk from $\alpha$ to 
  $\psi_1$ with the open walk from $\psi_1$ to $\beta$ that has a head at 
  $\beta$. 
  Lemma \ref{lem:connInNu} gives the existence of this walk. This also holds if 
  $\rho_1 = \psi_0$, $\psi_2 = \varepsilon$, or $\psi_1 = \varepsilon$.
  
  On the other hand, if the subwalk between $\alpha$ and $\beta$ has no 
  noncolliders, then either it stays within a line segment or either 
  $\alpha$, $\beta$, or $\varepsilon$ occur on the subwalk as a nonendpoint. We 
  can assume 
  that $\alpha$ is only an endpoint. If $\beta$ occurs as a nonendpoint, then 
  this $\beta$ 
  is a collider and this means that there is an open subwalk from $\alpha$ to 
  $\beta$ with a head at $\beta$ which we can concatenate with $\omega_2$. If 
  $\varepsilon$ is a collider (other than right before the final $\beta$), then 
  we can remove the cycle from $\varepsilon$ 
  to $\varepsilon$ from the walk. In any case, we can find a connecting 
  collider walk in 
  $\mathcal{G}_1$ (no noncolliders) such that
  $\alpha$, $\beta$, and $\varepsilon$ will each occur once. This means that 
  the subwalk 
  only contains nodes from a single line segment. This segment cannot be 
  $N_\varepsilon$, $N_\beta$, $\bar{N}_\varepsilon$, nor $\bar{N}_\beta$ as 
  $\alpha$ is not 
  adjacent with any node in these line segments. If the walk only intersects 
  the 
  $V^-$-line segment, then it must either go through $\Phi$-nodes or the 
  $\mathrm{X}\cup\Lambda$-nodes, not both, as it has no noncolliders (or such a 
  walk can be found). If it 
  does not visit any 
  $\chi$- or $\lambda$-nodes, then there is 
  an open walk in the $\bar{\Gamma}\cup \bar{\Phi} \cup 
  \bar{\Delta}$-segment (the analogous walk through the 
  barred versions). Finally, assume it does not visit any $\Phi$-nodes. As $H$ 
  is a tautology, there is 
  also a conjunction segment in $\bar{\Phi}$ which is open and connecting 
  from $\alpha$ to 
  $\beta$ with a head at $\beta$. If instead the bidirected walk is in 
  $\bar{V}^-$, the result follows, 
  and if $\varepsilon$ and $\beta$ occur in the opposite order on the original 
  $\mu$-connecting walk, we can use similar arguments.
  
  If the formula is not a tautology,
  let $A$ be an assignment of values such that the formula evaluates to false. 
  We then consider the set
  
  $$
  C^- = \{\chi_l,\nu_\varepsilon^{\chi_l},\nu_\beta^{\chi_l}: x_l = 1 \text{ in 
  } A  \} \cup \{\lambda_l,\nu_\varepsilon^{\lambda_l},\nu_\beta^{\lambda_l}: 
  x_l = 0 \text{ 
  	in } A  \} \cup \Gamma \cup \Delta.
  $$
  
  \noindent 
  We also define $C = 
  \mathrm{an}(C^-) \cup \{\beta,\delta\}$. We see immediately that $\beta$ is 
  not $\mu$-separated from $\alpha$ given $C$ in $\mathcal{G}_1$ as the $\chi - 
  \lambda$-segment 
  contains an open path from $\alpha$ to $\varepsilon$ with a head at 
  $\varepsilon$ 
  and 
  furthermore $\varepsilon\leftrightarrow\beta$ is in the graph. On the other 
  hand, consider a potential $\mu$-connecting walk from $\alpha$ to $\beta$ in 
  $\mathcal{G}$. If 
  $\varepsilon$ is on the walk, 
  it can only return to $\alpha$. It cannot go between bidirected components 
  because 
  the directed cycles are either completely contained in $C$ or in its 
  complement. It cannot go through a $\phi$-component because of the choice 
  of $A$, and we conclude that it cannot be $\mu$-connecting. In conclusion, 
  $\mathcal{G}$ and $\mathcal{G}_1$ are Markov equivalent if and only if $H$ is 
  a tautology.
  
  The arguments in the proof of Theorem \ref{thm:MEhard} show that 
  $\mathcal{G}_1$ and $\mathcal{G}_2$ are Markov equivalent. Arguments similar 
  to those in the proof of Theorem \ref{thm:MEhard} furthermore show that 
  Decision problems \ref{dc:add1bidirSparse} and \ref{dc:add1dirSparse} are 
  coNP-complete.
  
  Careful examination of the graphs reveals that all three are $16$-sparse.
\end{proof}

One should note that the graphs in the proof of Theorem \ref{thm:sparseMEhard} 
could also be interpreted as $\delta$-separation graphs \citep{didelez2008}. In 
this case, the result also holds, i.e., determining 
$\delta$-separation Markov equivalence of sparse DMGs is also coNP-complete. To 
see this one 
should simply note that $\mu$-separation Markov equivalence implies 
$\delta$-separation Markov equivalence and that the conditioning set used in 
the proof when $H$ is not a tautology contains $\beta$. The hardness result in 
the 
$\delta$-separation case then follows from the (A.1) 
property of the supplementary material of \cite{Mogensen2020a} and from noting 
that the
latent projection technique can also be used for $\delta$-separation.

\cite{richardson1997} studied DGs under $d$-separation and gave an example of 
`nonlocality' in this setting. The example consisted of a sequence of pairs of 
graphs, $\mathcal{D}_n^1$ and $\mathcal{D}_n^2$, such that $\mathcal{D}_n^1$ 
and $\mathcal{D}_n^2$ are not Markov equivalent, but the only separation on 
which the graphs disagree involves nodes that are arbitrarily far apart (for 
increasing values of $n$). Our 
setting is quite different, however, DMGs under $\mu$-separation do exibit the 
same 
`nonlocality' as seen from the proof of Theorem \ref{thm:sparseMEhard}. Say 
that $H$ is not a tautology, in which 
case $\mathcal{G}$ and $\mathcal{G}_1$ in the proof of Theorem 
\ref{thm:sparseMEhard} are not Markov equivalent. From the proof, it follows 
that the graphs only disagree on triples $(A,B,C)$ such that $\alpha\in A$ and 
$\beta\in B$, and this means that the proof (for 
non-tautological $H$ of increasing size) gives a sequence of pairs of graphs 
that only disagree on 
$\mu$-separation of a pair of nodes, $\alpha$ and $\beta$, that are
arbitrarily far from each other as measured by the shortest path between 
$\alpha$ and 
$\beta$. Note that this also holds in the maximal Markov equivalent graphs of 
$\mathcal{G}$ and $\mathcal{G}_1$, and it is therefore not due to 
non-maximality.

\subsection{Implications of hardness results}
\label{ssec:implic}

The hardness results have several implications that we will outline in this 
section, in particular, we argue that several other computational problems are 
also hard in $\mu$-separation DMGs.

Every Markov equivalence class has a greatest element \citep{Mogensen2020a}, 
and one can decide if two DMGs are Markov equivalent
by computing the greatest Markov equivalent graph for each of them and compare. 
This means that finding such a greatest element is also hard. There are similar 
implications for \emph{oracle} learning algorithms. A \emph{(local 
independence) oracle} is an abstract function 
which a learning algorithm may query and which, when provided with a triple 
$(A,B,C)$, outputs whether the corresponding local independence holds or not. 
The oracle gives the correct answer, but when using real data, the oracle has 
to be replaced by hypothesis tests of local independence, and the purpose of 
the oracle formalism is simply 
to separate the algorithmic aspects from the hypothesis testing. If we assume 
that there exists a constraint-based learning algorithm which can recover a 
unique 
representative of the Markov equivalence class (say the greatest element, or 
some other uniquely defined 
representative) of the true graph from when given access 
to a local independence oracle, then using this algorithm, one can 
decide Markov equivalence by 
querying the $\mu$-separation models of the graphs. This is done by 
testing $\mu$-separation in the graph and each test is done in polynomial time 
\citep{mogensenThesis2020}. If only a polynomial number of 
queries are required we could also solve Markov equivalence in polynomial time 
by comparing the output for two graphs. Again, this means that such a learning 
algorithm would need an exponential number of tests.

\subsubsection{Sparse DMGs}

All of the above holds even if we are willing to assume that all graphs are 
somewhat sparse ($m$-sparse, $m \geq 16$). This means that a restriction to 
sparse graphs will not remedy this. This is also different from DAG-based 
models in the following sense. In partially observed DAGs, we may learn a 
graphical representation of the equivalence class using tests of conditional 
independence. If we fix $m$ such that the node degree is less than $m$, this 
can be done in polynomial time \citep{claassen2013}.

These hardness results motivate the second part of this 
paper. Instead 
of requiring sparsity of the DMGs, we will reinterpret them to obtain a 
weaker type of equivalence. Essentially, the DMGs are too expressive leading to 
the above infeasibility results in connection to their Markov equivalence 
classes. We can avoid this by considering a weaker type of equivalence. This 
leads to a simple and useful theory and to practical graph learning algorithms 
as we will see in subsequent sections.


\section{Weak equivalence}
\label{sec:we}

In this section, we introduce a notion of \emph{weak equivalence} and 
argue that it provides a computationally feasible notion of equivalence of 
DMGs. Under a regularity condition, the associated equivalence classes each 
have a greatest element 
and this leads 
to a simple graphical theory.

\subsection{Classes of weak equivalence}

We define three types of equivalence in this section and present them in 
decreasing order of 
generality. They each limit the set of 
triples, $(A,B,C)$, that are used to distinguish between independence 
models represented by DMGs.

\subsubsection{General weak equivalence}

If $\mathcal{G}_1 = (V,E_1)$ and $\mathcal{G}_2 = (V,E_2)$ are Markov 
equivalent, then $(A,B,C) \in \mathcal{I}(\mathcal{G}_1)$ if and only if 
$(A,B,C) \in \mathcal{I}(\mathcal{G}_2)$ for all $A,B,C \subseteq V$. This 
means that Markov equivalence requires the independence models of 
$\mathcal{G}_1$ and $\mathcal{G}_2$ to agree on all 
triplets in the set $\mathcal{P} = \{(A,B,C): A,B,C \subseteq V \}$. A very 
general approach 
to defining weaker notions of equivalence is to only 
compare the independence models on a subset of $\mathcal{P}$.

\begin{defn}[General weak equivalence]
	Let $\mathcal{J} \subseteq \{(A,B,C): A,B,C \subseteq V \}$. We say that 
	$\mathcal{G}_1=(V,E_1)$ and $\mathcal{G}_2=(V,E_2)$ are 
	$\mathcal{J}$-weakly equivalent if 
	
	$$
	\mathcal{I}(\mathcal{G}_1) \cap \mathcal{J} = \mathcal{I}(\mathcal{G}_2) 
	\cap \mathcal{J}.
	$$
	
	\noindent We use $\mathcal{I}_\mathcal{J}(\mathcal{G}_1)$ to denote the 
	\emph{$\mathcal{J}$-weak independence 
	model} 
	induced by $\mathcal{G}_1$, 	
	$\mathcal{I}_\mathcal{J}(\mathcal{G}_1) = \mathcal{I}(\mathcal{G}_1) \cap 
	\mathcal{J}
	$. We use $[\mathcal{G}_1]_\mathcal{J}$ to denote the $\mathcal{J}$-weak 
	equivalence class of $\mathcal{G}_1$, that is, the set of graphs, 
	$\mathcal{G}=(V,E)$, such 
	that 
	$\mathcal{I}_\mathcal{J}(\mathcal{G})=\mathcal{I}_\mathcal{J}(\mathcal{G}_1)$.
	\label{def:genEqui}
\end{defn}

\begin{prop}
	Let $\mathcal{J} \subseteq \mathcal{P}$ and let $V$ be a finite set. 
	Definition \ref{def:genEqui} defines an equivalence relation on the set of 
	DMGs with node set $V$.
	\label{prop:weakIsEqui}
\end{prop}

\begin{proof}
	Let $\mathcal{G}$ be a DMG. We see that $\mathcal{G}$ is 
	$\mathcal{J}$-weakly equivalent with itself such that the relation is 
	reflexive. The 
	relation is also symmetric and transitive. 
\end{proof}

The next statement follows directly from the definition of weak equivalence.

\begin{prop}
	Let $\mathcal{J}_1\subseteq \mathcal{J}_2\subseteq \mathcal{P}$ and let 
	$\mathcal{G}$ be a DMG. It holds that 
	$\mathcal{I}_{\mathcal{J}_1}(\mathcal{G})\subseteq 
	\mathcal{I}_{\mathcal{J}_2}(\mathcal{G})$.
\end{prop}

A Markov equivalence class has a greatest element. However, a 
$\mathcal{J}$-weak equivalence class does not necessarily have a 
greatest element as illustrated by the following example.

\begin{exmp}
We consider the graph, $\mathcal{G}$, in Figure \ref{fig:genEquiNoMax} with all 
loops included 
	as well. We define the set $\mathcal{J} \subseteq \mathcal{P}$,
	
	\begin{align*}
		\mathcal{J} = \left[\bigcup_{\alpha,\beta \in V} (\alpha,\beta,\beta) 
		\right] \cup \bigg\{\big(1,5, \{2,3,4,5 \} \big) \bigg\}.
	\end{align*}
	
	\noindent We also define three other graphs from $\mathcal{G} = (V,E)$, 
	$\mathcal{G}_i = (V,E_i)$, where $i = 1,2,3$, and
	
	\begin{align*}
		E_1 = E \cup \{2\leftrightarrow 3 \},\qquad E_2 = E \cup 
		\{3\leftrightarrow 4 
		\},\qquad E_3 = E \cup \{2\leftrightarrow 3, 3\leftrightarrow 4 \}.
	\end{align*}
	
	\noindent  Graphs $\mathcal{G}_1$ and $\mathcal{G}_2$ are both 
	$\mathcal{J}$-weakly equivalent with $\mathcal{G}$ which can be seen from 
	simply listing their $\mathcal{J}$-weak independence models. 
	
	We see that $(1,5, \{2,3,4,5\}) \in \mathcal{I}_\mathcal{J}(\mathcal{G})$, 
	but $(1,5, \{2,3,4,5\}) \notin \mathcal{I}_\mathcal{J}(\mathcal{G}_3)$ 
	which means that 
	$\mathcal{G}$ and $\mathcal{G}_3$ are not $\mathcal{J}$-weakly equivalent. 
	We have that $\mathcal{G}_1,\mathcal{G}_2\in [\mathcal{G}]_\mathcal{J}$, 
	and a greatest element of $[\mathcal{G}]_\mathcal{J}$ must be a supergraph 
	of both $\mathcal{G}_1$ and $\mathcal{G}_2$, and therefore of 
	$\mathcal{G}_3$. If $\mathcal{N}$ is a supergraph of $\mathcal{G}_3$, then 
	$\mathcal{I}_\mathcal{J}(\mathcal{N}) \subseteq 
	\mathcal{I}_\mathcal{J}(\mathcal{G}_3) \subsetneq 
	\mathcal{I}_\mathcal{J}(\mathcal{G})$, and  
	we conclude that the 
	$\mathcal{J}$-weak equivalence class of $\mathcal{G}$ does not contain a 
	greatest element.
	
		\begin{figure}[h]
			\centering
			\begin{tikzpicture}
			\node[shape=circle,draw=black] (a) at (0,0) {$1$};
			\node[shape=circle,draw=black] (b) at (2,0) 
			{$2$};
			\node[shape=circle,draw=black] (c) at (4,0) {$3$};
			\node[shape=circle,draw=black] (d) at (6,0) {$4$};
			\node[shape=circle,draw=black] (e) at (8,0) {$5$};
			
			\path 
			[->](a) edge [bend left = 0] node {} (b)
			[->](b) edge [bend left = 20] node {} (c)
			[->](c) edge [bend left = 20] node {} (b)
			[->](d) edge [bend left = 20] node {} (c)
			[->](c) edge [bend left = 20] node {} (d)
			[<->](e) edge [bend left = 0] node {} (d)
	[->](a) edge [loop above,  in=220,out=250, min distance=5mm] node {} (a)
	[->](b) edge [loop above,  in=220,out=250, min distance=5mm] node {} (b)
	[->](c) edge [loop above,  in=220,out=250, min distance=5mm] node {} (c)
	[->](d) edge [loop above,  in=220,out=250, min distance=5mm] node {} (d)
	[->](e) edge [loop above,  in=220,out=250, min distance=5mm] node {} (e);
			\end{tikzpicture}
			\caption{The graph $\mathcal{G}$ in Example 
			\ref{exmp:genEquiNoMax}.}
			\label{fig:genEquiNoMax}
		\end{figure}
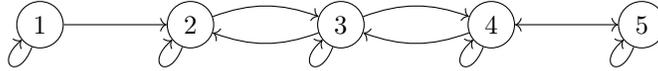
	\label{exmp:genEquiNoMax}
\end{exmp}

Let $\mathcal{J} \subseteq \mathcal{P}$. If two graphs are Markov equivalent, 
they are of course also equivalent when 
restricting to comparisons on the set $\mathcal{J}$. Therefore, every graph is 
also weakly equivalent with the 
unique, maximal graph of its Markov equivalence class. However, the above 
example shows that when considering a general $\mathcal{J}$-weak 
equivalence, an equivalence class need not have a greatest element as the 
maximal Markov equivalent graph need not be a greatest element of the larger 
weak equivalence class. This leads us to introducing the notion of a 
\emph{homogeneous} weak equivalence by imposing a regularity condition on the 
set 
$\mathcal{J}$. The equivalence classes of a homogeneous weak equivalence 
relation do indeed contain a greatest element (Section 
\ref{sec:greatElemWEqui}).

\subsubsection{Homogeneous weak equivalence}

We define \emph{homogeneous} equivalence relation to obtain well-behaved 
equivalence classes.

\begin{defn}[Homogeneous equivalence]
	Consider some weak equivalence induced by $\mathcal{J} \subseteq 
	\mathcal{P}$. 
	We say that this equivalence is \emph{homogeneous} if there exists a set 
	$\mathcal{C}$, $\mathcal{C} \subseteq \{C: C\subseteq V\}$, such 
	that 

$$
\mathcal{J} = \{(A,B,C)\in \mathcal{P}: A,B \subseteq V, C\in \mathcal{C} \}.
$$

\noindent In this case, we will also say that the set $\mathcal{J}$ is 
\emph{homogeneous} and we will say that $\mathcal{C}$ is the \emph{collection 
of conditioning sets} of $\mathcal{J}$.
	\label{def:homogenEqui}
\end{defn}

In other words, a \emph{homogeneous} equivalence relation is one that restricts 
only the set of conditioning sets, $C$. That is, if $\mathcal{J}$ is 
homogeneous, then $\mathcal{J}$-weak equivalence of $\mathcal{G}_1$ and 
$\mathcal{G}_2$ means that for all $A,B \subseteq V$ and $C \in \mathcal{C}$ 
we have $(A,B,C) \in 
\mathcal{I}(\mathcal{G}_1)$ if and only if $(A,B,C) \in 
\mathcal{I}(\mathcal{G}_2)$ 
where $\mathcal{C}$ is some collection of subsets of $V$.  Therefore, the 
restriction of the independence model imposed by a homogeneous $\mathcal{J}$ 
only applies to 
the conditioning 
sets. 

\subsubsection{$k$-weak equivalence}

We will now introduce a certain type of homogeneous equivalence which simply 
restricts the size of the conditioning sets.

\begin{defn}[$k$-weak equivalence]
	Let $0\leq k \leq n$. We say 
	that $\mathcal{G}_1$ and 
	$\mathcal{G}_2$ are \emph{$k$-weakly equivalent} if for 
	all $C$ such that 
	$\vert C\vert \leq k$, it holds that $(A,B,C) \in 
	\mathcal{I}(\mathcal{G}_1)$ if and 
	only if $(A,B,C) \in \mathcal{I}(\mathcal{G}_2)$.
	\label{def:kWeakEqui}
\end{defn}

 The above is formulated slightly differently than Definitions 
 \ref{def:genEqui} and \ref{def:homogenEqui}, however, $k$-weak equivalence is 
 a 
 homogeneous weak equivalence relation by using the set $\mathcal{C} = 
 \{C\subseteq V: \vert 
 C\vert \leq k \}$ in Definition \ref{def:homogenEqui}. On the other hand, not 
 all homogeneous 
 equivalences correspond to a $k$-weak equivalence. We see that $k$-weak 
 equivalence only compares graphs using `small' conditioning sets of size less 
 than $k$ and that Markov equivalence is the same as $n$-weak equivalence.

For $\mathcal{G}_1 = (V,E_1)$, we use $\mathcal{I}_k(\mathcal{G}_1)$ to denote 
the $k$-weak independence model of $\mathcal{G}_1$, 
$\mathcal{I}_k(\mathcal{G}_1) 
= \{(A,B,C) \in \mathcal{I}(\mathcal{G}_1), \vert C\vert \leq k \}$. We let 
$[\mathcal{G}]_k$ denote the set of 
graphs on nodes $V$ that are $k$-weakly equivalent with 
$\mathcal{G}$, and we say that $[\mathcal{G}]_k$ is the \emph{$k$-weak 
	equivalence class} of $\mathcal{G}$. When $k = n$, we also use 
	$\mathcal{I}(\mathcal{G})$, that is,  
$\mathcal{I}(\mathcal{G}) = \mathcal{I}_n(\mathcal{G})$.

\subsection{Properties of weak equivalence}
\label{ssec:propOfWeakEqui}

This section describes some properties of weak equivalence and weak 
equivalence 
classes. Throughout the section $\mathcal{J}$ is a subset of $\mathcal{P} = 
\{(A,B,C) : A,B,C \subseteq V \}$. For 
Markov equivalence, it holds that $\mathcal{G}_1\subseteq 
\mathcal{G}_2$ implies
$\mathcal{I}(\mathcal{G}_2) \subseteq \mathcal{I}(\mathcal{G}_1)$ which follows 
from the definition of $\mu$-separation. This is quite natural as a larger 
graph has more 
edges, therefore fewer independences. The same holds for weak equivalence 
classes as shown by the next proposition.

\begin{prop}
	If $\mathcal{G}_1\subseteq \mathcal{G}_2$, then 
	$\mathcal{I}_\mathcal{J}(\mathcal{G}_2) 
	\subseteq 
	\mathcal{I}_\mathcal{J}(\mathcal{G}_1)$.
	\label{prop:monotonicity}
\end{prop}

\begin{proof}
	If $(A,B,C) \in \mathcal{I}_\mathcal{J}(\mathcal{G}_2)$ then $(A,B,C) \in 
	\mathcal{I}(\mathcal{G}_2)$ and $(A,B,C) \in \mathcal{J}$, and 
	therefore $(A,B,C) \in \mathcal{I}(\mathcal{G}_1)$. This means that 
	$(A,B,C) \in 
	\mathcal{I}_\mathcal{J}(\mathcal{G}_1)$.
\end{proof}

\begin{prop}[Well-ordered $\mathcal{J}$-classes]
	Let $\mathcal{J}_1 \subseteq \mathcal{J}_2 \subseteq \mathcal{P}$. If 
	$\mathcal{G}_1$ and 
	$\mathcal{G}_2$ are 
	$\mathcal{J}_2$-weakly 
	equivalent, then they are also $\mathcal{J}_1$-weakly equivalent.
	\label{prop:wellorder}
\end{prop}

\begin{proof}
	Let $(A,B,C) \in \mathcal{I}_{\mathcal{J}_1}(\mathcal{G}_1)$, then 
	$(A,B,C) \in 
	\mathcal{I}(\mathcal{G}_1)$ and $(A,B,C)\in \mathcal{J}_1$. Therefore 
	$(A,B,C)\in \mathcal{J}_2$ 
	and  $(A,B,C) \in \mathcal{I}_{\mathcal{J}_2}(\mathcal{G}_1) = 
	\mathcal{I}_{\mathcal{J}_2}(\mathcal{G}_2)$. It follows that $(A,B,C) \in 
	\mathcal{I}(\mathcal{G}_2)$ and $(A,B,C) \in 
	\mathcal{I}_{\mathcal{J}_1}(\mathcal{G}_2)$. Interchanging the roles of 
	$\mathcal{G}_1$ and $\mathcal{G}_2$ and repeating the argument gives the 
	result.
\end{proof}

From the above, we also see that $\mathcal{J}_1\subseteq \mathcal{J}_2$ implies 
$[\mathcal{G}]_{\mathcal{J}_2}\subseteq 
[\mathcal{G}]_{\mathcal{J}_1}$. The next corollary follows directly 
from the above proposition.

\begin{cor}[Well-ordered $k$-classes]
	Let $0\leq k_1 \leq k_2 \leq n$. If $\mathcal{G}_1$ and 
	$\mathcal{G}_2$ are 
	$k_2$-weakly 
	equivalent, then they are also $k_1$-weakly equivalent.
	\label{cor:wellorder}
\end{cor}

\begin{defn}
	We say that $\mathcal{J}$ is \emph{singleton stable} if for all 
	$A,B,C\subseteq V$, $(A,B,C) 
	\in \mathcal{J}$ implies that $(\alpha,\beta,C)\in \mathcal{J}$ for all 
	$\alpha\in A$ and $ \beta \in B$.
\end{defn}

Note that the requirement is only on the $A$- and $B$-sets, not the $C$-set. If 
$\mathcal{J}$ is homogeneous and $(A,B,C)\in \mathcal{J}$, then 
$(\bar{A},\bar{B},C)\in \mathcal{J}$ for all $\bar{A},\bar{B}\subseteq V$, thus 
a homogeneous $\mathcal{J}$ is also singleton stable. The following proposition 
shows, for a singleton stable 
$\mathcal{J}$, the independence model $\mathcal{I}_\mathcal{J}(\mathcal{G})$
is characterized by the independences $(A,B,C)$ where $A$ 
and $B$ are singletons and $A$ and $C$ are disjoint. This proof uses the fact 
that 
$\mu$-separation models satisfy so-called left and 
right composition as well as left and right
decomposition which are \emph{asymmetric graphoid properties} 
\citep{didelezUAI2006,mogensenUAI2018}. These are similar to classical graphoid 
properties \citep{lauritzen1996}, but left and right version are needed due to 
the lack of symmetry.

\begin{prop}
	Let $\mathcal{J}$ be singleton stable, let $V$ be a finite set and let 
	$\mathcal{S} = 
	\{(A,B,C) \in \mathcal{P}: \vert 
	A\vert =\vert B\vert = 1, A\cap C=\emptyset \}$. 
	If 
	$\mathcal{I}_\mathcal{J}(\mathcal{G}_1) \cap \mathcal{S} \subseteq 
	\mathcal{I}_\mathcal{J}(\mathcal{G}_2) \cap \mathcal{S} $, then 
	$\mathcal{I}_\mathcal{J}(\mathcal{G}_1) \subseteq 
	\mathcal{I}_\mathcal{J}(\mathcal{G}_2)$.
	\label{prop:singletonGraphIndep}
\end{prop}

Without the assumption of singleton stability, this above statement is not 
true. 
For instance, if $\mathcal{J}\cap \mathcal{S} = \emptyset$, then 
$\mathcal{I}(\mathcal{G}_1)_\mathcal{J} \cap \mathcal{S} \subseteq 
\mathcal{I}(\mathcal{G}_2)_\mathcal{J} \cap \mathcal{S} $ is trivially true for 
any pair of graphs.

\begin{proof}
	Let $(A,B,C) \in \mathcal{I}_\mathcal{J}(\mathcal{G}_1)$. If $A$ or $B$ is 
	empty, then it follows immediately that $(A,B,C) \in 
	\mathcal{I}_\mathcal{J}(\mathcal{G}_2)$. Assume that $A$ and $B$ are both 
	nonempty. We can write $A = 
	\{\alpha_1,\ldots,\alpha_{n_a}\}$ and $B = \{\beta_1,\ldots,\beta_{n_b}\}$. 
	From the definition of $\mu$-separation and using singleton stability of 
	$\mathcal{J}$ it follows that 
	$(\alpha_i,\beta_j,C) \in 
	\mathcal{I}_\mathcal{J}(\mathcal{G}_1)$ for all $i = 1,\ldots,n_a$ and $j = 
	1,\ldots,n_b$. Therefore $(\alpha_i,\beta_j,C) \in  
	\mathcal{I}_\mathcal{J}(\mathcal{G}_2)$ for all $i = 1,\ldots,n_a$ and $j = 
	1,\ldots,n_b$ (if $\alpha_i \in C$, then it holds trivially). From the 
	definition of $\mu$-separation, $(A,B,C) \in  
	\mathcal{I}(\mathcal{G}_2)$ and therefore also $(A,B,C) \in  
	\mathcal{I}_\mathcal{J}(\mathcal{G}_2)$. 
\end{proof}

\begin{prop}[Maximality]
	The graph $\mathcal{G}=(V,E) \in [{\mathcal{G}}_1]_\mathcal{J}$ is 
	maximal in $[{\mathcal{G}}_1]_\mathcal{J}$ if and only if it is 
	complete 
	or if $\mathcal{G} + e \notin [{\mathcal{G}}_1]_\mathcal{J}$ for all 
	edges $e$ 
	such that $e\notin E$.
	\label{prop:max}
\end{prop}

When $\mathcal{G}=(V,E) \in [{\mathcal{G}}_1]_\mathcal{J}$ is maximal in 
$[{\mathcal{G}}_1]_\mathcal{J}$, then we also say that $\mathcal{G}$ is 
\emph{$\mathcal{J}$-maximal} (the equivalence class is implicit as a graph can 
only be maximal in its own equivalence class). A graph is $\mathcal{J}$-maximal 
if the addition of any edge 
will change the $\mathcal{J}$-weak independence model.

\begin{proof}
	If $\mathcal{G}$ is complete, then it is clearly maximal. If 
	$\mathcal{G}\subsetneq {\mathcal{G}}_2$, then $\mathcal{G}\subsetneq 
	\mathcal{G}+e\subseteq {\mathcal{G}}_2$ for some $e\notin E$. We have 
	$\mathcal{I}_\mathcal{J}({\mathcal{G}}_2) \subseteq 
	\mathcal{I}_\mathcal{J}({\mathcal{G}} + e)$ and 
	$\mathcal{I}_\mathcal{J}({\mathcal{G}}+e) \subsetneq 
	\mathcal{I}_\mathcal{J}({\mathcal{G}} 
	)=\mathcal{I}_\mathcal{J}({\mathcal{G}}_1 )$ and therefore 
	${\mathcal{G}}_2\notin [{\mathcal{G}}_1]_\mathcal{J}$.
	
	On the other hand, assume that $\mathcal{G}$ is maximal, and that 
	$\mathcal{G}$ is not complete. It follows from the definition of maximality 
	that $\mathcal{G}+e \notin [\mathcal{G}]_\mathcal{J}$ for all $e\notin E$. 
\end{proof}

If $\mathcal{G}_1 \subseteq \mathcal{G}_2$ then
$\mathcal{I}(\mathcal{G}_2)\subseteq \mathcal{I}(\mathcal{G}_1)$ 
(Proposition \ref{prop:monotonicity}). One may 
ask if $\mathcal{I}(\mathcal{G}_2)\subseteq \mathcal{I}(\mathcal{G}_1)$ 
implies $\mathcal{G}_1 \subseteq \mathcal{G}_2$. The next example shows that 
this is not the case, also not for maximal graphs.  

\begin{exmp}
	We consider two graphs, $\mathcal{G}_1$ and 
	$\mathcal{G}_2$ as shown in Figure \ref{fig:indSubsetNotGraphSubset} (both 
	graphs also have all directed and bidirected loops). Let $\mathcal{S} = 
	\{(A,B,C): \vert 
	A\vert =\vert B\vert = 1, A\cap C=\emptyset \}$. Then 
	$\mathcal{I}(\mathcal{G}_1) \cap \mathcal{S}$ 
	equals
	
	$$
	\biggl\{ \Bigl(2,3,1\Bigr), \Bigl(2,3, \{1,3\}\Bigr), \Bigl(3,2,1\Bigr), 
	\Bigl(3,2,\{1,2\}\Bigr), 
	\Bigl(3,1,1\Bigr), 
	\Bigl(3,1,\{1,2\}\Bigr) 
	\biggr\}.
	$$
	
	\noindent $\mathcal{I}(\mathcal{G}_2)\cap \mathcal{S}$ equals
	
	$$
	\biggl\{\Bigl(2,3,\{1,3\}\Bigl), \Bigl(3,1,1\Bigl) \biggr\}
	$$
	
	\noindent and therefore it is a subset of $\mathcal{I}(\mathcal{G}_1) \cap 
	\mathcal{S}$. Markov equivalence corresponds to $\mathcal{J}$-weak 
	equivalence with $\mathcal{J} = \mathcal{P}$, and by Proposition 
	\ref{prop:singletonGraphIndep}, 
	$\mathcal{I}(\mathcal{G}_2) \subseteq \mathcal{I}(\mathcal{G}_1)$. Both 
	graphs are maximal which means that $2\rightarrow 1$ cannot be 
	added to $\mathcal{G}_2$ Markov equivalently. This illustrates that 
	$\mathcal{I}(\mathcal{G}_2)\subseteq \mathcal{I}(\mathcal{G}_1)$ 
	does not imply $\mathcal{G}_1 \subseteq \mathcal{G}_2$, not even if 
	$\mathcal{G}_2$ is maximal.

		\begin{figure}[h]
			\begin{subfigure}{.49\textwidth}
				\centering
				\begin{tikzpicture}
				\node[shape=circle,draw=black] (a) at (0,0) {$1$};
				\node[shape=circle,draw=black] (b) at (2,0) 
				{$2$};
				\node[shape=circle,draw=black] (c) at (4,0) {$3$};
				
				\path 
				[->](a) edge [bend left = 30] node {} (b)
				[->](a.south) edge [bend left = -25] node {} (c)
				[->](b) edge [bend left = 30] node {} (a)
				[<->](a) edge [bend left = 0] node {} (b);
				\end{tikzpicture}
			\end{subfigure}\hfill
			\begin{subfigure}{.49\textwidth}
				\centering
				\begin{tikzpicture}
				\node[shape=circle,draw=black] (a) at (0,0) {$1$};
				\node[shape=circle,draw=black] (b) at (2,0) 
				{$2$};
				\node[shape=circle,draw=black] (c) at (4,0) {$3$};
				
				\path 
				[->](a) edge [bend left = 30] node {} (b)
				[->](a.south) edge [bend left = -25] node {} (c)
				[->](c) edge [bend left = 0] node {} (b)
				[<->](a) edge [bend left = 0] node {} (b);
				\end{tikzpicture}
			\end{subfigure}
			\caption{Graphs $\mathcal{G}_1$ (left) and $\mathcal{G}_2$ 
			(right) in Example 
				\ref{exmp:indSubsetNotGraphSubset}. Loops are omitted from the 
				visualization.}
			\label{fig:indSubsetNotGraphSubset}
		\end{figure}
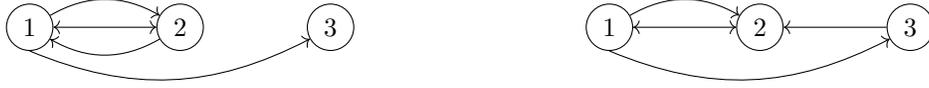
		
		\label{exmp:indSubsetNotGraphSubset}
\end{exmp}

\begin{prop}
	Let $\mathcal{J}_1 \subseteq \mathcal{J}_2$. If $\mathcal{G}$ is 
	$\mathcal{J}_1$-maximal, then it is also $\mathcal{J}_2$-maximal.
	\label{prop:cMaxMono}
\end{prop}

\begin{proof}
	If $\mathcal{G}$ is complete, then it is also $\mathcal{J}_2$-maximal. 
	Assume 
	instead that $\mathcal{G}=(V,E)$ is not complete and $e\notin E$. 
	$\mathcal{G}$ 
	is $\mathcal{J}_1$-maximal, so 
	$\mathcal{G} + e \notin [\mathcal{G}]_{\mathcal{J}_1}$ (Proposition 
	\ref{prop:max}). Using Proposition 
	\ref{prop:monotonicity}, there exist a triple $(A,B,C)$ such that 
	$(A,B,C) \in \mathcal{I}_{\mathcal{J}_1}(\mathcal{G})$ and $(A,B,C) \notin 
	\mathcal{I}_{\mathcal{J}_1}(\mathcal{G} + e)$ and therefore $(A,B,C) \notin 
	\mathcal{I}(\mathcal{G} + e)$. We see that $(A,B,C) \in 
	\mathcal{I}_{\mathcal{J}_2}(\mathcal{G})$ and $(A,B,C) \notin 
	\mathcal{I}_{\mathcal{J}_2}(\mathcal{G} + e)$. It follows that 
	$\mathcal{G}$ is $\mathcal{J}_2$-maximal (Proposition \ref{prop:max}).
\end{proof}

We say that a graph, 
 $\mathcal{G}$, is \emph{$k$-maximal} if is $\mathcal{J}$-maximal for 
 $\mathcal{J} = \{(A,B,C) \in \mathcal{P}: \vert C\vert \leq k \}$ which means 
 that $\mathcal{J}$ induces a $k$-weak equivalence 
 relation.

\begin{cor}
	Let $0\leq k_1\leq k_2 \leq n$. If $\mathcal{G}$ is 
	$k_1$-maximal, then it 
	is also $k_2$-maximal.
	\label{cor:kMaxMono}
\end{cor}

In particular, if a graph is $k$-maximal for some $k\leq 
n$, 
then it is also the unique 
maximal element 
in its Markov equivalence class.

\begin{prop}[Minimality]
	The graph $\mathcal{G}=(V,E) \in [{\mathcal{G}}_1]_\mathcal{J}$ is 
	minimal in $[{\mathcal{G}}_1]_\mathcal{J}$
	if and only if it is
	empty or if $\mathcal{G} - e\notin [{\mathcal{G}}_1]_\mathcal{J}$ for 
	all edges 
	such that $e\in E$.
	\label{prop:min}
\end{prop}

\begin{proof}
	If it is empty, then it is clearly also minimal. Otherwise, let 
	${\mathcal{G}_2} \subsetneq \mathcal{G}$. We have 
	$\mathcal{I}_\mathcal{J}(\mathcal{G}) \subsetneq 
	\mathcal{I}_\mathcal{J}(\mathcal{G} - e) \subseteq 
	\mathcal{I}_\mathcal{J}({\mathcal{G}}_2)$ for $e \in E$ (Proposition 
	\ref{prop:monotonicity}). Therefore, ${\mathcal{G}}_2\notin 
	[\mathcal{G}]_\mathcal{J}$. 
	
	If $\mathcal{G}$ is minimal in $[{\mathcal{G}}_1]_\mathcal{J}$, then it 
	is either the empty graph, or for all $e\in E$, $\mathcal{G} - e \notin 
	[{\mathcal{G}}_1]_\mathcal{J}$ by definition of minimality.
\end{proof}

\begin{prop}
	Let $\mathcal{J}_1\subseteq \mathcal{J}_2$. If $\mathcal{G}$ is 
	$\mathcal{J}_1$-minimal, then it is also $\mathcal{J}_2$-minimal. 
\end{prop}

The proposition states that the property of being minimal is 
preserved when considering a larger set of independences. An equivalence class 
is finite and nonempty, hence, it always contains a maximal element and a 
minimal element. We will show later that it also contains a greatest element. 
However, a least element need not exist and Example \ref{exmp:DMEG} provides an 
example of this.

\begin{proof}
	If $\mathcal{G} = (V,E)$ is empty, then it is also $\mathcal{J}_2$-minimal. 
	Assume 
	instead that $e \in E$. There exists a triple $(A,B,C)$ such that 
	$(A,B,C)\in \mathcal{I}_{\mathcal{J}_1}(\mathcal{G} - e)$ and 
	$(A,B,C)\notin \mathcal{I}_{\mathcal{J}_1}(\mathcal{G})$. Then $(A,B,C)\in 
	\mathcal{J}_1$ and therefore $(A,B,C)\in 
	\mathcal{J}_2$. It follows that $(A,B,C)\in 
	\mathcal{I}_{\mathcal{J}_2}(\mathcal{G} - e)$ and 
	$(A,B,C)\notin \mathcal{I}_{\mathcal{J}_2}(\mathcal{G})$. As this holds for 
	all $e\in E$, we see that $\mathcal{G}$ is $\mathcal{J}_2$-minimal 
	(Proposition \ref{prop:min}).
\end{proof}

\subsubsection{Marginalization}

We say that a class of graphs, $\mathbb{G}$, is \emph{closed under 
	marginalization} if for every $\mathcal{G}= (V,E)\in \mathbb{G}$ and every 
	$O 
\subseteq V$ 
there exists $\mathcal{M} = (O,E_O) \in \mathbb{G}$ such that for every $A,B,C 
\subseteq O$,

\begin{align}
(A,B,C) \in \mathcal{I}_\star(\mathcal{G}) \Leftrightarrow (A,B,C) \in 
\mathcal{I}_\star(\mathcal{M})
\label{eq:margin}
\end{align}

\noindent where $\mathcal{I}_\star(\mathcal{G})$ is the independence model 
induced by 
$\mathcal{G}$. When $\mathbb{G}$ is 
the class of DMGs, $\mathcal{I}_\star(\cdot)$ could for instance be a 
$\mathcal{J}$-weak independence model. 
Appendix 
\ref{app:margin} shows that DMGs with weak 
equivalence are closed 
under marginalization. This follows directly from the analogous result in the 
case of 
Markov equivalence 
\citep{Mogensen2020a} using a so-called \emph{latent projection} \citep[see 
also][]{vermaEquiAndSynthesis,richardson2017}.

\subsection{$k$-weak equivalence}

In this subsection, we restrict our attention to $k$-weak equivalence 
relations. The following result shows that if $\vert V\vert=n$, then $n$-weak 
and 
$(n-1)$-weak equivalence is the same. By convention, $\beta$ is always 
$\mu$-separated from 
$\alpha$ given $C$ when $\alpha\in C$. If $\vert 
C\vert = n$, then $C = V$, and leads to a trivial 
separation.

\begin{prop}
	Let $\mathcal{G}_1 = (V,E_1)$ and $\mathcal{G}_2 = (V,E_2)$ such that 
	$\vert V\vert = n$. Graphs $\mathcal{G}_1$ and $\mathcal{G}_2$ are 
	$(n-1)$-weakly 
	equivalent if and only if they are $n$-weakly equivalent.
	\label{prop:nn1Equi}
\end{prop}

\begin{proof}
	If $\mathcal{G}_1$ and $\mathcal{G}_2$ are $n$-weakly equivalent, then they 
	are also $(n-1)$-weakly equivalent.
	
	On the other hand, assume that $\mathcal{G}_1$ and $\mathcal{G}_2$ are 
	$(n-1)$-weakly equivalent, and let $(\alpha,\beta,C)\in 
	\mathcal{I}_n(\mathcal{G}_1)$ such that $\alpha,\beta\in V$, 
	$C\subseteq V$, and $\alpha\notin C$. We must then have $\vert C\vert \leq 
	n-1$, and therefore $(\alpha,\beta,C)\in 
	\mathcal{I}_n(\mathcal{G}_2)$ by $(n-1)$-weak equivalence of 
	$\mathcal{G}_1$ and $\mathcal{G}_2$. By Proposition 
	\ref{prop:singletonGraphIndep}, this implies 
	$\mathcal{I}_n(\mathcal{G}_1)\subseteq 
	\mathcal{I}_n(\mathcal{G}_2)$. Changing the roles of 
	$\mathcal{G}_1$ and $\mathcal{G}_2$ completes 
	the argument.
\end{proof}

\begin{exmp}[Weak equivalence class]
		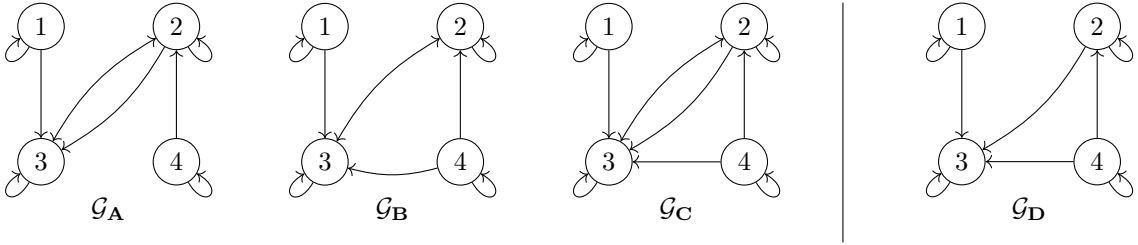
\begin{figure}[h]
			\begin{subfigure}{.19\textwidth}
				\centering
				\begin{tikzpicture}[scale = .9]
				\node[shape=circle,draw=black] (a) at (0,0) {$1$};
				\node[shape=circle,draw=black] (b) at (2,0) 
				{$2$};
				\node[shape=circle,draw=black] (c) at (0,-2) {$3$};
				\node[shape=circle,draw=black] (d) at (2,-2) {$4$};
				
				\node[shape=circle,draw=white] (e) at (1,-2.7) 
				{$\mathcal{G}_\mathrm{\textbf{A}}$};
				
				\path 
				[->](a) edge [bend left = 0] node {} (c)
				[->](b) edge [bend left = 15] node {} (c)
				[->](d) edge [bend left = 0] node {} (b)
				[<->](c) edge [bend left = 15] node {} (b);
	\path			
	[->](a) edge [loop above,  in=210,out=240, min distance=5mm] node {} (a)
	[->](b) edge [loop above,  in=330,out=300, min distance=5mm] node {} (b)
	[->](c) edge [loop above,  in=210,out=240, min distance=5mm] node {} (c)
	[->](d) edge [loop above,  in=330,out=300, min distance=5mm] node {} (d);

				\end{tikzpicture}
			\end{subfigure}\hfill
			\begin{subfigure}{.19\textwidth}
				\centering
				\begin{tikzpicture}[scale = .9]
				\node[shape=circle,draw=black] (a) at (0,0) {$1$};
				\node[shape=circle,draw=black] (b) at (2,0) 
				{$2$};
				\node[shape=circle,draw=black] (c) at (0,-2) {$3$};
				\node[shape=circle,draw=black] (d) at (2,-2) {$4$};
				
				\node[shape=circle,draw=white] (e) at (1,-2.7) 
				{$\mathcal{G}_\mathrm{\textbf{B}}$};
				
				\path
				[->](a) edge [bend left = 0] node {} (c)
				[->](d) edge [bend left = 15] node {} (c)
				[->](d) edge [bend left = 0] node {} (b)
				[<->](c) edge [bend left = 15] node {} (b);
				
	\path			
	[->](a) edge [loop above,  in=210,out=240, min distance=5mm] node {} (a)
	[->](b) edge [loop above,  in=330,out=300, min distance=5mm] node {} (b)
	[->](c) edge [loop above,  in=210,out=240, min distance=5mm] node {} (c)
	[->](d) edge [loop above,  in=330,out=300, min distance=5mm] node {} (d);
				\end{tikzpicture}
			\end{subfigure}\hfill
			\begin{subfigure}{.19\textwidth}
				\centering
				\begin{tikzpicture}[scale = .9]
				\node[shape=circle,draw=black] (a) at (0,0) {$1$};
				\node[shape=circle,draw=black] (b) at (2,0) 
				{$2$};
				\node[shape=circle,draw=black] (c) at (0,-2) {$3$};
				\node[shape=circle,draw=black] (d) at (2,-2) {$4$};
				
				\node[shape=circle,draw=white] (e) at (1,-2.7) 
				{$\mathcal{G}_\mathrm{\textbf{C}}$};
				
				\path 
				[->](a) edge [bend left = 0] node {} (c)
				[->](b) edge [bend left = 15] node {} (c)
				[->](d) edge [bend left = 0] node {} (b)
				[->](d) edge [bend left = 0] node {} (c)
				[<->](c) edge [bend left = 15] node {} (b);
				
	\path			
	[->](a) edge [loop above,  in=210,out=240, min distance=5mm] node {} (a)
	[->](b) edge [loop above,  in=330,out=300, min distance=5mm] node {} (b)
	[->](c) edge [loop above,  in=210,out=240, min distance=5mm] node {} (c)
	[->](d) edge [loop above,  in=330,out=300, min distance=5mm] node {} (d);
				\end{tikzpicture}
			\end{subfigure}\hfill\vline\hfill
			\begin{subfigure}{.19\textwidth}
				\centering
				\begin{tikzpicture}[scale = .9]
				\node[shape=circle,draw=black] (a) at (0,0) {$1$};
				\node[shape=circle,draw=black] (b) at (2,0) 
				{$2$};
				\node[shape=circle,draw=black] (c) at (0,-2) {$3$};
				\node[shape=circle,draw=black] (d) at (2,-2) {$4$};
				
				\node[shape=circle,draw=white] (e) at (1,-2.7) 
				{$\mathcal{G}_\mathrm{\textbf{D}}$};
				
				\path 
				[->](a) edge [bend left = 0] node {} (c)
				[->](b) edge [bend left = 15] node {} (c)
				[->](d) edge [bend left = 0] node {} (b)
				[->](d) edge [bend left = 0] node {} (c);
				
	\path			
	[->](a) edge [loop above,  in=210,out=240, min distance=5mm] node {} (a)
	[->](b) edge [loop above,  in=330,out=300, min distance=5mm] node {} (b)
	[->](c) edge [loop above,  in=210,out=240, min distance=5mm] node {} (c)
	[->](d) edge [loop above,  in=330,out=300, min distance=5mm] node {} (d);
				\end{tikzpicture}
			\end{subfigure}
			\caption{Graphs from 
				Example \ref{exmp:wEqui}. All bidirected loops are present in 
				the graphs 
				but omitted from the 
				visualization.}
			\label{fig:wEqui}
		\end{figure} 	
	
In this example, we restrict our attention to graphs with all loops included in 
which case graphs 
$\mathcal{G}_\mathrm{\textbf{A}}$, $\mathcal{G}_\mathrm{\textbf{B}}$, and 
$\mathcal{G}_\mathrm{\textbf{C}}$ in Figure 
\ref{exmp:wEqui} constitute a 
$2$-weak equivalence class and a $3$-weak equivalence class. Graph 
$\mathcal{G}_C$ 
is the greatest element in both cases. We have that 
$[\mathcal{G}_\mathrm{\textbf{C}}]_2 
\subseteq [\mathcal{G}_\mathrm{\textbf{C}}]_1$ (Corollary 
\ref{cor:wellorder}) and 
$[\mathcal{G}_\mathrm{\textbf{C}}]_1 = 
\{\mathcal{G}_\mathrm{\textbf{A}}, \mathcal{G}_\mathrm{\textbf{B}}, 
\mathcal{G}_\mathrm{\textbf{C}}, \mathcal{G}_\mathrm{\textbf{D}}
 \}$. We 
		see that 
$\mathcal{G}_\mathrm{\textbf{C}}$ and $\mathcal{G}_\mathrm{\textbf{D}}$ are not 
$2$-weakly equivalent as $2$ is 
$\mu$-separated from $3$ given $\{2,4 \}$ in $\mathcal{G}_\mathrm{\textbf{D}}$ 
while this is not 
the case in $\mathcal{G}_\mathrm{\textbf{C}}$.
\label{exmp:wEqui}
\end{exmp}

\begin{exmp}
	We give an example of how `strong connectivity', that is, many similar 
	paths, may 
	lead to more edges in a $k$-weak graph than in the corresponding $n$-weak 
	graph, $k\leq n$. For this purpose, we consider graphs $\mathcal{G}_1$ and 
	$\mathcal{G}_2$ as shown in 
	Figure \ref{fig:demarcWeakEqui}. The graph $\mathcal{G}_2$ is 
	$2$-maximal and therefore it is $k$-maximal for all $k\geq 2$, including $k 
	= n$ (Corollary \ref{cor:kMaxMono}). We construct a smaller graph, 
	$\mathcal{G}_1$, by removing 
	$1\rightarrow 2$. The smaller graph is not Markov equivalent, but 
	it is $(n-3)$-equivalent.
	
	In terms of interpretation, we see that in this class of graphs there are 
	many directed paths from $\alpha$ to $\beta$ and if there are more than 
	$k$, then the edge $\alpha\rightarrow\beta$ can be added $k$-weakly 
	equivalently. In a graphical sense, nodes $\alpha$ and $\beta$ are 
	`strongly' 
	connected as there are more than $k$ disjoint, directed paths from $\alpha$ 
	to 
	$\beta$ and they cannot all be blocked by conditioning on at most $k$ nodes.

	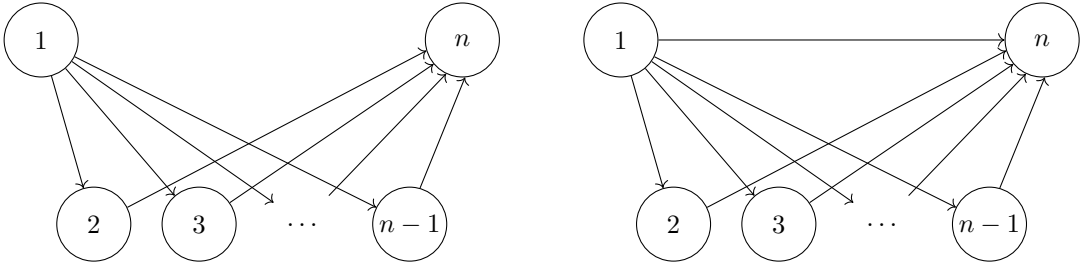
\begin{figure}[h]
		\centering
		\begin{subfigure}{.49\textwidth}
			\begin{tikzpicture}[scale = .7]
			\node[shape=circle,draw=black,inner sep=0pt,minimum size=28pt] (a) 
			at (0,1.5) {$1$};
			\node[shape=circle,draw=black,inner sep=0pt,minimum size=28pt] (b) 
			at (1,-2) 
			{$2$};
			\node[shape=circle,draw=black,inner sep=0pt,minimum size=28pt] (c) 
			at (3,-2) 
			{$3$};
			\node[shape=circle,draw=white,inner sep=0pt,minimum size=28pt] (d) 
			at (5,-2) 
			{\ldots};
			\node[shape=circle,draw=black,inner sep=0pt,minimum size=28pt] (e) 
			at (7,-2) 
			{$n-1$};
			\node[shape=circle,draw=black,inner sep=0pt,minimum size=28pt] (f) 
			at (8,1.5) {$n$};
			
			\path 
			[->](a) edge [bend left = 0] node {} (b)
			[->](a) edge [bend left = 0] node {} (c)
			[->](a) edge [bend left = 0] node {} (d)
			[->](a) edge [bend left = 0] node {} (e)
			[->](b) edge [bend left = 0, in=170, looseness = 0] node {} (f)
			[->](c) edge [bend left = 0, in=185, looseness = 0] node {} (f)
			[->](d) edge [bend left = 0, in=195, looseness = 0] node {} (f)
			[->](e) edge [bend left = 0, in=200, looseness = 0] node {} (f);
			\end{tikzpicture}
		\end{subfigure}\hfill
		\begin{subfigure}{.49\textwidth}
			\begin{tikzpicture}[scale = .7]
			\node[shape=circle,draw=black,inner sep=0pt,minimum size=28pt] 
			(a) 
			at (0,1.5) {$1$};
			\node[shape=circle,draw=black,inner sep=0pt,minimum size=28pt] 
			(b) 
			at (1,-2) 
			{$2$};
			\node[shape=circle,draw=black,inner sep=0pt,minimum size=28pt] 
			(c) 
			at (3,-2) 
			{$3$};
			\node[shape=circle,draw=white,inner sep=0pt,minimum size=28pt] 
			(d) 
			at (5,-2) 
			{\ldots};
			\node[shape=circle,draw=black,inner sep=0pt,minimum size=28pt] 
			(e) 
			at (7,-2) 
			{$n-1$};
			\node[shape=circle,draw=black,inner sep=0pt,minimum size=28pt] 
			(f) 
			at (8,1.5) {$n$};
			
			\path 
			[->](a) edge [bend left = 0] node {} (f)
			[->](a) edge [bend left = 0] node {} (b)
			[->](a) edge [bend left = 0] node {} (c)
			[->](a) edge [bend left = 0] node {} (d)
			[->](a) edge [bend left = 0] node {} (e)
			[->](b) edge [bend left = 0, in=170, looseness = 0] node {} (f)
			[->](c) edge [bend left = 0, in=185, looseness = 0] node {} (f)
			[->](d) edge [bend left = 0, in=195, looseness = 0] node {} (f)
			[->](e) edge [bend left = 0, in=200, looseness = 0] node {} (f);
			\end{tikzpicture}
		\end{subfigure}
		\caption{Graphs $\mathcal{G}_1$ (left) and $\mathcal{G}_2$ (right) from 
			Example \ref{exmp:demarcWeakEqui}. All loops are present in the 
			graphs, but not shown above.}
		\label{fig:demarcWeakEqui}
	\end{figure} 
	\label{exmp:demarcWeakEqui}
\end{exmp}

						We now define \emph{treks} and \emph{directed treks} 
						\cite[see 
						also][]{foygelHalftrek2012,mogensen2020causal}. 
						\cite{foygelHalftrek2012,mogensen2020causal} used paths 
						in their definitions of treks, however, we use walks 
						such 
						that treks between $\alpha$ and $\alpha$ are also 
						allowed.
						
						\begin{defn}[Trek, directed trek]
							Let $\omega$ be a nontrivial walk between $\alpha$ 
							and $\beta$,
							
							$$
							\alpha \sim \ldots \sim_e \beta.
							$$
							
							\noindent We say that $\omega$ is a \emph{trek} if 
							it has no colliders. We 
							say that a trek is \emph{directed} from $\alpha$ to 
							$\beta$ if $\sim_e$ has 
							a head at $\beta$.
							\label{def:trek}
							\end{defn}

							We let $\dtr_\mathcal{G}(\beta)\subseteq V$ denote 
							the set of nodes, $\alpha$, 
							such that there exists a directed trek from 
							$\alpha$ to $\beta$ in 
							$\mathcal{G}$.
							
							\begin{defn}
								Let $\mathcal{G}_1$ and $\mathcal{G}_2$ be 
								DMGs. We say that 
								$\mathcal{G}_1$ and $\mathcal{G}_2$ are 
								\emph{trek equivalent} if for all 
								$\beta\in V$, it holds that
								
								$$
								\dtr_{\mathcal{G}_1}(\beta) = 
								\dtr_{\mathcal{G}_2}(\beta).
								$$
								\label{def:trekEq}
								\end{defn}
								
								A walk is $\mu$-connecting from $\alpha$ to 
								$\beta$ given $\emptyset$ if and 
								only if it is a directed trek from $\alpha$ to 
								$\beta$ which is reflected in the next 
								corollary.
								
								\begin{cor}
									Graphs $\mathcal{G}_1$ and $\mathcal{G}_2$ 
									are $0$-weakly equivalent if and 
									only 
									if they are 
									trek equivalent.
									\label{cor:trekEq}
									\end{cor}
									
									\begin{proof}
										This follows from Corollary 
										\ref{cor:kEqkColl}.
										\end{proof}
										
In Corollary \ref{cor:trekEq}, it is important to 
	define treks using walks, not 
	paths. For instance, the graph in Figure \ref{fig:trek} is $0$-weak 
	equivalent with 
	the complete graph, but the only directed treks from $1$ to $2$ is not are 
	paths. Therefore, the result in Corollary \ref{cor:trekEq} does not hold 
	if directed treks are required to be paths. We say that a DMG $\mathcal{G} 
	= (V,E)$, $V = \{1,2,\ldots,n\}$, \emph{contains 
		a directed cycle} if there is some permutation of $V$, $\sigma$, such 
		that 
	$\sigma(1) \rightarrow \sigma(2) \rightarrow \ldots \rightarrow \sigma(n-1) 
	\rightarrow \sigma(n) \rightarrow \sigma(1)$ in $\mathcal{G}$ (see an 
	example in Figure \ref{fig:dirCycle}).

\begin{figure}	
	\begin{minipage}[t]{.45\textwidth}
			\centering
			\begin{tikzpicture}[scale = .9]
			\node[shape=circle,draw=white,inner sep=0pt,minimum size=18pt] (a1) 
			at (0,0) {};
			\node[shape=circle,draw=white,inner sep=0pt,minimum size=18pt] (b1) 
			at (2,0) 
			{};
			\node[shape=circle,draw=black,inner sep=0pt,minimum size=18pt] (a) 
			at (0,-1) {$1$};
			\node[shape=circle,draw=black,inner sep=0pt,minimum size=18pt] (b) 
			at (2,-1) 
			{$2$};
			\node[shape=circle,draw=white,inner sep=0pt,minimum size=18pt] (c) 
			at (0,-2) {};
			\node[shape=circle,draw=white,inner sep=0pt,minimum size=18pt] (d) 
			at (2,-2) {};

			\path 
			[->](b) edge [bend left = 0] node {} (a);
			\path
	[->](a) edge [loop above,  in=90,out=120, min distance=5mm] node {} (a)
	[->](b) edge [loop above,  in=90,out=120, min distance=5mm] node {} (b);
			\end{tikzpicture}
		\caption{The above graph is trek equivalent with 
			the complete graph, and therefore also $0$-weakly equivalent with 
			the complete graph (Corollary \ref{cor:trekEq}).}
		\label{fig:trek}	
	\end{minipage}\hfill
	\begin{minipage}[t]{.45\textwidth}
		\centering
		\begin{tikzpicture}[scale = .9]
		\node[shape=circle,draw=black,inner sep=0pt,minimum size=18pt] (a) at 
		(0,0) {$1$};
		\node[shape=circle,draw=black,inner sep=0pt,minimum size=18pt] (b) at 
		(2,0) 
		{$2$};
		\node[shape=circle,draw=black,inner sep=0pt,minimum size=18pt] (c) at 
		(4,0) {$3$};
		\node[shape=circle,draw=black,inner sep=0pt,minimum size=18pt] (d) at 
		(4,-2) {$4$};
		\node[shape=circle,draw=white,inner sep=0pt,minimum size=18pt] (e) at 
		(2,-2) {\ldots};
		\node[shape=circle,draw=black,inner sep=0pt,minimum size=18pt] (f) at 
		(0,-2) {$n$};

		\path 
		[->](a) edge [bend left = 0] node {} (b)
		[->](b) edge [bend left = 0] node {} (c)
		[->](c) edge [bend left = 0] node {} (d)
		[->](d) edge [bend left = 0] node {} (e)
		[->](e) edge [bend left = 0] node {} (f)
		[->](f) edge [bend left = 0] node {} (a);
		
			\path
			[->](a) edge [loop above,  in=90,out=120, min distance=5mm] node {} 
			(a)
			[->](b) edge [loop above,  in=90,out=120, min distance=5mm] node {} 
			(b)
			[->](c) edge [loop above,  in=90,out=120, min distance=5mm] node {} 
			(c)
			[->](d) edge [loop above,  in=270,out=240, min distance=5mm] node 
			{} (d)
			[->](f) edge [loop above,  in=270,out=240, min distance=5mm] node 
			{} (f);
		\end{tikzpicture}
		\caption{Directed cycle, see Proposition 
			\ref{prop:cycle}.}
		\label{fig:dirCycle}	
	\end{minipage}
\end{figure}
										
										\begin{prop}
											Let $\mathcal{G} = (V,E)$ be a DMG, 
											$V = \{1,2,\ldots,n\}$, which 
											contains a directed 
											cycle. If 
											every node has a loop, then the
											complete 
											DMG on $V$ is the greatest 
											element of both $[\mathcal{G}]_0$ 
											and 
											$[\mathcal{G}]_1$.
											\label{prop:cycle}
											\end{prop}
											
											\begin{proof}
												For $k = 0$, this follows from 
												Corollary \ref{cor:trekEq} as 
												there is a 
												directed trek between any 
												ordered pair of nodes in 
												$\mathcal{G}$. Let $k=1$ 
												and consider nodes $\alpha$ and 
												$\beta$. We show that there is 
												no 
												separating set, $C$, such that 
												$C\leq 1$. If $C=\emptyset$, 
												this is clear. 
												If $C = \{\gamma\}$, 
												$\gamma\neq \alpha$, then 
												either $\alpha 
												\starrightarrow 
												\ldots \rightarrow 
												\beta$ is open, or $\gamma\neq 
												\beta$ and $\alpha\leftarrow 
												\ldots 
												\leftarrow \beta 
												\starrightarrow \beta$ is open.
												\end{proof}

\section{Greatest elements under homogeneous weak equivalences}
\label{sec:greatElemWEqui}

In the rest of the paper, we assume every weak 
equivalence relation to be \emph{homogeneous} (Definition 
\ref{def:homogenEqui}) as this leads to the existence of a greatest element in 
each equivalence class which we will prove in Subsection \ref{ssec:existGreat}. 
\cite{Mogensen2020a} showed the analogous result in the case of Markov 
equivalence classes. The notions of \emph{$C$-potential siblings} and 
\emph{$C$-potential parents} are central to 
this proof and are introduced in the next subsection.

\subsection{$C$-potential siblings and $C$-potential parents}

The existence of a greatest element in each $\mathcal{J}$-weak equivalence 
class 
can be proven using \emph{$C$-potential 
siblings} 
and 
\emph{$C$-potential parents} as introduced in Definitions \ref{def:Cps} and 
\ref{def:Cpp}. We say that two graphs, $\mathcal{G}_1=(V,E_1)$ and 
$\mathcal{G}_2=(V,E_2)$, are 
\emph{$C$-equivalent}, $C\subseteq V$, if for all $\gamma,\delta\in V$,

\begin{align*}
	(\gamma,\delta,C) \in \mathcal{I}(\mathcal{G}_1) \Leftrightarrow 
	(\gamma,\delta,C) \in \mathcal{I}(\mathcal{G}_2).
\end{align*}

\noindent Let $\alpha,\beta\in V$ and let $e$ be the 
edge $\alpha\leftrightarrow\beta$. The conditions (cs1)-(cs3) in 
Definition \ref{def:Cps} are 
sufficient and necessary
for $\mathcal{G}$ and $\mathcal{G} + e$ to be $C$-equivalent. When $e$ is 
directed, the conditions (cp1)-(cp4) 
in Definition \ref{def:Cpp} 
are 
analogously necessary and sufficient for $\mathcal{G}$ and $\mathcal{G} + e$ to 
be $C$-equivalent. The sufficiency is proven in 
Lemmas \ref{lem:Cps} and \ref{lem:Cpp} and the necessity follows from applying 
Propositions \ref{prop:bidirEdgeCps} and \ref{prop:dirEdgeCpp} to $\mathcal{G} 
+ e$.

Definitions \ref{def:Cps} and \ref{def:Cpp} use an abstract independence model, 
$\mathcal{I}$, while Propositions \ref{prop:gps} and \ref{prop:gpp} describe 
the content of those definitions in the case of a graphical independence model, 
$\mathcal{I} = \mathcal{I}(\mathcal{G})$.

\begin{defn}[$C$-potential sibling]
		Let $\mathcal{I}$ be an independence model over $V$, let 
		$\alpha,\beta\in 
		V$, and let $C\subseteq V$. We say that $\alpha$ and $\beta$ are 
		\emph{$C$-potential siblings} in $\mathcal{I}$ if (cs1)-(cs3) hold.
	\begin{itemize}
		\item[(cs1)] \begin{itemize}[leftmargin=0pt]
			\item[] if $\alpha\notin
			C$: $(\alpha,\beta,C) \notin \mathcal{I}$, and
			\item[] if $\beta\notin C$: $(\beta,\alpha,C) \notin \mathcal{I}$ 
		\end{itemize}
		\item[(cs2)] if $\beta\in C$: for all $\gamma\in V$, $$
		(\gamma,\alpha,C) \in \mathcal{I} \Rightarrow (\gamma,\beta,C) \in 
		\mathcal{I}
		$$ 
		\item[(cs3)] if $\alpha\in C$: for all $\gamma\in V$, $$
		(\gamma,\beta,C) \in \mathcal{I} \Rightarrow (\gamma,\alpha,C) \in 
		\mathcal{I}
		$$ 
	\end{itemize}
	\label{def:Cps}
\end{defn}

\begin{defn}[$C$-potential parent]
	Let $\mathcal{I}$ be an independence model over $V$, let $\alpha,\beta\in 
	V$, and let $C\subseteq V$. We say that $\alpha$ is a \emph{$C$-potential 
	parent} of $\beta$ in $\mathcal{I}$ if (cp1)-(cp4) hold.
	\begin{itemize}
		\item[(cp1)] if $\alpha\notin C$: $(\alpha,\beta,C) \notin \mathcal{I}$
		\item[(cp2)] if $\alpha\notin C$: for all $\gamma\in V$, $$ 
		(\gamma,\beta, C) \in \mathcal{I} \Rightarrow (\gamma,\alpha,C) \in 
		\mathcal{I}$$
		\item[(cp3)] if $\alpha\notin C,\beta\in C$: for all $\gamma,\delta\in 
		C$, $$(\gamma,\delta, C) \in \mathcal{I} \Rightarrow (\gamma,\beta,C) 
		\in \mathcal{I} \vee (\alpha,\delta,C) \in \mathcal{I} $$
		\item[(cp4)] if $\alpha,\beta \notin C$: for all $\gamma\in V$, $$ 
		(\beta,\gamma,C) \in \mathcal{I} \Rightarrow (\alpha,\gamma,C) \in 
		\mathcal{I}$$	
		\end{itemize}
	\label{def:Cpp}
\end{defn}

If $\mathcal{I}$ is graphical, $\mathcal{I} = \mathcal{I}(\mathcal{G})$, and 
$\alpha$ and $\beta$ are $C$-potential siblings in $\mathcal{I}(\mathcal{G})$, 
we will say that
$\alpha\leftrightarrow\beta$ is a 
\emph{$C$-potential sibling edge} between $\alpha$ and $\beta$. Similarly, we 
will say that 
$\alpha\rightarrow\beta$ is a 
\emph{$C$-potential parent edge} from $\alpha$ to $\beta$ if $\alpha$ is a 
$C$-potential parent of 
$\beta$ in $\mathcal{I}(\mathcal{G})$. The following two propositions simply 
rewrite Definitions \ref{def:Cps} and 
\ref{def:Cpp} to explicitly use $\mu$-connecting walks in the case of a 
graphical independence model. Their proofs follow directly from the 
definitions of $\mu$-separation and the independence model 
$\mathcal{I}_\mathcal{J}(\mathcal{G})$.

\begin{prop}[Graphical version of $C$-potential siblings]
Let $\mathcal{I}_\mathcal{J}(\mathcal{G})$ be the weak independence model 
induced by 
$\mathcal{G} = (V,E)$. Let $C\subseteq V$ and let $\mathcal{C}$ be the 
collection of conditioning sets of $\mathcal{J}$. Nodes $\alpha$ and $\beta$ 
are 
$C$-potential siblings if and only if $C \notin \mathcal{C}$ or (gcs1)-(gcs3) 
holds.

	\begin{itemize}
		\item[(gcs1)] \begin{itemize}[leftmargin=0pt]
			\item[] If $\alpha\notin C$, there exists a $\mu$-connecting walk 
			from $\alpha$ to 
			$\beta$ 
			given $C$, and
			\item[] if $\beta\notin C$, there exists a $\mu$-connecting walk 
			from $\beta$ 
			to $\alpha$ 
			given $C$.
		\end{itemize} 
		\item[(gcs2)] If $\beta\in C$, then for all $\gamma\in V$ such that 
		there 
		exists a $\mu$-connecting walk from $\gamma$ to $\beta$ given $C$,
		there also exists a $\mu$-connecting walk from $\gamma$ to $\alpha$ 
		given $C$.
		\item[(gcs3)] If $\alpha\in C$, then for all $\gamma\in V$ such that 
		there 
		exists a $\mu$-connecting walk from $\gamma$ to $\alpha$ given $C$, 
		there also exists a $\mu$-connecting walk from $\gamma$ to $\beta$ 
		given $C$.
	\end{itemize}
	\label{prop:gps}
\end{prop}

\begin{prop}[Graphical version of $C$-potential parents]
	Let $\mathcal{I}_\mathcal{J}(\mathcal{G})$ be the weak independence model 
	induced by 
	$\mathcal{G} = (V,E)$. Let $C\subseteq V$ and let $\mathcal{C}$ be the 
	collection of conditioning sets of $\mathcal{J}$. The node $\alpha$ is a 
	$C$-potential parent of $\beta$ if and only if $C\notin \mathcal{C}$ or 
	(gcp1)-(gcp4) holds.
	
	\begin{itemize}
		\item[(gcp1)] If $\alpha\notin C$, there exists a $\mu$-connecting walk 
		from $\alpha$ to 
		$\beta$ 
		given $C$.
		\item[(gcp2)] If $\alpha\notin C$, then for all $\gamma\in V$ such that 
		there 
		exists a $\mu$-connecting walk from $\gamma$ to $\alpha$ given $C$, 
		there also exists a $\mu$-connecting walk from $\gamma$ to $\beta$ 
		given $C$.
		\item[(gcp3)] If $\alpha\notin C$ and $\beta\in C$, then for all 
		$\gamma,\delta\in V$ such that there exists a $\mu$-connecting walk 
		from 
		$\gamma$ to $\beta$ given $C$ and a $\mu$-connecting walk from $\alpha$ 
		to $\delta$ given $C$, there also exists a $\mu$-connecting walk 
		from $\gamma$ to $\delta$ given $C$.
		\item[(gcp4)] If $\alpha,\beta\notin C$ then for all $\gamma\in V$ such 
		that
		there 
		exists a 
		$\mu$-connecting walk 
		from $\alpha$ to $\gamma$ given $C$, there also exists a 
		$\mu$-connecting 
		walk 
		from $\beta$ to $\gamma$ given $C$.
	\end{itemize}
	
	\label{prop:gpp}
\end{prop} 

The next two propositions show that if $\alpha\leftrightarrow\beta$ 
($\alpha\rightarrow\beta$) is in a graph, then $\alpha$ and $\beta$ are 
$C$-potential siblings ($\alpha$ is a $C$-potential parent of $\beta$)  in the 
independence model of the graph for all 
$C\subseteq V$. The edge $e$ is therefore a $C$-potential sibling edge 
($C$-potential parent edge) in $\mathcal{I}(\mathcal{G} + e)$, and if 
$\mathcal{G}$ and 
$\mathcal{G} + e$ are $C$-equivalent, then $e$ is also a $C$-potential sibling 
edge ($C$-potential parent edge) in $\mathcal{I}(\mathcal{G})$. This means that 
$\mathcal{I}(\mathcal{G})$ satisfying the 
conditions in Definitions \ref{def:Cps} and \ref{def:Cpp} is necessary for 
$C$-equivalence of $\mathcal{G}$ and $\mathcal{G} + e$.

\begin{prop}
	Let $\mathcal{J}$ be homogeneous. If $\alpha \leftrightarrow \beta$ is in 
	$\mathcal{G}$, then $\alpha$ and 
	$\beta$ are
	$C$-potential siblings in $\mathcal{I}_\mathcal{J}(\mathcal{G})$ 
	for all $C \subseteq V$.
	\label{prop:bidirEdgeCps}
\end{prop}

\begin{proof}
	If $C\notin \mathcal{C}$, then it follows immediately. We assume $C\in 
	\mathcal{C}$ and prove (gcs1)-(gcs3). (gcs1) If $\alpha\notin C$, then 
	$\alpha\leftrightarrow\beta$ is a 
	$\mu$-connecting 
	walk in $\mathcal{G}$ given $C$. The proof of the other statement is 
	analogous. (gcs2) Assume that $\beta\in C$ and let $\gamma\in V$ such that 
	there exists a $\mu$-connecting walk from $\gamma$ 
	to $\beta$ given $C$. Composing this with $\beta\leftrightarrow\alpha$ 
	gives a $\mu$-connecting walk from $\gamma$ to $\alpha$ given $C$ as 
	$\beta\in C$. (gcs3) This is shown similarly to 
	(gcs2).
\end{proof}

\begin{prop}
	Let $\mathcal{J}$ be homogeneous. If $\alpha \rightarrow \beta$ is in 
	$\mathcal{G}$, then $\alpha$ is a 
	$C$-potential parent of $\beta$ in $\mathcal{I}_\mathcal{J}(\mathcal{G})$ 
	for all $C \subseteq V$.
	\label{prop:dirEdgeCpp}
\end{prop}

\begin{proof}
	If $C\notin \mathcal{C}$, then this again follows immediately. We instead 
	assume $C\in\mathcal{C}$ and prove (gcp1)-(gcp4). (gcp1) If $\alpha\notin 
	C$, then $\alpha\rightarrow\beta$ is a 
	$\mu$-connecting walk given $C$. (gcp2) Assume that $\alpha\notin C$ 
	and let $\gamma\in V$, and assume there is a $\mu$-connecting walk from 
	$\gamma$ to $\alpha$ given $C$. Concatenating 
	this with the edge $\alpha\rightarrow\beta$ gives a $\mu$-connecting walk 
	from $\gamma$ to $\beta$ given $C$ as $\alpha\notin C$. (gcp3) 
	Assume that $\alpha\notin C,\beta\in C$ and let $\gamma,\delta\in V$ such 
	that there exist a 
	$\mu$-connecting walk from $\gamma$ to $\beta$ given $C$ and a 
	$\mu$-connecting walk from $\alpha$ to $\delta$ given $C$. Concatenating 
	them with the edge $\alpha\rightarrow\beta$ gives a $\mu$-connecting walk 
	from  $\gamma$ to $\delta$ given $C$ as 
	$\beta\in C$ and $\alpha\notin C$. (gcp4) Assume 
	$\alpha,\beta\notin C$ and let $\gamma\in V$ such that 
	there exists a $\mu$-connecting walk from $\alpha$ to $\gamma$ given 
	$C$. Concatenating the edge $\alpha \rightarrow \beta$ with this walk gives 
	a $\mu$-connecting walk from $\beta$ to $\gamma$ given $C$ as 
	$\alpha,\beta\notin C$.
\end{proof}

\subsection{Existence of greatest elements}
\label{ssec:existGreat}

Markov equivalence classes of DMGs are known to contain a greatest element 
\citep{Mogensen2020a}. This means that for an equivalence class 
$[\mathcal{G}]$, there exists a graph $\bar{\mathcal{G}} \in [\mathcal{G}]$ 
such that  $\bar{\mathcal{G}}$ is a supergraph of all graphs 
$\tilde{\mathcal{G}} \in [\mathcal{G}]$. This is a very 
convenient result as it allows a succinct representation of the entire Markov 
equivalence class as illustrated in Example \ref{exmp:greatME}. The main 
result of this section, Theorem \ref{thm:weGreat}, shows that 
$\mathcal{J}$-weak 
equivalence classes enjoy the 
same property when $\mathcal{J}$ is homogeneous. This means that we can 
represent weak equivalence classes in a similar way. Section 
\ref{sec:representkWeak} discusses this further and introduces a hierarchy of 
$k$-weak equivalence 
classes for different values of $k$.

\begin{lem}
	Let $\mathcal{G}_1$ be a DMG. Let $\mathcal{J}$ be homogeneous and let 
	$\mathcal{C}$ be the collection of conditioning sets of $\mathcal{J}$. If 
	$\alpha$ and $\beta$ are $C$-potential siblings for all $C\in \mathcal{C}$ 
	and e denotes the edge $\alpha\leftrightarrow\beta$, 
	then $\mathcal{I}_\mathcal{J}(\mathcal{G}) = 
	\mathcal{I}_\mathcal{J}(\mathcal{G} + e)$.  If 
	$\alpha$ is a $C$-potential parent of $\beta$ for all $C\in \mathcal{C}$ 
	and e denotes the edge $\alpha\rightarrow\beta$, 
	then $\mathcal{I}_\mathcal{J}(\mathcal{G}) = 
	\mathcal{I}_\mathcal{J}(\mathcal{G} + e)$. 
	\label{lem:addJequi}
\end{lem}

\begin{proof}
	The inclusion $
	\mathcal{I}_\mathcal{J}(\mathcal{G} + e)\subseteq 
	\mathcal{I}_\mathcal{J}(\mathcal{G})$ follows from Proposition 
	\ref{prop:monotonicity}. We show the other inclusion by contraposition. 
	Proposition \ref{prop:singletonGraphIndep} implies that it is enough to 
	consider triples of the form $(\gamma,\delta,D)$, $\gamma,\delta\in V$, 
	$D\subseteq V$, $\gamma\notin D$.
	Assume $(\gamma,\delta,D) \notin \mathcal{I}_\mathcal{J}(\mathcal{G} + e)$. 
	If $(\gamma,\delta,D) \notin \mathcal{J}$, then $(\gamma,\delta,D) \notin 
	\mathcal{I}_\mathcal{J}(\mathcal{G})$. If instead $(\gamma,\delta,D) \in 
	\mathcal{J}$, then $(\gamma,\delta,D) \notin 
	\mathcal{I}(\mathcal{G} + e)$ and $D\in\mathcal{C}$. In this case, there 
	exist a $\mu$-connecting walk from $\gamma$ to $\delta$ given $D$ in 
	$\mathcal{G}+e$. Nodes $\alpha$ and $\beta$ are $C$-potential siblings (or 
	$\alpha$ is a $C$-potential parent of $\beta$) for all $C\in \mathcal{C}$, 
	and therefore also for $D\in \mathcal{C}$. Lemma \ref{lem:Cps} (Lemma 
	\ref{lem:Cpp}) gives the result.
\end{proof}

Lemmas \ref{lem:Cps} and \ref{lem:Cpp} that are used in the above proof are 
adaptations of lemmas in \cite{Mogensen2020a}. Appendix \ref{app:proofs} 
describes how to make this generalization.

From an independence model $\mathcal{I}_\mathcal{J}(\mathcal{G})$ such that 
$\mathcal{J}$ is homogeneous we now define a graph on nodes $V$, $\mathcal{G} = 
(V,E)$. As $\mathcal{J}$ is homogeneous, we know that $\mathcal{J} = \{(A,B,C) 
\in \mathcal{P}: C \in \mathcal{C} \}$ for some $\mathcal{C} \subseteq  \{C: 
C\subseteq V \}$. For all $\alpha,\beta\in V$, we include 
the directed edge $\alpha\rightarrow\beta$ if and only if $\alpha$ is a 
$C$-potential parent of $\beta$ for all $C\in \mathcal{C}$. We include the 
bidirected edge $\alpha\leftrightarrow\beta$ if and only if $\alpha$ and 
$\beta$ are $C$-potential siblings for all $C\in \mathcal{C}$. We denote the 
resulting graph by $\mathcal{N}$. We  see that 
${\mathcal{N}}$ is uniquely defined 
from the $\mathcal{J}$-independence model of $\mathcal{G}$, 
$\mathcal{I}_\mathcal{J}(\mathcal{G})$, and is therefore the same for all 
elements of the equivalence class $[\mathcal{G}]_\mathcal{J}$. The following 
shows that $\mathcal{N}$ is a unique maximal element, that is, a greatest 
element, in $[\mathcal{G}]_\mathcal{J}$.

\begin{thm}
	Let $\mathcal{G}$ be a DMG and let $\mathcal{J}$ be homogeneous. The graph 
	$\mathcal{N}$ defined above is $\mathcal{J}$-weakly equivalent with 
	$\mathcal{G}$ and it is the unique maximal element in 
	$[\mathcal{G}]_\mathcal{J}$.
	\label{thm:weGreat}
\end{thm}

\begin{proof}
Let $\bar{\mathcal{G}} \in [\mathcal{G}]_\mathcal{J}$. If a directed edge, 
$\alpha\rightarrow\beta$, 
is in $\bar{\mathcal{G}}$, then $\alpha$ is a $C$-potential parent of $\beta$ 
in $\mathcal{I}_\mathcal{J}(\bar{\mathcal{G}}) = 
\mathcal{I}_\mathcal{J}(\mathcal{G})$ for 
all $C$ (Proposition  \ref{prop:dirEdgeCpp}). This 
means that the directed edge is in ${\mathcal{N}}$. 
Similarly, for bidirected 
edges (Proposition \ref{prop:bidirEdgeCps}), and $\mathcal{N}$ is a supergraph 
of all graphs in $[\mathcal{G}]_\mathcal{J}$. 
	
	Every edge in 
	$\mathcal{N}$ is a $C$-potential edge in $[\mathcal{G}]_\mathcal{J}$ for 
	all $C\in 
	\mathcal{C}$. We can construct a finite sequence of graphs starting from 
	$\mathcal{G}$ and adding the edges that are in $\mathcal{N}$, but not in 
	$\mathcal{G}$, sequentially. Lemma 
	\ref{lem:addJequi} shows that all graphs in this sequence are 
	$\mathcal{J}$-weakly 
	equivalent with $\mathcal{G}$, and therefore so is $\mathcal{N}$.

	In conclusion, ${\mathcal{N}}$ is a greatest element of the 
	equivalence class.
\end{proof}

Theorem \ref{thm:weGreat} is central in our development of graphical modeling
based on weak equivalence as it provides a unique and interpretable 
representative of each equivalence class. We give examples of applications of 
this theorem in Section 
\ref{sec:representkWeak}.

\subsubsection{Comparison with Markov equivalence case}

The above definitions and results are related to results in the case of
Markov equivalence \citep{Mogensen2020a}. Definitions \ref{def:Cps} and 
\ref{def:Cpp} can be thought of as $C$-specific versions of
Definitions 5.1 and 5.5 in \cite{Mogensen2020a}. This leads to $C$-specific 
versions of Propositions 
\ref{prop:bidirEdgeCps} and \ref{prop:dirEdgeCpp} that are analogous to 
propositions in 
\cite{Mogensen2020a}.

Importantly, the potential parent conditions of \cite{Mogensen2020a} use 
multiple conditioning sets and are therefore not amenable as a foundation for 
the proof of Theorem \ref{thm:weGreat}. The conditions in this paper use a 
single $C$ which facilitates the generalization from Markov equivalence classes 
to weak equivalence classes. The reformulation of the definitions also entails 
an 
important change of 
perspective. Instead of describing conditions such that the addition of an edge 
does not change the independence model for \emph{any} conditioning set (Markov 
equivalence), the above conditions describe conditions such that the addition 
of an edge 
does not change the independence model when restricted to a \emph{specific} 
conditioning set. This allows us to aggregate these conditions for any set of 
conditioning sets as defined by a homogeneous $\mathcal{J}$, and from this we 
can prove the existence of a greatest element in this more general setting.

\section{Representation of weak equivalence classes}
\label{sec:representkWeak}

The previous section proved the existence of a greatest element in each weak 
equivalence class when $\mathcal{J}$ is homogeneous. In Subsection 
\ref{ssec:DMEG}, 
we first desribe how this leads to a simple and concise representation of an 
entire equivalence class, and Subsection \ref{ssec:alarm2} illustrates this 
representation using the alarm example. In Subsection \ref{ssec:hier}, we 
restrict our 
attention to $k$-weak equivalence and describe a \emph{hierarchy} of $k$-weak 
equivalence classes. Choosing a 
$k=0,1,\ldots, n-1$ leads to different notions of equivalence with different 
levels of granularity. The 
hierarchy in Subsection \ref{ssec:hier} provides a graphical representation of 
$k$-weak equivalence classes across different values of $k$ which is meant to 
illuminate how equivalence classes change across different values of $k$.

\subsection{Directed mixed equivalence graph}
\label{ssec:DMEG}

The following definition provides a graphical object representing an entire 
weak equivalence class. \cite{Mogensen2020a} gave the same definition in the 
context of Markov equivalence as illustrated in Example \ref{exmp:greatME}.

\begin{defn}[Directed mixed equivalence graph (DMEG)]
	Let $\mathcal{J}$ be homogeneous and assume that $\mathcal{N}=(V,F)$ is 
	$\mathcal{J}$-maximal and $\mathcal{N} \in [\mathcal{G}]_\mathcal{J}$. We 
	define 
	$\bar{F}\subseteq F$ such that $e\in 
	\bar{F}$ if and only if $e\in F$ and there exists $\mathcal{G} =(V,E) \in 
	[\mathcal{G}]_\mathcal{J}$ such 
	that $e\notin E$. We define the \emph{directed mixed weak equivalence 
	graph (DMEG)} of $[\mathcal{G}]_\mathcal{J}$ as the triple 
	$(V,F,\bar{F})$. 
	\label{def:DMEG}
\end{defn}

We visualize a directed mixed weak equivalence graph by drawing the 
corresponding maximal graph and making all edges in $\bar{F}$ dashed (see the 
example in Figure \ref{fig:DMEG}). A DMEG summarizes the equivalence class in 
the following sense. Let 
$\mathcal{N}$ be a $\mathcal{J}$-maximal element such that $\mathcal{N}\in 
[\mathcal{G}]_\mathcal{J}$, that is, $\mathcal{N}$ is the greatest 
element of $[\mathcal{G}]_\mathcal{J}$, and let $\mathcal{N}'$ be the 
corresponding DMEG. If an edge is solid in $\mathcal{N}'$, then this edge is in 
every ${\mathcal{G}}_1\in [\mathcal{G}]_\mathcal{J}$. If an edge is absent in 
$\mathcal{N}'$, then no ${\mathcal{G}}_1\in [\mathcal{G}]_\mathcal{J}$ 
contains 
this edge. If an edge, $e$, is dashed in $\mathcal{N}'$, then there exists a 
${\mathcal{G}}_1 = (V,E)\in [\mathcal{G}]_\mathcal{J}$ such that $e\notin E$. 
Clearly $e$ is in $\mathcal{N} \in [\mathcal{G}]_\mathcal{J}$ and therefore
$e$ is in some elements of $[\mathcal{G}]_\mathcal{J}$, but not in others. 
One 
should note that removing multiple dashed edges from $\mathcal{N}'$ does not 
necessarily lead to a $\mathcal{J}$-weakly equivalent graph as removing an
edge may 
impose 
restrictions on which other edges can be removed while maintaining 
$\mathcal{J}$-weak equivalence. This is related to the fact that a weak 
equivalence class need not contain a least element (see Figure \ref{fig:DMEG}).

\begin{exmp}[Directed mixed equivalence graph]
			\begin{figure}[h]
				\begin{subfigure}{.19\textwidth}
					\centering
					\begin{tikzpicture}[scale = .9]
					\node[shape=circle,draw=black] (a) at (0,0) {$1$};
					\node[shape=circle,draw=black] (b) at (2,0) 
					{$2$};
					\node[shape=circle,draw=black] (c) at (0,-2) {$3$};
					\node[shape=circle,draw=black] (d) at (2,-2) {$4$};
					
					\node[shape=rectangle,draw=white] (e) at 
					(1,-2.7) 
					{$\mathcal{G}_\mathrm{\textbf{A}}$};
					
					\path 
					[->](a) edge [bend left = 0] node {} (c)
					[->](b) edge [bend left = 15] node {} (c)
					[->](d) edge [bend left = 0] node {} (b)
					[<->](c) edge [bend left = 15] node {} (b);
					\end{tikzpicture}
				\end{subfigure}\hfill
				\begin{subfigure}{.19\textwidth}
					\centering
					\begin{tikzpicture}[scale = .9]
					\node[shape=circle,draw=black] (a) at (0,0) {$1$};
					\node[shape=circle,draw=black] (b) at (2,0) 
					{$2$};
					\node[shape=circle,draw=black] (c) at (0,-2) {$3$};
					\node[shape=circle,draw=black] (d) at (2,-2) {$4$};
					
					\node[shape=rectangle,draw=white] (e) at 
					(1,-2.7) 
					{$\mathcal{G}_\mathrm{\textbf{B}}$};
					
					\path
					[->](a) edge [bend left = 0] node {} (c)
					[->](d) edge [bend left = 15] node {} (c)
					[->](d) edge [bend left = 0] node {} (b)
					[<->](c) edge [bend left = 15] node {} (b);
					\end{tikzpicture}
				\end{subfigure}\hfill
				\begin{subfigure}{.19\textwidth}
					\centering
					\begin{tikzpicture}[scale = .9]
					\node[shape=circle,draw=black] (a) at (0,0) {$1$};
					\node[shape=circle,draw=black] (b) at (2,0) 
					{$2$};
					\node[shape=circle,draw=black] (c) at (0,-2) {$3$};
					\node[shape=circle,draw=black] (d) at (2,-2) {$4$};
					
					\node[shape=rectangle,draw=white] (e) at 
					(1,-2.7) 
					{$\mathcal{G}_\mathrm{\textbf{C}}$};
					
					\path 
					[->](a) edge [bend left = 0] node {} (c)
					[->](b) edge [bend left = 15] node {} (c)
					[->](d) edge [bend left = 0] node {} (b)
					[->](d) edge [bend left = 0] node {} (c)
					[<->](c) edge [bend left = 15] node {} (b);
					\end{tikzpicture}
				\end{subfigure}\hfill\vline\hfill
				\begin{subfigure}{.19\textwidth}
					\centering
					\begin{tikzpicture}[scale = .9]
					\node[shape=circle,draw=black] (a) at (0,0) {$1$};
					\node[shape=circle,draw=black] (b) at (2,0) 
					{$2$};
					\node[shape=circle,draw=black] (c) at (0,-2) {$3$};
					\node[shape=circle,draw=black] (d) at (2,-2) {$4$};
					
					\node[shape=rectangle,draw=white] (e) at 
					(1,-2.7) 
					{\footnotesize DMEG};
					
					\path 
					[->](a) edge [bend left = 0] node {} (c)
					[->](b) edge [bend left = 15, dashed] node {} (c)
					[->](d) edge [bend left = 0] node {} (b)
					[->](d) edge [bend left = 0, dashed] node {} (c)
					[<->](c) edge [bend left = 15] node {} (b);
					\end{tikzpicture}
				\end{subfigure}
				\caption{Graphs from 
					Example \ref{exmp:DMEG}. Loops are omitted from the 
					visualization.}
				\label{fig:DMEG}
			\end{figure}
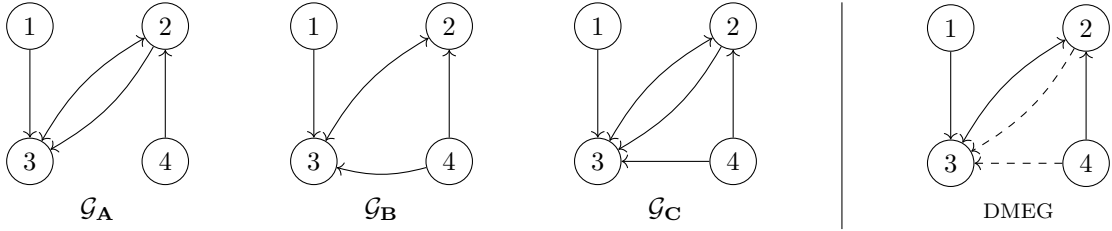 
			
	Graphs $\mathcal{G}_\mathbf{A}$, 
	$\mathcal{G}_\mathbf{B}$, and $\mathcal{G}_\mathbf{C}$ in Figure 
	\ref{fig:DMEG}
	constitute a 2-weak and a 3-weak equivalence class when restricting to DMGs 
	that have all loops present (for simplicity we make this assumption). The 
	graph 
	$\mathcal{G}_\mathbf{C}$ is the greatest element. The corresponding DMEG is 
	also shown in Figure \ref{fig:DMEG}, see Definition \ref{def:DMEG}. 
	The 3-weak equivalence class (2-weak equivalence class) does not contain a 
	least element as 
	removing both $4\rightarrow 3$ and $2\rightarrow 3$ does 
	not lead to a 3-weakly equivalent graph (2-weakly equivalent graph).		
			
	\label{exmp:DMEG}
\end{exmp}

\begin{exmp}
	This example describes a setting which leads to a weak equivalence with a 
	homogeneous $\mathcal{J}$ which is not a $k$-weak equivalence. We 
	consider a setting where a $5$-dimensional process is observed, 
	$V=\{1,2,3,4,5\}$, but not 
	every coordinate process is observed simultaneously. This is essentially a 
	setting with
	\emph{overlapping variable sets}, see, e.g.,
	\cite{danks2002learning,danks2008integrating,triantafillou2010learning,huang2020causal}.
	We assume that data contains observations of $X_t^R$ over an interval 
	$T_R$ for $R \in 
	\mathcal{R}$, 
	
	\begin{align*}
	\Bigl\{ \{1,2,3,4 \}, \{1,2,3,5 \}, \{1,2,4,5 \},\{3,4,5 \} \Bigr\}.
	\end{align*}
	
	\noindent The intervals are disjoint, $T_{R_1} \cap T_{R_2} = \emptyset$ 
	for $R_1\neq 
	R_2$. We will approach this problem by restricting the local independences 
	that can be tested using this data and require that there exists 
	$R\in\mathcal{R}$ such 
	that $A,B,C \subseteq R$ for us to be able to test the local independence 
	$(A,B,C)$.
	
	We see that all local independences, $(\alpha,\beta,C)$, such that 
	$\alpha,\beta\in V$ and $\vert C\vert 
	\leq 1$ can be 
	tested from this data as every triple, $\{\alpha,\beta,\gamma \}$, 
	$\alpha,\beta,\gamma\in V$, is observed simultaneously (that is, 
	$\alpha,\beta,\gamma\in R$
	for some 
	$R \in \mathcal{R}$). We can also test $(\alpha,\beta,\{1,2 \})$ for all 
	$\alpha,\beta \in \{1,2,3,4,5\}$, but not $(4,5,\{1,3 \})$. This means that 
	we can model this using $k$-weak equivalence, but only for $k = 0$ or $k = 
	1$. We can obtain further information by defining 
	
	\begin{align*}
		\mathcal{J} = \Bigl\{(\alpha,\beta,C): \alpha,\beta\in V, \vert C\vert 
		\leq 
		1 \Bigr\}\ \bigcup\ \Bigl\{(\alpha,\beta,C): \alpha,\beta\in V, C = 
		\{1,2\} \Bigr\}.
	\end{align*}
	
	\noindent This leads to a homogeneous weak equivalence relation which is 
	not a $k$-weak equivalence.

	\label{exmp:nonkWeak}
\end{exmp}



\subsection{Hierarchy of $k$-weak equivalence}
\label{ssec:hier}

The previous section describes a graph, the \emph{directed mixed equivalence 
graph}, which can help us understand a single weak equivalence class for a 
fixed, homogeneous $\mathcal{J}$. In this section, we restrict our attention to 
$k$-weak equivalence relations and study a description of $k$-weak equivalence 
classes for varying values of $k$. We consider a fixed node set, $V$. For each 
value of $k$, the $k$-weak equivalence classes form a partition of the DMGs on 
node set $V$, with smaller $k$ corresponding to more coarse partitions. Each 
weak equivalence class can be represented by its maximal element and there is 
an interpretable structure between $k$-weak equivalence classes for 
different values of $k$ which can 
help us understand the connection between these different notions of 
equivalence. This section describes this \emph{hierarchy} 
of $k$-weak equivalences.	

\subsubsection{Levels of granularity}
Let $\mathcal{G}$ be a DMG, and let $k_1 < k_2$. Let $\mathcal{N}_1$ denote the 
greatest element of $[\mathcal{G}]_{k_1}$ and let $\mathcal{N}_2$ denote the 
greatest element of $[\mathcal{G}]_{k_2}$. We know that $[\mathcal{G}]_{k_2} 
\subseteq [\mathcal{G}]_{k_1}$ and it follows that $\mathcal{N}_2\subseteq 
\mathcal{N}_1$. The graphs $\mathcal{N}_1$ and $\mathcal{N}_2$ are both 
representatives of $\mathcal{G}$, but at different levels of granularity. The 
$k_2$-equivalence class of $\mathcal{G}$ is smaller, thus $k_2$-weak 
equivalence is more expressive than $k_1$-weak equivalence. We may ask what 
`approximation error' 
we make by using $k_1$-weak equivalence instead of $k_2$-weak equivalence. Let 
$e$ be an edge in $\mathcal{N}_1$ which is not in $\mathcal{N}_2$. 
We know that $\mathcal{G}$ and $\mathcal{G} + e$ are $k_1$-weakly equivalent, 
so they can only differ on $\mu$-separations with $C$ such that $\vert C\vert > 
k_1$. The approximation error induced by including $e$ is therefore restricted 
to 
`large' conditioning sets. From a practical point 
of view, local independence tests with large conditioning sets are expected to 
perform poorly. This means that the loss of information when testing local 
independences from finite samples may be 
small.

\subsubsection{Forest representation}
We can provide a convenient representation of the $k$-weak equivalence 
hierarchy using \emph{trees} and \emph{forests}. A \emph{tree}, $\mathcal{T} = 
(V_\mathcal{T}, E_\mathcal{T})$, is an undirected graph in which each pair of 
distinct nodes are connected by exactly one path. A \emph{forest} is the 
disjoint union of a set of trees. We can construct a \emph{forest} in the 
following way. 
For a 
fixed $V$, $\vert V\vert = n$, and $k \in K = \{0,1,\ldots,n-1\}$, we consider 
the 
set 
of $k$-weak equivalence classes of DMGs on node set $V$. We let $n_k$ denote 
the number 
of such 
equivalence classes. The $i$'th $k$-weak equivalence class, $i=1,\ldots,n_k$, 
contains a unique maximal 
element and we 
denote this graph by $\mathcal{G}_{k,i}$. We 
do this for every $k\in \{0,1,\ldots,n-1\}$ and define a node set

$$
V_{\mathrm{g}} = \bigcup_{k \in K} \{ (\mathcal{G}_{k,i}, k): i = 
1,2,\ldots,n_k 
\}.
$$

\noindent Note that we write this as a disjoint union as the same graph may be 
a maximal element for different $k$. Therefore, the set 
$V_{\mathrm{g}}$ contains 
pairs $(\mathcal{G}, k)$ such that $\mathcal{G}$ is $k$-maximal. For instance, 
if $\mathcal{G}$ is a maximal element of a $k_1$-weak equivalence class and of 
a $k_2$-weak equivalence class, then $(\mathcal{G}, k_1) \in V_{\mathrm{g}}$ 
and 
$(\mathcal{G}, k_2)\in V_{\mathrm{g}}$ and these are different nodes. 

We now 
construct a forest with node set $V_\mathrm{g}$ in the following way. For 
each $(\mathcal{G}, k)$ such that $k>0$, there exist a unique $(k-1)$-maximal 
graph, 
$\bar{\mathcal{G}}$, such that $\mathcal{G} \in 
[\bar{\mathcal{G}}]_{k-1}$, and we join $({\mathcal{G}},k)$ to 
$(\bar{\mathcal{G}}, k -1)$ by an undirected edge. We call the resulting graph 
the \emph{weak equivalence hierarchy} over $V$ and denote it by 
$\mathcal{H}_V$. For $k < n-1$, we will use $\mathrm{up}(\mathcal{G},k))$ to 
denote the 
(nonempty) set of graphs $\bar{\mathcal{G}}$ such that 
$(\mathcal{G},k)$ and $(\bar{\mathcal{G}},\bar{k})$ are adjacent in 
$\mathcal{H}_V$ and such that $\bar{k} = k+1$. For $k > 0 $, we will use 
$\mathrm{down}(\mathcal{G},k)$ to denote the unique graph 
$\bar{\mathcal{G}}$ such that $(\mathcal{G},k)$ and 
$(\bar{\mathcal{G}},\bar{k})$ are adjacent in $\mathcal{H}_V$ and such that 
$\bar{k} = k-1$. Example \ref{exmp:H1234} and Figure \ref{fig:hierarchy4Weak} 
describe (parts of) the weak hierarchy over $V = \{1,2,3,4 \}$.

\paragraph{Properties of $\mathcal{H}_V$} We first argue that $\mathcal{H}_V$ 
is a forest. The nodes $(\mathcal{G}_{0,i}, 0)$, 
$i 
= 1,\ldots,n_0$, must be in different connected components as for each node 
there is at most a single edge down in the hierarchy. Using induction on $k$ 
and 
Corollary \ref{cor:wellorder}, we see that if $\mathcal{G}_{k,i}\in 
[\mathcal{G}_{0,j}]_0$, then there is a path between $(\mathcal{G}_{k,i},k)$ 
and $(\mathcal{G}_{0,j}, 0)$, and $V_j = 
\{(\mathcal{G}_{k,i}, k):  \mathcal{G}_{k,i}\in [\mathcal{G}_{0,j}]_0\}$ is 
therefore a connected subset of $V_{\mathrm{g}}$. It contains exactly $\vert 
V_j\vert - 1$ 
edges and is 
thus a tree. This means that 
$\mathcal{H}_V$
consists of $n_0$ 
disjoint trees, each tree rooted at $\mathcal{G}_{0,j}$ for some $j = 
1,2,\ldots,n_0$. Corollary \ref{cor:trekEq} characterizes $0$-weak equivalence.

When $i_1\neq 
i_2$, $[\mathcal{G}]_{k_1,i_1}$ and 
$[\mathcal{G}]_{k_2,i_2}$ 
are disjoint 
when $k_1=k_2$, but need not be when $k_1\neq k_2$. 
For $k_2\geq 
k_1$ and $i_1 =1,\ldots,n_{k_1}$, there exist $i_2$ such that 
$\mathcal{G}_{k_1,i_1} = 
\mathcal{G}_{k_2,i_2}$ which is due to the fact that if a graph is 
$k_1$-maximal, then it is also $k_2$-maximal (Corollary 
\ref{cor:kMaxMono}). The leaves of the trees are the 
greatest elements of the Markov equivalence 
classes (Proposition \ref{prop:nn1Equi}).

The graph $\mathcal{H}_V$ represents the entire system of $k$-weak equivalence 
classes and 
can be conveniently drawn in levels such that the vertical placement is 
determined by $k$ (see Figure \ref{fig:hierarchy4Weak}). 
Let $[\mathcal{G}_{k,i}]_k$ be a $k$-weak equivalence class represented by its 
greatest element $\mathcal{G}_{k,i}$. If we 
move along the unique edge towards 
a 
$(k-1)$-maximal graph, we 
obtain the maximal element of the $(k-1)$-weak equivalence class containing 
graph the $\mathcal{G}_{k,i}$ by definition of $\mathcal{H}_V$. If we move to 
the $(k+1)$-level, one of the 
$(k+1)$-equivalence classes will be represented by 
$\mathcal{G}_{k,i}$ itself. Naturally, moving towards larger $k$ in the 
hierarchy achieves 
smaller equivalence classes as if $\mathcal{G}$ is $(k-1)$-maximal, then 
$[\mathcal{G}]_{k-1} = 
\bigcup_{\bar{\mathcal{G}} \in \mathrm{up}(\mathcal{G}, k-1)} 
[\bar{\mathcal{G}}]_{k}$.

\paragraph{Dashed edges in the hierarchy}
In $\mathcal{H}_V$, one may use DMEGs instead of the corresponding maximal 
DMGs, and in this paragraph we think of a node $(\mathcal{G},k)$ in 
$\mathcal{H}_V$ as a pair consisting of a DMEG and an integer. In this case, 
there is also a certain structure in the dashed/solid 
status of edges across levels of $k$. If an edge $\alpha\sim\beta$ is solid in 
$(\mathcal{G},k)$, then it is 
also solid in all graphs $\bar{\mathcal{G}}\in\mathrm{up}(\mathcal{G},k)$. 
This is seen from the fact that if $\tilde{\mathcal{G}} \in 
[\bar{\mathcal{G}}]_{k+1}$ then $\tilde{\mathcal{G}}  \in [\mathcal{G}]_k$ and 
every graph in this equivalence class contains $\alpha\sim\beta$ which is why 
all graphs in $[\bar{\mathcal{G}}]_{k+1}$ also contain 
it. If the edge $\alpha\sim\beta$ is dashed in $(\mathcal{G},k)$, then it is 
also dashed in $\bar{\mathcal{G}} = \mathrm{down}(\mathcal{G},k-1)$. This is 
because 
there exists a graph $\tilde{\mathcal{G}}\in [{\mathcal{G}}]_{k}$ without 
this edge, and $\tilde{\mathcal{G}} \in [\bar{\mathcal{G}}]_{k-1}$. On 
the other hand, the edge is in the maximal element of $[{\mathcal{G}}]_{k}$, 
thus the edge must be present in $\bar{\mathcal{G}}$ and dashed. 

On the other hand, moving up (towards larger values of $k$) in the hierarchy a 
dashed edge may be removed, 
become solid, or remain dashed. Moving down (towards smaller values of $k$) in 
the hierarchy a solid edge may 
become dashed.

\begin{exmp}[$k$-weak hierarchy over $V=\{1,2,3,4\}$]
Figure \ref{fig:hierarchy4Weak} shows a subgraph of $\mathcal{H}_V$ for $V = 
\{1,2,3,4\}$. A node in $\mathcal{H}_V$, $(\mathcal{G},k)$, is shown as 
$\mathcal{G}$ (or rather, the corresponding DMEG), and $k$ determines the 
vertical placement of the node. All loops are present in the maximal graphs, 
but omitted from the visualization for 
simplicity. We use the edge $\alpha\allEdges\beta$ to indicate that all 
three possible edges between a pair of nodes, $\alpha$ and 
$\beta$, are present in the graph, that is, 
$\alpha\rightarrow\beta,\alpha\leftrightarrow\beta,\alpha\leftrightarrow\beta$. 
The letters $(x,y)$, to the right of a graph index the graphs 
shown in the figure.

Figure \ref{fig:hierarchy4Weak}
shows two subtrees of trees in the hierarchy. We see that the two graphs shown 
on level 
$k=0$, $(a,a)$ and $(e,a)$, are 
not $0$-weak equivalent as there is no directed trek from $1$ to $2$ in $(e,a)$ 
(see also Corollary \ref{cor:trekEq}).

In the figure, a red undirected edge indicates graph equality, for example, 
the edge between $(a,c)$ and $(a,d)$. As noted above, 
if 
$\mathcal{G}$ is $k$-maximal, then $\mathcal{G} \in \mathrm{up}(\mathcal{G},k)$ 
and when $\mathcal{G}$ is drawn both in levels $k$ and $k+1$, we indicate this 
by making the undirected edge connecting them red.

\label{exmp:H1234}
\end{exmp}

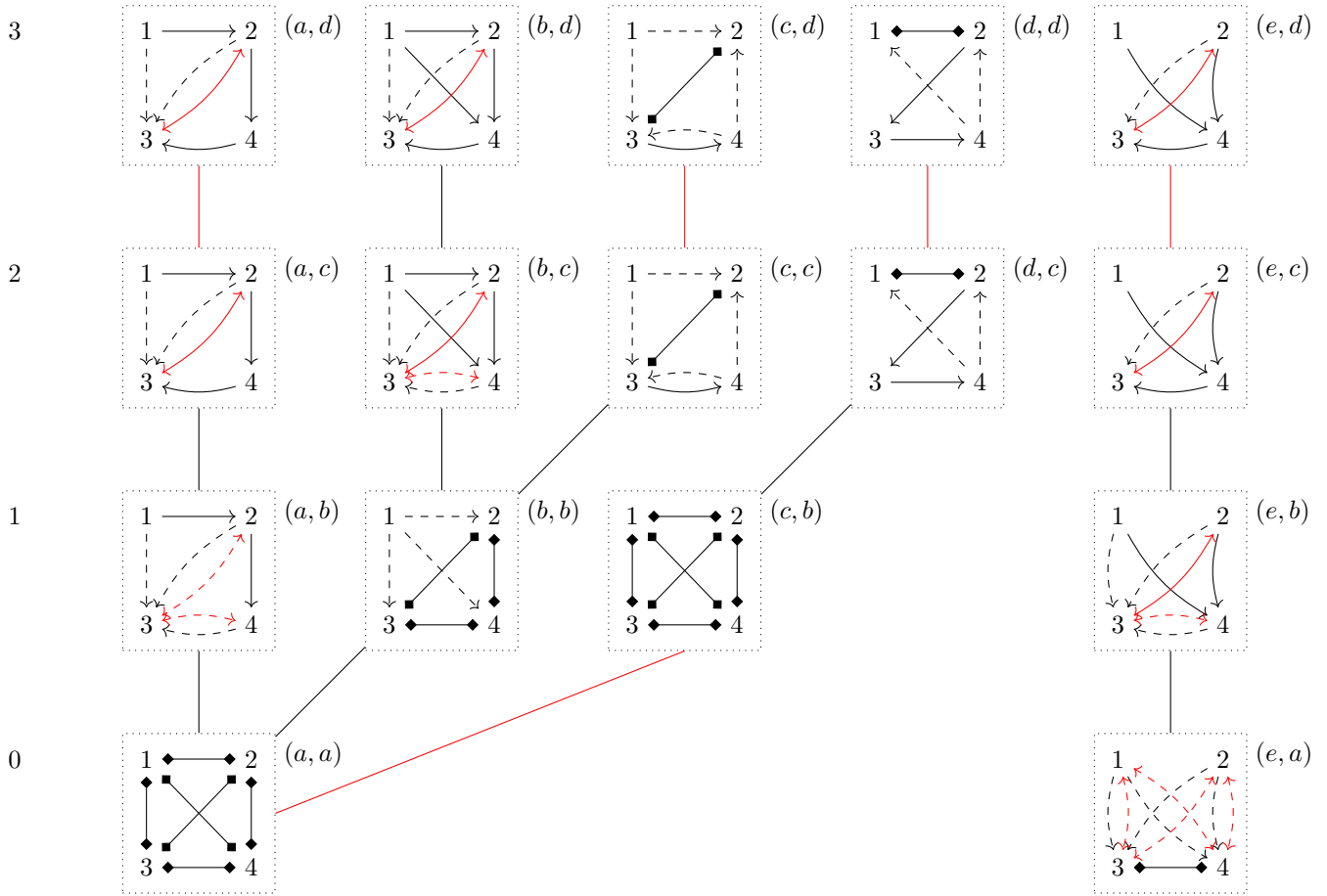
\begin{figure}
	\makebox[\textwidth][c]{%
	\begin{tikzpicture}
	\def\xshift{33mm}\def\yshift{33mm}\def\xint{20mm}
	\node[] (initial) at (0,0) {0};
	\node[] (k1) at (0,\yshift) {1};
	\node[] (k2) at (0,2*\yshift) {2};
	\node[] (k3) at (0,3*\yshift) {3};
	\begin{scope}[local bounding box=B00, shift=(initial.west), 
	xshift=\xint + 0*\xshift, 
	yshift = 0]
	\node (a) {$1$};
	\node (b) [right=of a] {$2$};
	\node (c) [below=of a] {$3$};
	\node (d) [below=of b] {$4$};
	\node[draw,dotted,fit=(B00)] (B00) {};
    \node[anchor=north west,inner sep=3pt] at (B00.north east) 
    {$(a,a)$};
	\begin{scope}[every node/.style={scale=.7}]
	\path[+-+]
	(a) edge (b)
	(a) edge (c)
	(a) edge (d)
	(b) edge (c)
	(b) edge (d)
	(c) edge (d);
	\end{scope}
	\end{scope}
		
	\begin{scope}[local bounding box=B05, shift=(initial.west), 
	xshift=\xint + 4*\xshift, 
	yshift = 0]
	\node (a) {$1$};
	\node (b) [right=of a] {$2$};
	\node (c) [below=of a] {$3$};
	\node (d) [below=of b] {$4$};
	\node[draw,dotted,fit=(B05)] (B05) {};
    \node[anchor=north west,inner sep=3pt] at (B05.north east) 
    {$(e,a)$};
	\begin{scope}[every node/.style={scale=.7}]
	\path[->] 
	(a) edge [bend right = 15, dashed]  (c)    
	(a) edge [bend right = 15, dashed]  (d)  
	(b) edge [bend right = 15, dashed]  (c)
	(b) edge [bend right = 15, dashed]  (d);
	\path[<->, red] 
	(a) edge [bend right = -15, dashed]  (c) 
	(a) edge [bend right = -15, dashed]  (d)     
	(b) edge [bend right = -15, dashed]  (c)
	(b) edge [bend right = -15, dashed]  (d);
	\path[+-+] 
	(c) edge  (d);
	\end{scope}
	\end{scope}

	\begin{scope}[local bounding box=B10, shift=(initial.west), 
	xshift=\xint + 0*\xshift, yshift = \yshift]
	\node (a) {$1$};
	\node (b) [right=of a] {$2$};
	\node (c) [below=of a] {$3$};
	\node (d) [below=of b] {$4$};
	\node[draw,dotted,fit=(B10)] (B10) {};
    \node[anchor=north west,inner sep=3pt] at (B10.north east) 
    {$(a,b)$};
	\begin{scope}[every node/.style={scale=.7}]
	\path[->] 
	(a) edge  (b)    
	(a) edge [dashed]  (c)
	(b) edge [bend right = 15, dashed]  (c)
	(d) edge [bend left = 15, dashed]  (c)
	(b) edge  (d);
	\path[<->, red]
	(b) [bend left = 15, dashed] edge (c)
	(c) [bend left = 15, dashed] edge (d);
	\end{scope}
	\end{scope}
	
	\begin{scope}[local bounding box=B11, shift=(initial.west), 
	xshift=\xint + 1*\xshift, yshift = \yshift]
	\node (a) {$1$};
	\node (b) [right=of a] {$2$};
	\node (c) [below=of a] {$3$};
	\node (d) [below=of b] {$4$};
	\node[draw,dotted,fit=(B11)] (B11) {};
	\node[anchor=north west,inner sep=3pt] at (B11.north east) 
	{$(b,b)$};
	\begin{scope}[every node/.style={scale=.7}]
	\path[->] 
	(a) edge [dashed] (b)
	(a) edge [dashed]  (c) 
	(a) edge [dashed]  (d);
	\path[+-+]
	(b) edge (c)
	(b) edge (d)
	(c) edge (d);
	\end{scope}
	\end{scope}
	
	\begin{scope}[local bounding box=B12, shift=(initial.west), 
	xshift=\xint + 2*\xshift, 
	yshift = \yshift]
	\node (a) {$1$};
	\node (b) [right=of a] {$2$};
	\node (c) [below=of a] {$3$};
	\node (d) [below=of b] {$4$};
	\node[draw,dotted,fit=(B12)] (B12) {};
	\node[anchor=north west,inner sep=3pt] at (B12.north east) 
	{$(c,b)$};
	\begin{scope}[every node/.style={scale=.7}]
	\path[+-+]
	(a) edge (b)
	(a) edge (c)
	(a) edge (d)
	(b) edge (c)
	(b) edge (d)
	(c) edge (d);
	\end{scope}
	\end{scope}

	\begin{scope}[local bounding box=B15, shift=(initial.west), 
	xshift=\xint + 4*\xshift, 
	yshift = \yshift]
	\node (a) {$1$};
	\node (b) [right=of a] {$2$};
	\node (c) [below=of a] {$3$};
	\node (d) [below=of b] {$4$};
	\node[draw,dotted,fit=(B15)] (B15) {};
	\node[anchor=north west,inner sep=3pt] at (B15.north east) 
	{$(e,b)$};
	\begin{scope}[every node/.style={scale=.7}]
	\path[->] 
	(a) edge [bend right = 15, dashed]  (c)    
	(a) edge [bend right = 15]  (d)  
	(b) edge [bend right = 15, dashed]  (c)
	(b) edge [bend right = 15]  (d)
	(d) edge [bend right = -15, dashed]  (c);
	\path[<->, red]    
	(b) edge [bend right = -15]  (c)
	(c) edge [bend right = -15, dashed]  (d);
	\end{scope}
	\end{scope}

	\begin{scope}[local bounding box=B20, shift=(initial.west), 
	xshift=\xint + 0*\xshift, yshift = 2*\yshift]
	\node (a) {$1$};
	\node (b) [right=of a] {$2$};
	\node (c) [below=of a] {$3$};
	\node (d) [below=of b] {$4$};
	\node[draw,dotted,fit=(B20)] (B20) {};
    \node[anchor=north west,inner sep=3pt] at (B20.north east) 
    {$(a,c)$};
	\begin{scope}[every node/.style={scale=.7}]
	\path[->] 
	(a) edge  (b)    
	(a) edge [dashed]  (c)
	(b) edge [bend right = 15, dashed]  (c)
	(d) edge [bend left = 15]  (c)
	(b) edge  (d);
	\path[<->, red]
	(b) [bend left = 15] edge (c);
	\end{scope}
	\end{scope}
	
	\begin{scope}[local bounding box=B21, shift=(initial.west), 
	xshift=\xint + 1*\xshift, yshift = 2*\yshift]
	\node (a) {$1$};
	\node (b) [right=of a] {$2$};
	\node (c) [below=of a] {$3$};
	\node (d) [below=of b] {$4$};
	\node[draw,dotted,fit=(B21)] (B21) {};
	\node[anchor=north west,inner sep=3pt] at (B21.north east) 
	{$(b,c)$};
	\begin{scope}[every node/.style={scale=.7}]
	\path[->] 
	(a) edge  (b)    
	(a) edge [dashed]  (c)
	(a) edge  (d)
	(b) edge [bend right = 15, dashed] (c)
	(b) edge (d)
	(d) edge [bend left = 15, dashed] (c);
	\path[<->, red]
	(b) edge [bend right = -15]  (c)
	(d) edge [bend left = -15, dashed]  (c);
	\end{scope}
	\end{scope}

	\begin{scope}[local bounding box=B22, shift=(initial.west), 
	xshift=\xint + 2*\xshift, yshift = 2*\yshift]
	\node (a) {$1$};
	\node (b) [right=of a] {$2$};
	\node (c) [below=of a] {$3$};
	\node (d) [below=of b] {$4$};
	\node[draw,dotted,fit=(B22)] (B22) {};
	\node[anchor=north west,inner sep=3pt] at (B22.north east) 
	{$(c,c)$};
	\begin{scope}[every node/.style={scale=.7}]
	\path[->] 
	(a) edge [dashed]  (b)    
	(a) edge [dashed] (c)
	(d) edge [dashed] (b)
	(d) edge [bend right = 15, dashed] (c)
	(c) edge [bend right = 15] (d);
	\path[+-+]
	(b) edge [bend right = 0]  (c);
	\end{scope}
	\end{scope}
	
	\begin{scope}[local bounding box=B23, shift=(initial.west), 
	xshift=\xint + 3*\xshift, yshift = 2*\yshift]
	\node (a) {$1$};
	\node (b) [right=of a] {$2$};
	\node (c) [below=of a] {$3$};
	\node (d) [below=of b] {$4$};
	\node[draw,dotted,fit=(B23)] (B23) {};
	\node[anchor=north west,inner sep=3pt] at (B23.north east) 
	{$(d,c)$};
	\begin{scope}[every node/.style={scale=.7}]
	\path[->] 
	(b) edge  (c)    
	(c) edge  (d)
	(d) edge [dashed] (b)
	(d) edge [bend right = 0, dashed] (a);
	\path[+-+]
	(a) edge [bend right = 0]  (b);
	\end{scope}
	\end{scope}
	
	\begin{scope}[local bounding box=B25, shift=(initial.west), 
	xshift=\xint + 4*\xshift, 
	yshift = 2*\yshift]
	\node (a) {$1$};
	\node (b) [right=of a] {$2$};
	\node (c) [below=of a] {$3$};
	\node (d) [below=of b] {$4$};
	\node[draw,dotted,fit=(B25)] (B25) {};
	\node[anchor=north west,inner sep=3pt] at (B25.north east) 
	{$(e,c)$};
	\begin{scope}[every node/.style={scale=.7}]
	\path[->]    
	(a) edge [bend right = 15]  (d)  
	(b) edge [bend right = 15, dashed]  (c)
	(b) edge [bend right = 15]  (d)
	(d) edge [bend right = -15]  (c);
	\path[<->, red]    
	(b) edge [bend right = -15]  (c);
	\end{scope}
	\end{scope}
	
	\begin{scope}[local bounding box=B30, shift=(initial.west), 
	xshift=\xint + 0*\xshift, yshift = 3*\yshift]
	\node (a) {$1$};
	\node (b) [right=of a] {$2$};
	\node (c) [below=of a] {$3$};
	\node (d) [below=of b] {$4$};
	\node[draw,dotted,fit=(B30)] (B30) {};
    \node[anchor=north west,inner sep=3pt] at (B30.north east) 
    {$(a,d)$};
	\begin{scope}[every node/.style={scale=.7}]
	\path[->] 
	(a) edge  (b)    
	(a) edge [dashed] (c)
	(b) edge [bend right = 15, dashed]  (c)
	(d) edge [bend left = 15]  (c)
	(b) edge  (d);
	\path[<->, red]
	(b) [bend left = 15] edge (c);
	\end{scope}
	\end{scope}

	\begin{scope}[local bounding box=B31, shift=(initial.west), 
	xshift=\xint + 1*\xshift, yshift = 3*\yshift]
	\node (a) {$1$};
	\node (b) [right=of a] {$2$};
	\node (c) [below=of a] {$3$};
	\node (d) [below=of b] {$4$};
	\node[draw,dotted,fit=(B31)] (B31) {};
	\node[anchor=north west,inner sep=3pt] at (B31.north east) 
	{$(b,d)$};
	\begin{scope}[every node/.style={scale=.7}]
	\path[->] 
	(a) edge  (b)    
	(a) edge [dashed]  (c)
	(a) edge  (d)
	(b) edge [bend right = 15, dashed] (c)
	(b) edge (d)
	(d) edge [bend left = 15] (c);
	\path[<->, red]
	(b) edge [bend right = -15]  (c);
	\end{scope}
	\end{scope}
	
	\begin{scope}[local bounding box=B32, shift=(initial.west), 
	xshift=\xint + 2*\xshift, yshift = 3*\yshift]
	\node (a) {$1$};
	\node (b) [right=of a] {$2$};
	\node (c) [below=of a] {$3$};
	\node (d) [below=of b] {$4$};
	\node[draw,dotted,fit=(B32)] (B32) {};
	\node[anchor=north west,inner sep=3pt] at (B32.north east) 
	{$(c,d)$};
	\begin{scope}[every node/.style={scale=.7}]
	\path[->] 
	(a) edge [dashed] (b)    
	(a) edge [dashed] (c)
	(d) edge [dashed] (b)
	(d) edge [bend right = 15, dashed] (c)
	(c) edge [bend right = 15] (d);
	\path[+-+]
	(b) edge [bend right = 0]  (c);
	\end{scope}
	\end{scope}
	
	\begin{scope}[local bounding box=B33, shift=(initial.west), 
	xshift=\xint + 3*\xshift, yshift = 3*\yshift]
	\node (a) {$1$};
	\node (b) [right=of a] {$2$};
	\node (c) [below=of a] {$3$};
	\node (d) [below=of b] {$4$};
	\node[draw,dotted,fit=(B33)] (B33) {};
	\node[anchor=north west,inner sep=3pt] at (B33.north east) 
	{$(d,d)$};
	\begin{scope}[every node/.style={scale=.7}]
	\path[->] 
	(b) edge  (c)    
	(c) edge  (d)
	(d) edge [dashed] (b)
	(d) edge [bend right = 0, dashed] (a);
	\path[+-+]
	(a) edge [bend right = 0]  (b);
	\end{scope}
	\end{scope}
	
	\begin{scope}[local bounding box=B35, shift=(initial.west), 
	xshift=\xint + 4*\xshift, 
	yshift = 3*\yshift]
	\node (a) {$1$};
	\node (b) [right=of a] {$2$};
	\node (c) [below=of a] {$3$};
	\node (d) [below=of b] {$4$};
	\node[draw,dotted,fit=(B35)] (B35) {};
	\node[anchor=north west,inner sep=3pt] at (B35.north east) 
	{$(e,d)$};
	\begin{scope}[every node/.style={scale=.7}]
	\path[->]    
	(a) edge [bend right = 15]  (d)  
	(b) edge [bend right = 15, dashed]  (c)
	(b) edge [bend right = 15]  (d)
	(d) edge [bend right = -15]  (c);
	\path[<->, red]    
	(b) edge [bend right = -15]  (c);
	\end{scope}
	\end{scope}

	\begin{scope}[every node/.style={scale=.7}]
	\path[-] 
	(B00) edge  (B10)
	(B00) edge  (B11)
	(B00.east) edge [red] (B12.south)
	(B11) edge (B21)
	(B11) edge (B22)
	(B10) edge  (B20)
	(B12) edge (B23)
	(B23) edge [red] (B33)
	(B20) edge [red]  (B30)
	(B21) edge  (B31)
	(B22) edge [red] (B32)
	(B05) edge  (B15)
	(B15) edge  (B25)
	(B25) edge [red]  (B35);
	\end{scope}
	\end{tikzpicture}
}
	\caption{Subtrees from the $k$-weak equivalence hierarchy on $V = \{1,2,3,4 
	\}$. For simplicity, we use $\alpha\allEdges\beta$ to indicate that 
	$\alpha\rightarrow\beta;\alpha\leftarrow\beta;\alpha\leftrightarrow\beta$ 
	are all present in the graph. We use $\alpha\allEdges\beta$ regardless of 
	whether some or all of these underlying edges are dashed. In this figure, 
	bidirected edges are red for better legibility.}
	\label{fig:hierarchy4Weak}

\end{figure}

\subsection{Alarm network}
\label{ssec:alarm2}

We return to the alarm example from Subsection \ref{ssec:alarm1}. This is a 
network of moderate size with $10$ observable coordinate processes. If we 
consider graphical modeling of this network using a $k$-weak equivalence 
relation, different values of $k \in \{0,1,\ldots,10\}$ lead to different 
levels 
of granularity as larger values of $k$ will give us smaller equivalence 
classes. Let $\mathcal{G}$ denote the latent projection of the system (see 
Figure \ref{fig:alarm1}), and let $\mathcal{N}_k$ denote the greatest element 
of $[\mathcal{G}]_k$. Figure \ref{fig:alarm2} shows the DMEGs of 
$\mathcal{N}_{10}$ and of $\mathcal{N}_3$. We know that 
$\mathcal{N}_{10}\subseteq \mathcal{N}_3$. In this example, we see that the 
only difference between the two DMEGs in Figure \ref{fig:alarm2} is the 
bidirected edge between $3$ and $10$. This edge is necessarily dashed as 
$\mathcal{N}_{10} \in [\mathcal{G}]_3 = [\mathcal{N}_3]_3$. The added 
complexity of using $k=10$ does therefore not provide much additional 
information in this example.

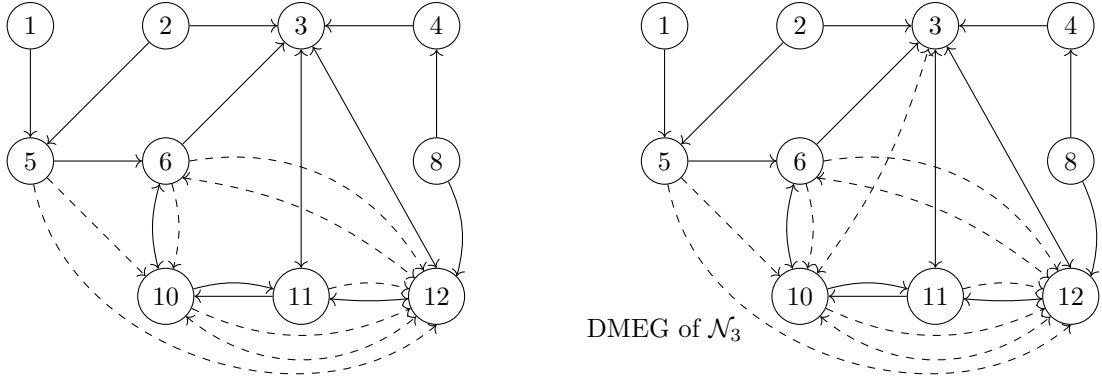
\begin{figure}[h]
	\begin{subfigure}{.47\textwidth}
		\centering
		\begin{tikzpicture}[scale = .9]
		\node[shape=circle,draw=black] (a) at (0,0) {$1$};
		\node[shape=circle,draw=black] (b) at (2,0) 
		{$2$};
		\node[shape=circle,draw=black] (c) at (4,0) {$3$};
		\node[shape=circle,draw=black] (d) at (6,0) {$4$};
		\node[shape=circle,draw=black] (e) at (0,-2) {$5$};
		\node[shape=circle,draw=black] (f) at (2,-2) 
		{$6$};
		\node[shape=circle,draw=black] (h) at (6,-2) {$8$};
		\node[shape=circle,draw=black] (j) at (2,-4) 
		{$10$};
		\node[shape=circle,draw=black] (k) at (4,-4) {$11$};
		\node[shape=circle,draw=black] (l) at (6,-4) {$12$};

		\path 
		[->](a) edge [bend left = 0] node {} (e)
		[->](b) edge [bend left = 0] node {} (e)
		[->](b) edge [bend left = 0] node {} (c)
		[->](d) edge [bend left = 0] node {} (c)
		[->](e) edge [bend left = 0] node {} (f)
		[->](e) edge [bend left = 0, dashed] node {} (j)
		[->](e) edge [bend left = -60, dashed] node {} (l.south)
		[->](f) edge [bend left = 0] node {} (c)
		[->](h) edge [bend left = 0] node {} (d)
		[->](h) edge [bend left = 20] node {} (l.north east)
		[->](j) edge [bend left = 15] node {} (k)
		[->](k) edge [bend left = 0] node {} (j)
		[->](f.east) edge [bend left = 35, dashed] node {} (l)
		[->](f) edge [bend left = 15, dashed] node {} (j)
		[->](j) edge [bend left = -25, dashed] node {} (l)
		[->](k) edge [bend left = 15, dashed] node {} (l);
		
		\path
		[<->](c) edge [bend left = 0] node {} (k)
		[<->](c) edge [bend left = 0] node {} (l.north)
		[<->](f) edge [bend left = -15] node {} (j)
		[<->](f.south east) edge [bend left = 15, dashed] node {} (l)
		[<->](j) edge [bend left = -40, dashed] node {} (l.south west)
		[<->](k) edge [bend left = -5] node {} (l);
		\end{tikzpicture}
	\end{subfigure}\hfill
	\begin{subfigure}{.47\textwidth}
		\centering
		\begin{tikzpicture}[scale = .9]
		\node[shape=circle,draw=black] (a) at (0,0) {$1$};
		\node[shape=circle,draw=black] (b) at (2,0) 
		{$2$};
		\node[shape=circle,draw=black] (c) at (4,0) {$3$};
		\node[shape=circle,draw=black] (d) at (6,0) {$4$};
		\node[shape=circle,draw=black] (e) at (0,-2) {$5$};
		\node[shape=circle,draw=black] (f) at (2,-2) 
		{$6$};
		\node[shape=circle,draw=black] (h) at (6,-2) {$8$};
		\node[shape=circle,draw=black] (j) at (2,-4) 
		{$10$};
		\node[shape=circle,draw=black] (k) at (4,-4) {$11$};
		\node[shape=circle,draw=black] (l) at (6,-4) {$12$};
		
		\node[shape=rectangle,draw=white] (x) at (0,-4.5) {DMEG of 
			$\mathcal{N}_3$};

		\path 
		[->](a) edge [bend left = 0] node {} (e)
		[->](b) edge [bend left = 0] node {} (e)
		[->](b) edge [bend left = 0] node {} (c)
		[->](d) edge [bend left = 0] node {} (c)
		[->](e) edge [bend left = 0] node {} (f)
		[->](e) edge [bend left = 0, dashed] node {} (j)
		[->](e) edge [bend left = -60, dashed] node {} (l.south)
		[->](f) edge [bend left = 0] node {} (c)
		[->](h) edge [bend left = 0] node {} (d)
		[->](h) edge [bend left = 20] node {} (l.north east)
		[->](j) edge [bend left = 15] node {} (k)
		[->](k) edge [bend left = 0] node {} (j)
		[->](f.east) edge [bend left = 35, dashed] node {} (l)
		[->](f) edge [bend left = 15, dashed] node {} (j)
		[->](j) edge [bend left = -25, dashed] node {} (l)
		[->](k) edge [bend left = 15, dashed] node {} (l);
		
		\path
		[<->](c) edge [bend left = 0] node {} (k)
		[<->](c) edge [bend left = 0] node {} (l.north)
		[<->](f) edge [bend left = -15] node {} (j)
		[<->](f.south east) edge [bend left = 15, dashed] node {} (l)
		[<->](j) edge [bend left = -40, dashed] node {} (l.south west)
		[<->](k) edge [bend left = -5] node {} (l)
		[<->](j) edge [bend left = -10, dashed] node {} (c);
		\end{tikzpicture}
	\end{subfigure}
	\caption{Graphs from 
		Subsection \ref{ssec:alarm2}. Loops are omitted from the 
		visualization. Left: 
		\emph{directed mixed equivalence graph} of 
		$\mathcal{N}_{10}$ which is the greatest element of
		$[\mathcal{G}]_k$ for $k = 
		4,5,\ldots,10$. Right: \emph{directed 
			mixed equivalence 
			graph} of $\mathcal{N}_{3}$ which is the greatest element 
		of $[\mathcal{G}]_k$ for $k = 2,3$. The only 
		difference between the two DMEGs is the dashed, bidirected 
		edge betweeen $2$ and 
		$3$. In $\mathcal{N}_{3}$, we see that $5\rightarrow 10 
		\leftrightarrow 3$ is a $\mu$-connecting walk from $5$ to 
		$3$ given $\{ 2,6,10,11 \}$. In $\mathcal{N}_{10}$, there 
		is 
		no 
		such connecting walk, and this illustrates that 
		$\mathcal{N}_{10}$ and $\mathcal{N}_3$ are not $4$-weakly 
		equivalent.}
	\label{fig:alarm2}
\end{figure} 

\section{Algorithms for weak equivalence}
\label{sec:algo}

The results in Section \ref{sec:hard} imply that several computational tasks 
that occur naturally when using $\mu$-separation and local independence for 
graphical modeling of stochastic processes are not feasible, even for a 
moderate number of coordinate processes. Section \ref{sec:we} introduces a more 
flexible notion of equivalence to circumvent these issues and Section 
\ref{sec:greatElemWEqui} shows that the convenient theory of Markov equivalence 
classes translates seamlessly to the more general notion of weak equivalence. 
As a last component of this paper, we argue that this more general theory leads 
to algorithms that are in fact feasible from a computational point of view.

\subsection{A parametrized hierarchy of graphical equivalence}

We start this subsection by providing a formal definition of the weak 
equivalence decision problem.

\begin{dec}[Weak Markov equivalence in DMGs]
	Let $\mathcal{G}_1 = (V,E)_1$ and $\mathcal{G}_2 = (V,E_2)$ be DMGs. Are 
	$\mathcal{G}_1$ and $\mathcal{G}_2$ $\mathcal{J}$-weakly equivalent?
	\label{dc:wME}
\end{dec}

Decision problem \ref{dc:wME} is coNP-complete 
as it is a more general problem 
than Decision problem \ref{dc:ME}. We restrict this to $k$-weak equivalence and 
obtain a \emph{parametrized} decision problem.

\begin{dec}[Weak Markov equivalence in DMGs]
	Let $k$ be a nonnegative integer, and let $\mathcal{G}_1 = (V,E)_1$ and 
	$\mathcal{G}_2 = (V,E_2)$ be DMGs. Are 
	$\mathcal{G}_1$ and $\mathcal{G}_2$ $k$-weakly equivalent?
	\label{dc:kME}
\end{dec}

A decision problem is said to be \emph{slicewise polynomial} if there 
exists an algorithm which solves the problem in $\mathcal{O}(n^{g(k)})$ steps 
for a computable function $g$, input length $n$, and parameter $k$. For fixed 
$k$, we can decide $k$-weak 
equivalence of two DMGs by simply 
checking every possible triple $(\alpha,\beta,C)$, $\alpha,\beta\in 
V,C\subseteq V$. This can be done in time bounded by $n^{g(k)}$ as the number 
of 
conditioning sets is bounded by $n^k$. This 
shows that parametrized $k$-weak equivalence is a \emph{slicewise polynomial} 
problem, in that for a fixed $k$ it is solvable by an algorithm which is 
polynomial in $n$. One should note that this is different from the $m$-sparse 
decision problems (e.g, Decision problem \ref{dc:MEsparse}) as they remain hard 
for a fixed $m$ whenever $m \geq 16$.

Intuitively, the unrestricted 
Markov equivalence problem is computationally hard as the maximal size of the 
conditioning sets also grows with $n$. On the other hand, if we consider 
$k$-weak equivalence for a fixed $k$ then the maximal size of the conditioning 
sets 
is fixed, and the problem can be solved in time which scales polynomially in 
$n$.

\subsection{Computing greatest elements and directed mixed equivalence graphs}

As explained above, for a fixed $k$ one can decide $k$-weak equivalence in 
polynomial time. The same applies to the related computational problems. 

Assume we have a graph $\mathcal{G}$ and want to find the maximal element of 
$[\mathcal{G}]_k$. A simple algorithm checks for each edge if its addition 
violates any of the independences in $[\mathcal{G}]_k$ and adds the edge if and 
only if this is not the case. For a fixed $k$, this is done in polynomial time.

When considering a weak equivalence class as represented by its greatest 
element, we are interested in computing the associated directed mixed 
equivalence graph (DMEG) as this graph represents the entire equivalence class 
concisely. We may remove a single edge at a time and decide Markov equivalence 
to obtain the corresponding DMEG from a greatest element.

\section{Learning}
\label{sec:learn}

There is a large literature on methods for recovering a graph from 
observational data \cite{spirtes2018search}. In the case of DAG-based models, 
many methods use 
tests of conditional independence. Similarly, it 
is possible to learn local 
independence graphs using tests of local 
independence. In this section, we briefly discuss graphical structure learning 
based on tests of 
local independence as described by \cite{meek2014toward} and its connection to 
weak equivalence of DMGs. 
\cite{mogensenUAI2018} described a learning algorithm outputting the Markov 
equivalence DMEG from 
tests of local independence. \cite{absar2021discovering} implemented a PC-like 
algorithm based on $\mu$-separation. \cite{bhattacharjya2022process} 
studied independence tests in proximal graphical event models and graphical 
structure learning based on tests of local independences. Other work 
described tests of local 
independence (\cite{thams2021local} and \cite{christgau2022nonparametric}) and 
good tests 
are of course a 
prerequisite for constraint-based structure learning. The learning problem has 
also 
been studied in the discrete-time processes \citep{eichler2013}.

As argued in previous sections, constrained-based algorithms that learn the 
Markov equivalence class of a 
partially observed local independence graph and
are correct in the oracle setting scale poorly with the size of the graph. 
Therefore, $k$-weak 
equivalence classes may constitute more reasonable targets for graphical 
structure learning. The oracle learning algorithm in \cite{mogensenUAI2018} 
leveraged the potential sibling and potential parent criteria to ensure 
correctness, though the number of these conditions also scales poorly with 
graph size, $n$. This naturally leads to the idea of using $C$-potential 
sibling and $C$-potential parent criteria 
directly for learning. In the oracle case this leads to a straightforward 
learning 
algorithm by starting from the complete DMG. For each pair of nodes, 
$(\alpha,\beta)$, one may test the $C$-potential parent criteria for all $\vert 
C\vert \leq k$. If one of these criteria is violated, one simply removes 
$\alpha\rightarrow\beta$, and similarly for the bidirected edges. For fixed 
$k$, this leads to a polynomial-time oracle learning algorithm which outputs 
the maximal $k$-weakly equivalent graph of the true graph. This is similar to 
early stopping in FCI \citep{spirtes2001anytimeFCI} as it only uses tests with 
small 
conditioning sets $C$. While smaller values of $k$ lead 
to less informative output (larger equivalence classes), the interpretation of 
a learned DMEG 
remains the same as when using $k=n$ as shown by the theory in 
previous sections.

Outside of the oracle setting, actual 
tests of local independence output a $p$-value. When learning local 
independence graphs, one may compute $p$-values from the local independence 
tests that comprise the $C$-potential parent/sibling criteria, $\vert C\vert$, 
and use these $p$-values to output a maximal graph which is in 
minimum violation with the data, see e.g. \cite{hyttinenASP} for a similar idea 
in 
DAG-based graphical structure learning. 

\section{Discussion}
\label{sec:discuss}

The results in Section \ref{sec:hard} show that deciding Markov equivalence is 
computationally hard, even under a sparsity constraint. This also implies that 
finding the unique maximal element of a Markov equivalence class is hard and 
that constraint-based learning algorithms 
that are correct in oracle versions need exponentially many tests in the worst 
case.

The theory developed in this paper provides a new interpretation of 
$\mu$-separation in directed mixed graphs as representations of local 
independence 
in partially observed stochastic processes. This leads to equivalence relations 
on directed mixed graphs that are weaker than Markov equivalence. Under a weak 
equivalence relation, each equivalence class of directed mixed 
graphs have a simple representation and interpretation using the existence of a 
greatest element. Importantly, they retain a clear interpretation and a 
convenient graphical 
representation of an entire $k$-weak equivalence class is available, just as in 
the case of Markov equivalence classes. The greatest element of an equivalence 
class also provides a 
feasible 
learning target, and 
one can give a constructive characterization of this element (the collection of 
$C$-potential sibling and $C$-potential parent conditions). The Markov 
equivalence class is often the learning target when trying to recover a 
graph from observational data, however, the complexity results in this paper 
imply that this target may be too expressive. The previous sections give the 
theoretical underpinning for feasible learning algorithms that output graphs 
that are less expressive than the Markov equivalence class.

A subset of the weak equivalence relations, $k$-weak equivalence relations, are 
naturally para\-metrized by a natural number $k$. Varying $k$, one obtains more 
or less granular graphical modeling, and a 
simple hierarchy of equivalence classes can be described across $k$. The 
parameter $k$ specifies both the granularity of the equivalence class and the 
complexity of, e.g., finding a 
maximal graph. The work in this paper mostly focused on the $k$-weak 
equivalence, however, 
the central 
results hold more general weak equivalences, and one may find applications of 
other types of equivalence relations, e.g., with inspiration from specific 
applications.

\section{Acknowledgments}

This work was funded by a DFF-International Postdoctoral Grant (0164-00023B) 
from
Independent Research Fund Denmark. The author is a member of the ELLIIT 
Strategic
Research Area at Lund University. The author thanks Karin Rathsman for 
discussing alarm handling at the European Spallation Source.


\bibliographystyle{plainnat}
\bibliography{C:/Users/swmo/Desktop/-/forsk/references}

\appendix

\section{Decision problems}

We list the formal decision problems used in Section \ref{sec:hard}.

\begin{dec}[Add-1 bidirected Markov equivalence in DMGs]
	Let $\mathcal{G}_1 = (V,E_1)$ and $\mathcal{G}_2 = (V,E_2)$ be DMGs such 
	that $E_2 = E_1 \cup \{ e\}$ and $e$ is bidirected edge. Are 
	$\mathcal{G}_1$ and $\mathcal{G}_2$ Markov equivalent?
	\label{dc:add1bidir}
\end{dec}

\begin{dec}[Add-1 directed Markov equivalence in DMGs]
	Let $\mathcal{G}_1 = (V,E_1)$ and $\mathcal{G}_2 = (V,E_2)$ be DMGs such 
	that $E_2 = E_1 \cup \{ e\}$ and $e$ is directed edge. Are 
	$\mathcal{G}_1$ and $\mathcal{G}_2$ Markov equivalent?
	\label{dc:add1dir}
\end{dec}

The next decision problems are sparse versions of Decision problems 
\ref{dc:add1bidir} and \ref{dc:add1dir}.

\begin{dec}[Add-1 birected Markov equivalence in sparse DMGs]
	Let $m$ be a nonnegative integer and let $\mathcal{G}_1 = (V,E_1)$ and 
	$\mathcal{G}_2 = (V,E_2)$ be $m$-sparse 
	DMGs such that $E_2 = E_1 \cup \{ e\}$ and $e$ is bidirected edge. Are 
	$\mathcal{G}_1$ and $\mathcal{G}_2$ Markov equivalent?
	\label{dc:add1bidirSparse}
\end{dec}

\begin{dec}[Add-1 directed Markov equivalence in sparse DMGs]
	Let $m$ be a nonnegative integer and let $\mathcal{G}_1 = (V,E_1)$ and 
	$\mathcal{G}_2 = (V,E_2)$ be $m$-sparse 
	DMGs such that $E_2 = E_1 \cup \{ e\}$ and $e$ is directed edge. Are 
	$\mathcal{G}_1$ and $\mathcal{G}_2$ Markov equivalent?
	\label{dc:add1dirSparse}
\end{dec}

\section{Node connectivity in DMGs}
\label{app:nodeConn}

In this section, we elaborate on the discussion in Subsection 
\ref{ssec:sparseDMGs} on different notions of node connectivity in a DMG. For a 
DMG, $\mathcal{G}=(V,E)$ and a node $\beta\in V$, we define $\beta$'s 
\emph{indegree}, $\id_\mathcal{G}(\beta)$, to be number of nodes, $\alpha\in 
V$, such that 
$\alpha\starrightarrow\beta$. Similarly, we define $\beta$'s \emph{outdegree}, 
$\od_\mathcal{G}(\beta)$, 
as the number of nodes, $\alpha\in V$, such that $\beta\starrightarrow\alpha$. 
This is an adaptation of the common definitions of in- and outdegree in DAGs. 
If 
$\alpha\starrightarrow\beta$ in $\mathcal{G}$, then $\alpha\in 
u(\beta,\mathcal{I}(\mathcal{G}))$, and it follows that the indegree of 
$\beta$ is less than or equal to
$\con_{\mathcal{G}_1}^{\scaleto{\rightarrow}{1.8pt}}(\beta)$. Similarly, the 
outdegree of $\beta$ is less than or equal to 
$\con_{\mathcal{G}_1}^{\scaleto{\leftarrow}{1.8pt}}(\beta)$. It holds that 
$\sum_{\beta\in V} \con_{\mathcal{G}_1}^{\scaleto{\rightarrow}{1.8pt}}(\beta) = 
\sum_{\beta\in V} \con_{\mathcal{G}_1}^{\scaleto{\leftarrow}{1.8pt}}(\beta)$. 
However, as illustrated in Figure \ref{fig:starDMG} it is possible for 
$\con_{\mathcal{G}_1}^{\scaleto{\rightarrow}{1.8pt}}(\beta)$ for some $\beta$ 
to be large while $\con_{\mathcal{G}_1}^{\scaleto{\leftarrow}{1.8pt}}(\alpha)$ 
is small for all $\alpha\in V$.

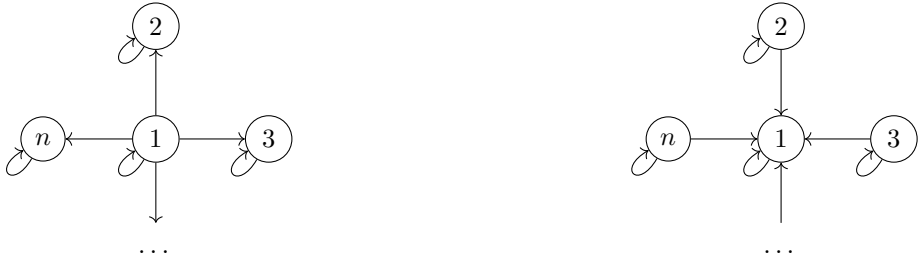
\begin{figure}[H]
	\def\x{1.5}
	\begin{subfigure}{.45\textwidth}
			\centering
	\begin{tikzpicture}
	\node[shape=circle,draw=black] (a) 
	at (0,0) {$1$};
	\node[shape=circle,draw=black] (b) 
	at (0,\x) 
	{$2$};
	\node[shape=circle,draw=black] (c) 
	at (\x,0) {$3$};
	\node[shape=circle,draw=white] (d) 
	at (0,-\x) {\ldots};
	\node[shape=circle,draw=black] (e) 
	at (-\x,0) {$n$};
	
	\path 
	[->](a) edge [bend left = 0] node {} (b)
	[->](a) edge [bend left = 0] node {} (c)
	[->](a) edge [bend left = 0] node {} (d)
	[->](a) edge [bend left = 0] node {} (e);
	\path
	[->](a) edge [loop above,  in=210,out=240, min distance=5mm] node {} (a)
	[->](b) edge [loop above,  in=210,out=240, min distance=5mm] node {} (b)
	[->](c) edge [loop above,  in=210,out=240, min distance=5mm] node {} (c)
	[->](e) edge [loop above,  in=210,out=240, min distance=5mm] node {} (e);
	\end{tikzpicture}
	\end{subfigure}\hfill
	\begin{subfigure}{.45\textwidth}
		\centering
		\begin{tikzpicture}
		\node[shape=circle,draw=black] (a) 
		at (0,0) {$1$};
		\node[shape=circle,draw=black] (b) 
		at (0,\x) 
		{$2$};
		\node[shape=circle,draw=black] (c) 
		at (\x,0) {$3$};
		\node[shape=circle,draw=white] (d) 
		at (0,-\x) {\ldots};
		\node[shape=circle,draw=black] (e) 
		at (-\x,0) {$n$};
		
		\path 
		[<-](a) edge [bend left = 0] node {} (b)
		[<-](a) edge [bend left = 0] node {} (c)
		[<-](a) edge [bend left = 0] node {} (d)
		[<-](a) edge [bend left = 0] node {} (e);
		\path
		[->](a) edge [loop above,  in=210,out=240, min distance=5mm] node {} (a)
		[->](b) edge [loop above,  in=210,out=240, min distance=5mm] node {} (b)
		[->](c) edge [loop above,  in=210,out=240, min distance=5mm] node {} (c)
		[->](e) edge [loop above,  in=210,out=240, min distance=5mm] node {} 
		(e);
		\end{tikzpicture}
	\end{subfigure}
	\caption{Graphs $\mathcal{G}_n^1$ and $\mathcal{G}_n^2$. In 
	$\mathcal{G}_n^1$ (left), 
	$\con_{\mathcal{G}_n^1}^{\scaleto{\leftarrow}{1.8pt}}(1) = 5$ while 
	$\con_{\mathcal{G}_n^1}^{\scaleto{\rightarrow}{1.8pt}}(\alpha) \leq 2$ for 
	all nodes $\alpha$. Similarly, in $\mathcal{G}_n^2$ above (right) 
	$\con_{\mathcal{G}_n^2}^{\scaleto{\rightarrow}{1.8pt}}(1) = 5$ while 
	$\con_{\mathcal{G}_n^2}^{\scaleto{\leftarrow}{1.8pt}}(\alpha) \leq 2$ for 
	all nodes $\alpha$. Therefore, we use a maximum over both measures of node 
	connectivity in Section \ref{ssec:sparseDMGs}.}
	\label{fig:starDMG}
\end{figure}

The indegree and outdegree of a node $\beta$ need not equal the 
$\con_{\mathcal{G}_1}^{\scaleto{\rightarrow}{1.8pt}}(\beta)$ and 
$\con_{\mathcal{G}_1}^{\scaleto{\leftarrow}{1.8pt}}(\beta)$, respectively (see 
the example in Figure \ref{fig:adjacDMG}). 
Moreover, the indegree and outdegree need not be the same for Markov equivalent 
graphs (Figure \ref{fig:adjacDMG}). 

The example in Figure 
\ref{fig:completeME} is exploiting non-maximality of the graph. In each Markov 
equivalence class, $[\mathcal{G}]$, there is 
a greatest element, 
$\mathcal{N}$ and one could 
define sparsity of the nodes in $\mathcal{G}$ by counting adjacencies 
in the $\mathcal{N}$ which is invariant under Markov 
equivalence. However, the in- and outdegree of $\beta$ in 
$\mathcal{N}$ may still be 
strictly less than $\con_{\mathcal{G}_1}^{\scaleto{\rightarrow}{1.8pt}}(\beta)$ 
and $\con_{\mathcal{G}_1}^{\scaleto{\leftarrow}{1.8pt}}(\beta)$, respectively 
(Figure \ref{fig:adjacDMG}). In fact, one can find a family of graphs, 
$\{\mathcal{G}_n = (V_n,E_n)\}$, and a node $\beta \in V_n$ for all $n$ such 
that $\con_{\mathcal{G}_1}^{\scaleto{\rightarrow}{1.8pt}}(\beta)$ is unbounded 
while the indegree and outdegree are fixed (see the example in Figure 
\ref{fig:adjacDMG3}).

If $\alpha$ is inseparable into $\beta$ and 
$\beta$ is inseparable into $\alpha$ in a maximal DMG, they need not be 
adjacent (see the example in Figure \ref{fig:adjacDMG2}).

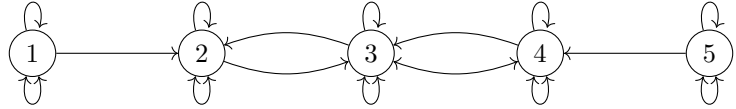
\begin{figure}[H]
	\centering
	\begin{tikzpicture}
	\node[shape=circle,draw=black] (a) 
	at (0,0) {$1$};
	\node[shape=circle,draw=black] (b) 
	at (2.25,0) 
	{$2$};
	\node[shape=circle,draw=black] (c) 
	at (4.5,0) {$3$};
	\node[shape=circle,draw=black] (d) 
	at (6.75,0) {$4$};
	\node[shape=circle,draw=black] (e) 
	at (9,0) {$5$};
	
	\path 
	[->](a) edge [bend left = 0] node {} (b)
	[->](b) edge [bend left = -20] node {} (c)
	[->](c) edge [bend left = -20] node {} (b)
	[->](d) edge [bend left = -20] node {} (c)
	[->](e) edge [bend left = 0] node {} (d);
	\path 
	[<->](c) edge [bend left = -20] node {} (d);
	\path
	[->](a) edge [loop above] node {} (a)
	[->](b) edge [loop above] node {} (b)
	[->](c) edge [loop above] node {} (c)
	[->](d) edge [loop above] node {} (d)
	[->](e) edge [loop above] node {} (e);
	\tikzset{every loop/.style={<->}}
	\path
	[<->](a) edge [loop below] node {} (a)
	[<->](b) edge [loop below] node {} (b)
	[<->](c) edge [loop below] node {} (c)
	[<->](d) edge [loop below] node {} (d)
	[<->](e) edge [loop below] node {} (e);
	\end{tikzpicture}
	\caption{The graph above, $\mathcal{G}=(V,E)$, is Markov equivalent 
		with the graph obtained by adding the edge $5\rightarrow 3$, 
		$\mathcal{N}=(V,F)$ where $F= E\cup \{5 \rightarrow 3\}$. This 
		shows that in- and outdegrees can be different for two Markov 
		equivalent graphs. The graph $\mathcal{N}$ is the greatest element 
		of $\mathcal{G}$'s Markov equivalence class. We see that $4$ cannot 
		be separated from $2$ by any subset of $\{1,3,4,5\}$, and therefore 
		$\id_\mathcal{N}(4) < 
		\con_{\mathcal{G}_1}^{\scaleto{\rightarrow}{1.8pt}}(4)$ and 
		$\od_\mathcal{N}(2) 
		<\con_{\mathcal{G}_1}^{\scaleto{\leftarrow}{1.8pt}}(2)$, even 
		though $\mathcal{N}$ is maximal.}
		\label{fig:adjacDMG}
		\end{figure}

\begin{figure}[H]
	\centering
	\begin{tikzpicture}
	\node[shape=circle,draw=black] (a) 
	at (0,0) {$1$};
	\node[shape=circle,draw=black] (b) 
	at (2.25,0) 
	{$2$};
	\node[shape=circle,draw=black] (c) 
	at (4.5,1) {$3$};
	\node[shape=circle,draw=black] (d) 
	at (4.5,-1) {$4$};
	\node[shape=circle,draw=black] (e) 
	at (6.75,0) {$5$};
	\node[shape=circle,draw=black] (f) 
	at (9,0) {$6$};
	
	\path 
	[->](a) edge [bend left = 0] node {} (b)
	[->](b) edge [bend left = -20] node {} (c)
	[->](c) edge [bend left = -20] node {} (b)
	[->](d) edge [bend left = -20] node {} (e)
	[->](e) edge [bend left = -20] node {} (d)
	[->](f) edge [bend left = 0] node {} (e);
	\path 
	[<->](c) edge [bend left = 0] node {} (e)
	[<->](b) edge [bend left = 0] node {} (d);
	\path
	[->](a) edge [loop above] node {} (a)
	[->](b) edge [loop above] node {} (b)
	[->](c) edge [loop above] node {} (c)
	[->](d) edge [loop above] node {} (d)
	[->](e) edge [loop above] node {} (e)
	[->](f) edge [loop above] node {} (f);
	\tikzset{every loop/.style={<->}}
	\path
	[<->](a) edge [loop below] node {} (a)
	[<->](b) edge [loop below] node {} (b)
	[<->](c) edge [loop below] node {} (c)
	[<->](d) edge [loop below] node {} (d)
	[<->](e) edge [loop below] node {} (e)
	[<->](f) edge [loop below] node {} (f);
	\end{tikzpicture}
	\caption{The graph above, $\mathcal{G}=(V,E)$, is maximal, however there is 
	no set $C\subset \{1,3,4,5,6\}$ such that $5$ is $\mu$-separated from $2$ 
	by $C$, and there is also no set $C\subset \{1,2,3,4,6\}$ such that $2$ is 
	$\mu$-separated from $5$ by $C$.}
	\label{fig:adjacDMG2}
\end{figure}
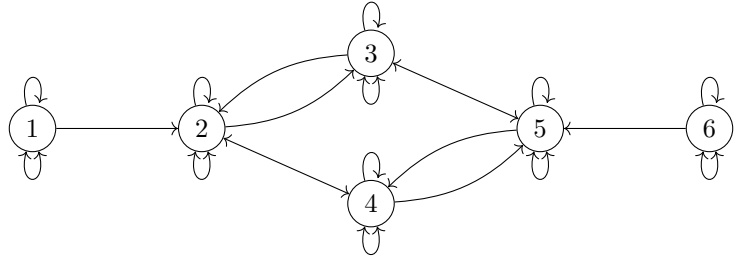

\begin{figure}[H]
	\centering
	\def\x{2.25}
	\def\mps{28pt}
	\begin{tikzpicture}
	\node[shape=circle,draw=black,inner sep=0pt,minimum size=\mps] (z) 
	at (2*\x,0) {$0$};
	\node[shape=circle,draw=black,inner sep=0pt,minimum size=\mps] (c) 
	at (-2*\x,\x) {$3$};
	\node[shape=circle,draw=black,inner sep=0pt,minimum size=\mps] (d) 
	at (-\x,\x/2) 
	{$4$};
	\node[shape=circle,draw=white,inner sep=0pt,minimum size=\mps] (e) 
	at (-2*\x,0) {\vdots};
	\node[shape=circle,draw=white,inner sep=0pt,minimum size=\mps] (f) 
	at (-\x,0) {\vdots};
	\node[shape=circle,draw=black,inner sep=0pt,minimum size=\mps] (h) 
	at (-\x,-\x/2) {$2n$};
	\node[shape=circle,draw=black,inner sep=0pt,minimum size=\mps] (g) 
	at (-2*\x,-\x) {$2n{-}1$};
	\node[shape=circle,draw=black,inner sep=0pt,minimum size=\mps] (b) 
	at (0,0) {$2$};
	\node[shape=circle,draw=black,inner sep=0pt,minimum size=\mps] (a) 
	at (\x,0) {$1$};
	
	\path 
	[->](z) edge [bend left = 0] node {} (a)
	[->](b) edge [bend left = -20] node {} (d)
	[->](d) edge [out = -30, in = 150, looseness = 0] node {} (b)
	[->](b) edge [bend left = -10] node {} (f)
	[->](f) edge [bend left = -10] node {} (b)
	[->](b) edge [bend left = 20] node {} (h)
	[->](h) edge [out = 30, in = -150, looseness = 0] node {} (b)
	[->](c) edge [bend left = 0] node {} (d)
	[->](g) edge [bend left = 0] node {} (h);
	\path 
	[<->](a) edge [bend left = 0] node {} (b);
	\path
	[->](z) edge [loop above] node {} (z)
	[->](a) edge [loop above] node {} (a)
	[->](b) edge [loop above, in=30,out=60, min distance=5mm] node {} (b)
	[->](c) edge [loop above] node {} (c)
	[->](d) edge [loop above] node {} (d)
	[->](g) edge [loop above] node {} (g)
	[->](h) edge [loop, in=120,out=150, min distance=5mm] node {} (h);
	\tikzset{every loop/.style={<->}}
	\path
	[<->](z) edge [loop below] node {} (z)
	[<->](a) edge [loop below] node {} (a)
	[<->](b) edge [loop below, in=300,out=330, min distance=5mm] node {} (b)
	[<->](c) edge [loop below] node {} (c)
	[<->](d) edge [loop below, in=210,out=240, min distance=5mm] node {} (d)
	[<->](g) edge [loop below] node {} (g)
	[<->](h) edge [loop below] node {} (h);
	\end{tikzpicture}
	\caption{We consider a sequence of graphs, $\mathcal{G}_n$, $n \geq 3$, as 
	illustrated 
	above. The graph $\mathcal{G}_n$ has $2n+1$ nodes, and $\mathcal{G}_n$ is 
	maximal for each $n$. For every $n$, the indegree of $1$ in the graph 
	$\mathcal{G}_n$ is three. On the other hand, 
	$\con_{\mathcal{G}_n}^{\scaleto{\rightarrow}{1.8pt}}(1)$ equals $n+2$, thus 
	is unbounded in this family of graphs.}
	\label{fig:adjacDMG3}
\end{figure}
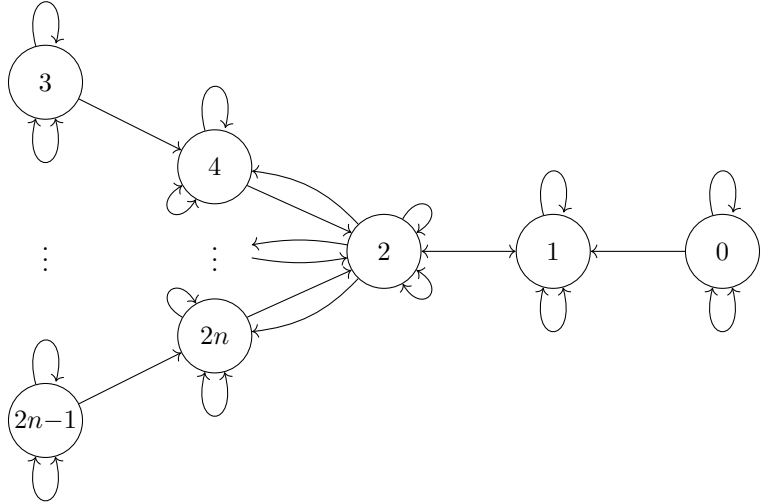

\section{Marginalization}
\label{app:margin}

This section argues that the representation of weak equivalence is closed under 
marginalization in the sense that we can marginalize any graph, $\mathcal{G}$, 
onto a smaller 
node set, $O$, which represents the same independence model as the original 
graph when restricting independence statements to triples $(A,B,C)$ such that 
$A,B,C\subseteq O$. This is formalized in Equation \ref{eq:margin}. A so-called 
\emph{latent projection} of $\mathcal{G}$ satisfies this requirement. The 
latent projection was also used in \cite{Mogensen2020a}, and earlier in 
\cite{vermaEquiAndSynthesis,richardson2017}.

\begin{defn}[Latent projection]
	We denote the latent projection on $\mathcal{G}$ on $O$ by 
	$m(\mathcal{G},O)$.
\end{defn}

The latent projection of a graph on a node set represents a marginalized 
version of the independence model, as formalized by the following corollary. 
\cite{Mogensen2020a} proved this result in the case of $\mathcal{J} = 
\mathcal{P}$, that is, in the case of Markov equivalence \citep[Theorem 
3.12]{Mogensen2020a}. The general case follows 
directly from the Markov equivalence result.

\begin{cor}
	Let $\mathcal{G}(V,E)$, $O\subseteq V$, and let $\mathcal{M} = 
	m(\mathcal{G},O)$. For $A,B,C \subseteq O$, it holds that
	\begin{align*}
		(A,B,C) \in \mathcal{I}_\mathcal{J}(\mathcal{G}) \Leftrightarrow 
		(A,B,C) 
		\in \mathcal{I}_\mathcal{J}(\mathcal{M}).
		\end{align*}
\end{cor}

\begin{proof}
	Theorem 3.12 of \cite{Mogensen2020a} shows that
	
	\begin{align*}
		(A,B,C) \in \mathcal{I}(\mathcal{G}) \Leftrightarrow (A,B,C) \in 
		\mathcal{I}(\mathcal{M}),
	\end{align*}
	
	\noindent and the result follows immediately.
\end{proof}

\cite{Mogensen2020a} stated an algorithm to output the latent projection of a 
DMG (Algorithm 1). This was similar to earlier algorithms in of other 
classes of graphs \cite{koster1999,sadeghi2013}. The following proposition was 
proved by \cite{Mogensen2020a}.

\begin{prop}[\cite{Mogensen2020a}]
	Let $\mathcal{G} = (V,E)$ be a DMG and $O\subseteq V$. Algorithm 
	\ref{algo:marg} outputs its latent projection, $m(\mathcal{G}, O)$. 
\end{prop}

One should note that the marginalization of a (weakly) maximal graph need not 
be (weakly) maximal as illustrated in Figure \ref{fig:margNotMax}.

\begin{figure}[h]
	\begin{subfigure}{.49\textwidth}
		\centering
		\begin{tikzpicture}
		\node[shape=circle,draw=black] (a) at (0,0) {$1$};
		\node[shape=circle,draw=black] (b) at (2,0) 
		{$2$};
		\node[shape=circle,draw=black] (c) at (4,0) {$3$};
		\node[shape=circle,draw=black] (d) at (3,-1) {$4$};
		
		\path 
		[->](a) edge [bend left = 0] node {} (b)
		[->](a) edge [bend left = 20, white] node {} (c)
		[->](b) edge [bend left = 0] node {} (c)
		[->](d) edge [bend left = 0] node {} (c)
		[->](d) edge [bend left = 0] node {} (b);
		;				\end{tikzpicture}
	\end{subfigure}\hfill
	\begin{subfigure}{.49\textwidth}
		\centering
		\begin{tikzpicture}
		\node[shape=circle,draw=black] (a) at (0,0) {$1$};
		\node[shape=circle,draw=black] (b) at (2,0) 
		{$2$};
		\node[shape=circle,draw=black] (c) at (4,0) {$3$};
		\node[shape=circle,draw=black,white] (d) at (3,-1) {$4$};
		
		\path 
		[->](a) edge [bend left = 0] node {} (b)
		[->](b) edge [bend left = 0] node {} (c)
		[<->](c) edge [bend left = 30] node {} (b);
		\end{tikzpicture}
	\end{subfigure}
	\caption{Graphs $\mathcal{G}_1$ (left) and $\mathcal{G}_2$ 
		(right) illustrating that marginalizations of (weakly) maximal 
		graphs need not be (weakly) maximal. $\mathcal{G}_2$ is the 
		marginalization of $\mathcal{G}_1$ over $O = \{ 1,2,3\} $. 
		$\mathcal{G}_1$ is 
		$2$-maximal, and therefore maximal for all $k\geq 2$. 
		$\mathcal{G}_2$ is 
		a marginalization, but is not $k$-maximal for any $k$, $0\leq k 
		\leq 3$. Loops are 
		omitted from the 
		visualization.}
	\label{fig:margNotMax}
\end{figure}
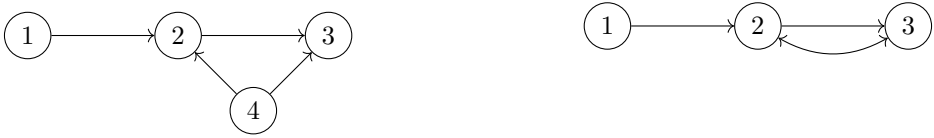 
\begin{algorithm}[H]
	\SetKwInOut{Input}{input}
	\SetKwInOut{Output}{output}
	\Input{$\mathcal{G} = (V,E)$ and $M\subseteq V$}
	\Output{$\mathcal{M} = (O, \bar{E})$}
	Initialize $E_0 = E$, $\mathcal{M}_0 = (V,E_0)$, $k = 0$\;
	\While{$\Omega_M(\mathcal{M}_k) \neq \emptyset$}{
		Choose $\theta =\ _{\theta}(\alpha, m, \beta) \in 
		\Omega_M(\mathcal{M}_k)$\;
		Set $e_{k+1}$ to be the edge between $\alpha$ and $\beta$ which is 
		endpoint-identical to $\theta$\;
		Set $E_{k+1} = E_k \cup \{e_{k+1}\}$\;
		Set $\mathcal{M}_{k+1} = (V,E_{k+1})$\;
		Update $k = k + 1$	
	}
	\Return  $(\mathcal{M}_k)_{O}$
	\newline
	\label{algo:marg}
	\caption{Computing the latent projection of a DMG \citep{Mogensen2020a}. We 
	consider a DMG, $\mathcal{G} = (V,E)$, and $M\subseteq V$ over which we 
	will 
		marginalize. We let $ O = V \setminus M$. A \emph{triroute} is a walk 
		$\alpha\sim \gamma \beta$ such that $\gamma\neq \alpha,\beta$. This is 
		similar 
		to a \emph{tripath} in \cite{lauritzen2018}, but we allow 
		$\alpha=\beta$. We say that a triroute is \emph{noncolliding} if 
		$\gamma$ is not a collider on the triroute. We say that heads and tails 
		are \emph{(edge) marks}. We say that two walks between 
		$\alpha$ and $\beta$, $(\alpha, e_1^\alpha \ldots, e_1^\beta,\beta)$ 
		and $(\alpha, e_2^\alpha \ldots, e_2^\beta,\beta)$, are 
		\emph{endpoint-identical} if 
		$e_1^\alpha$ and $e_2^\alpha$ have the same mark at $\alpha$ and 
		$e_1^\beta$ and $e_2^\beta$ have the same mark at $\beta$ (note that 
		this may depend on the orientation of directed edges on the walk). We 
		say that a walk between $\alpha$ and $\beta$, $\omega$, is 
		endpoint-identical with an edge $e$ between $\alpha$ and $\beta$ if 
		$\omega$ is endpoint-identical with the walk $(\alpha,e,\beta)$. The 
		set 
		$\Omega_M(\mathcal{G})$ is the set of noncolliding triroutes, $\alpha 
		\sim m \sim \beta$, such that $m\in M$ and such that an 
		endpoint-identical edge $\alpha \sim \beta$ is not in $\mathcal{G}$. We 
		let $(\mathcal{G})_O$ denote the \emph{induced subgraph} on 
		node set 
		$O$, that is, $(\mathcal{G})_O = (O,E_O)$ where $E_O$ is the subset of 
		$E$ 
		consisting of all edges, $e$, such that when $e$ is between $\alpha$ 
		and 
		$\beta$ then both of these edges are in $O$.
		}
\end{algorithm}

\section{Proofs and lemmas}
\label{app:proofs}

The proofs of the following lemmas are adaptations of the 
proofs of Lemmas 5.4 and 5.8 in \cite{Mogensen2020a}. We include them for 
completeness to show how the appropriate changes are made. Lemmas 5.4 and 5.8 
in \cite{Mogensen2020a} did not use the $C$-specific 
conditions that are essential in obtaining the stronger results that we present 
in this paper. 

\begin{defn}[Route]
	We say that a walk, $\omega 
	=(\gamma_1,e_1,\gamma_2,\ldots,e_l,\gamma_{l+1})$, is a \emph{route} if the 
	node $\gamma_{l+1}$ occurs at most twice on $\omega$ and no other node 
	occurs more than once on $\omega$.
	\label{def:route}
\end{defn}

Routes characterize $\mu$-connections in DMGs \citep{Mogensen2020a}, and we use 
them in the next proofs. Note that the below lemmas are formulated using 
$\mathcal{I}(\mathcal{G})$, not the 
restricted version  $\mathcal{I}_\mathcal{J}(\mathcal{G})$.

\begin{lem}
	Let $C \subseteq V$ and let $e$ be a $C$-potential sibling edge between 
	$\alpha$ and $\beta$ in 
	$\mathcal{I}(\mathcal{G})$. Let $\gamma,\delta\in V$.
	If there is a 
	$\mu$-connecting walk from $\gamma$ to $\delta$ given $C$ in $\mathcal{G} + 
	e$, then 
	there is a $\mu$-connecting walk from $\gamma$ to $\delta$ given $C$ in 
	$\mathcal{G}$.
	\label{lem:Cps}
	\end{lem}
	
\begin{proof}
	Consider any $\mu$-connecting walk from $\gamma$ to $\delta$ given $C$ in 
	$\mathcal{G} + e$. We can also find a $\mu$-connecting route from $\gamma$ 
	to $\delta$ given $C$ in $\mathcal{G} + e$ \citep{Mogensen2020a}, and we 
	denote this route by 
	$\rho$. If $\alpha\notin C$, then there exists a $\mu$-connecting walk from 
	$\alpha$ to $\beta$ given $C$ in $\mathcal{G}$ using (cs1) of Definition 
	\ref{def:Cps}. If $\beta\notin C$, then there exists a $\mu$-connecting 
	walk from $\beta$ to $\alpha$ given $C$ in $\mathcal{G}$, also using (cs1). 
	We denote these walks by $\nu_1$ and $\nu_2$, respectively, if they exist.
	
	If $e$ does not occur on $\rho$, then $\rho$ is $\mu$-connecting given $C$ 
	in $\mathcal{G}$. If $e$ occurs twice, then either $\rho$ contains a 
	subroute $\alpha\leftrightarrow\beta\leftrightarrow\alpha$ and 
	$\delta=\alpha$ or $\rho$ contains a subroute 
	$\beta\leftrightarrow\alpha\leftrightarrow\beta$ and $\delta=\beta$. Assume 
	first the former. There is either a $\mu$-connecting subroute from $\gamma$ 
	to $\alpha$, or $\alpha\notin C$. If $\beta\in C$, then consider the 
	subroute between $\gamma$ and $\alpha$. This subroute is either trivial or 
	has a tail 
	at $\alpha$. In either case, composing it with $\nu_1$ gives a 
	$\mu$-connecting walk from $\gamma$ to $\beta$ given $C$ 
	in $\mathcal{G}$, and using (cs2) there is also a $\mu$-connecting walk 
	from $\gamma$ to $\alpha$ given $C$ in $\mathcal{G}$. If $\beta\notin C$, 
	then we can compose the subroute from $\gamma$ to $\alpha$ with $\nu_1$ and 
	$\nu_2$. 
	The resulting walk will be $\mu$-connecting as $\beta\in \an(C)\setminus 
	C$. The argument is the same when 
	$\beta\leftrightarrow\alpha\leftrightarrow\beta$ and $\delta=\beta$.
	
	We now assume that $e$ occurs only once on $\rho$ and assume first that 
	
	\begin{align*}
		\underbrace{\gamma \sim \ldots \sim \alpha}_{\rho_1} \leftrightarrow 
		\underbrace{\beta \sim \ldots \starrightarrow
		\delta}_{\rho_2}.
	\end{align*}
	
	\noindent If $\alpha\notin C$, then we can compose $\rho_1$, $\nu_1$, and 
	$\rho_2$ to obtain a $\mu$-connecting walk given $C$. Note that this also 
	holds if $\rho_1$ is trivial. If $\alpha\in C$, then $\rho_1$ is not 
	trivial and it has a head at $\alpha$. Using (cs3), there exists a 
	$\mu$-connecting walk from $\gamma$ to $\beta$ and composing it with 
	$\rho_2$ gives the result. If instead 
	
	\begin{align*}
	\gamma \sim \ldots \sim \beta \leftrightarrow 
	\alpha \sim \ldots \starrightarrow
		\delta,
	\end{align*}
	
	\noindent the same arguments work, now using (cs2).
\end{proof}
	
	\begin{lem}
		Let $C \subseteq V$ and let $e$ be a $C$-potential parent edge from 
		$\alpha$ to $\beta$ in 
		$\mathcal{I}(\mathcal{G})$. Let $\gamma,\delta\in V$.
		If there is a 
		$\mu$-connecting walk from $\gamma$ to $\delta$ given $C$ in 
		$\mathcal{G} + 
		e$, then 
		there is a $\mu$-connecting walk from $\gamma$ to $\delta$ given $C$ in 
		$\mathcal{G}$.
		\label{lem:Cpp}
		\end{lem}
		
\begin{proof}
	We consider a $\mu$-connecting walk from $\gamma$ to $\delta$ given $C$ in 
	$\mathcal{G} + e$. If $\alpha\notin C$, then by (cp1) there exists a 
	$\mu$-connecting walk from $\alpha$ to $\beta$ given $C$, and we denote 
	this walk by $\nu$ when it exists. We can find a $\mu$-connecting route 
	from $\gamma$ to $\delta$ given $C$ in $\mathcal{G} + e$, and we denote 
	this route by $\rho$.
	
	In this proof, we will say that a collider on a walk is \emph{newly closed} 
	if the collider is in $\an_{\mathcal{G} + e}(C)$, but not in 
	$\an_\mathcal{G}(C)$. If there exists a newly closed collider, then 
	$\alpha\notin C$ and $\beta\in\an_\mathcal{G}(C)$. We assume first that $e$ 
	occurs at most once on $\rho$. If there are newly closed colliders on 
	$\rho$, the proof of Lemma 5.8 in \cite{Mogensen2020a} shows that we can 
	find a $\mu$-connecting 
	walk in $\mathcal{G} + e$ with no newly closed colliders such that $e$ 
	occurs at most once, and we denote this walk by $\tilde{\omega}$.
	
	If $\tilde{\omega}$ does not contain $e$, then the result follows. If it 
	does contain $e$, we split into two cases. Assume first that
	
	\begin{align*}
	\underbrace{\gamma \sim \ldots \sim \alpha}_{\rho_1} \rightarrow 
	\underbrace{\beta \sim \ldots \starrightarrow
		\delta}_{\rho_2}.
	\end{align*}
	
	We see that $\alpha\notin C$. If $\rho_1$ is trivial or if it has a tail at 
	$\alpha$, then composing $\rho_1$, $\nu$, and $\rho_2$ gives a 
	$\mu$-connecting walk. If $\rho_1$ has a head at $\alpha$, then (cp2) gives 
	a $\mu$-connecting walk from $\gamma$ to $\beta$ that we can compose with 
	$\rho_2$. Assume instead that 
	
	\begin{align*}
	\underbrace{\gamma \sim \ldots \sim \beta}_{\rho_1} \leftarrow 
	\underbrace{\alpha \sim \ldots \starrightarrow
		\delta}_{\rho_2}.
	\end{align*}
	
	If $\rho_1$ has a head at $\beta$ and $\beta\in C$, then (cp3) gives the 
	result. If $\beta\notin C$, we can find a walk in $\mathcal{G}+e$ with no 
	newly closed colliders and only one occurrence of $e$ of the type

	\begin{align*}
	\underbrace{\gamma \sim \ldots \leftarrow \beta}_{\rho_1} \leftarrow 
	\underbrace{\alpha \sim \ldots \starrightarrow
		\delta}_{\rho_2}.
	\end{align*}

	\noindent where $\rho_1$ can be trivial, using the same argument as in the 
	proof of 
	Lemma 5.8 in \cite{Mogensen2020a}. We have $\alpha,\beta\notin C$ and there 
	is a $\mu$-connecting walk from $\alpha$ to $\delta$. Using (cp4) there is 
	also one from $\beta$ to $\delta$. Composing this with $\rho_1$ gives the 
	result since $\rho_1$ is either trivial or has a tail at $\beta$.
	
	Finally, if $e$ occurs twice on $\rho$, we must have $\alpha\notin C$. We 
	can use the same arguments as in the proof of Lemma 5.8 in 
	\cite{Mogensen2020a} using the walk $\nu$ and condition (cp2). 
\end{proof}

\section{Additional results}
\label{app:add}

When we count the number of colliders on a 
walk, we count them with 
multiplicity, that is, if 

$$
(\gamma_1 ,e_1, \gamma_2 ,e_2,\ldots, e_{l-1}, \gamma_{l} ,e_l,\gamma_{l+1})
$$

\noindent is a walk, $\omega$, then the number of colliders on this walk equals 
the number of $i$, $2\leq i\leq l$, such that $e_{l-1}$ and $e_i$ both have 
heads at $\gamma_i$ on $\omega$. Note that the 
endpoints, $\gamma_1$ and $\gamma_{l+1}$ are not colliders by definition. The 
next 
lemma is useful for giving a characterization of $k$-weak equivalence 
in terms of $\mu$-connecting walks. 

\begin{lem}
	If there is a $\mu$-connecting walk from $\alpha$ 
	to 
	$\beta$ given $C$ in $\mathcal{G}$, then there is a 
	$\mu$-connecting walk from $\alpha$ to $\beta$ given $C$ in $\mathcal{G}$ 
	with at most $\vert C\vert$ colliders, all of which are in $C$.
	\label{lem:collInC}
\end{lem}

\begin{proof}
	Let $\gamma_1,\gamma_2,\ldots,\gamma_l$ denote the colliders on the 
	$\mu$-connecting walk. We know 
	that 
	$\gamma_i\in \an(C)$ and therefore there exist a directed path 
	$\gamma_i 
	\rightarrow \delta_1 \rightarrow \ldots \rightarrow \delta_{l_i}$ such 
	that 
	$\delta_{l_i}\in C$ and such that $\delta_{l_i}\in C$ is the only node 
	in 
	$C$ on this directed path. If $\gamma_i\in C$, then the path is 
	trivial, 
	that is, contains no edges and just a single node, $\gamma_i$. Adding 
	$\gamma_i 
	\rightarrow \delta_1 \rightarrow \ldots \rightarrow \delta_{l_i} 
	\leftarrow 
	\ldots \leftarrow \delta_1 \leftarrow \gamma_i$ for each $i$ creates a 
	walk which is 
	$\mu$-connecting from $\alpha$ to $\beta$ given $C$ such that every 
	collider is in $C$. If a node occurs as a collider more than once, we 
	can remove 
	the 
	loop. The resulting walk is also $\mu$-connecting, also if $\beta$ is a 
	collider, and it has strictly fewer colliders. We 
	can 
	repeat this to find a $\mu$-connecting walk with at most $\vert C\vert$ 
	colliders.
\end{proof}

\begin{prop}
	Let $\mathcal{G}$ be a DMG. Let 
	$\alpha,\beta\in V$ and $C\subseteq V$ such that $\vert C\vert \leq 
	k$. We 
	have $(\alpha,\beta,C) \in \mathcal{I}_k(\mathcal{G})$ if and only 
	if there 
	is no $\mu$-connecting walk from $\alpha$ to $\beta$ given $C$ in 
	$\mathcal{G}$ with at most $k$ colliders.
	\label{prop:kCollIk}
\end{prop}

\begin{proof}
	If there is a $\mu$-connecting walk given $C$, then clearly 
	$(\alpha,\beta,C)\notin \mathcal{I}_k(\mathcal{G})$. On the 
	other hand, if 
	$(\alpha,\beta,C)\notin \mathcal{I}_k(\mathcal{G})$ then there 
	is a 
	$\mu$-connecting walk from $\alpha$ to $\beta$ given $C$ and 
	Lemma 
	\ref{lem:collInC} gives the result.
\end{proof}

This means that the restriction of the independence models to 
$k$-weak equivalence ignores
$\mu$-connecting walks 
with more than $k$ colliders.

\begin{cor}
	Graphs $\mathcal{G}_1$ and $\mathcal{G}_2$ are $k$-weak 
	equivalent if and only if it 
	holds for all 
	$\alpha,\beta\in V$ and $C\subseteq V$ such that $\vert 
	C\vert\leq k$ that there 
	is 
	a 
	$\mu$-connecting walk from $\alpha$ to $\beta$ given $C$ in 
	$\mathcal{G}_1$ 
	with at most $k$ colliders if and only if there is a 
	$\mu$-connecting walk from $\alpha$ to $\beta$ given $C$ in 
	$\mathcal{G}_2$ 
	with at most $k$ colliders. 
	\label{cor:kEqkColl}
\end{cor}

\begin{proof}
	Assume first that $\mathcal{G}_1\in [\mathcal{G}_2]_k$, 
	and that $\omega$ 
	is a $\mu$-connecting walk from $\alpha$ to $\beta$ 
	given $C$, $\vert 
	C\vert \leq C$, in $\mathcal{G}_1$ with at most $k$ 
	colliders. Proposition 
	\ref{prop:kCollIk} gives that $(\alpha,\beta,C)\notin 
	\mathcal{I}_k(\mathcal{G}_1)$ and therefore 
	$(\alpha,\beta,C)\notin 
	\mathcal{I}_k(\mathcal{G}_2)$. Using Proposition 
	\ref{prop:kCollIk} again gives 
	the result.
	
	Assume instead that for all $\alpha,\beta, C$ such that 
	$\vert C\vert\leq 
	k$ it holds that there is a $\mu$-connecting walk from 
	$\alpha$ to $\beta$ 
	given $C$ with at most $k$ colliders in $\mathcal{G}_1$ 
	if and only if there is 
	one in 
	$\mathcal{G}_2$. If $(\alpha,\beta,C) \in 
	\mathcal{I}_k(\mathcal{G}_1)$, $\alpha\notin C$, then 
	there is no 
	$\mu$-connecting walk from 
	$\alpha$ to $\beta$ 
	given $C$ in $\mathcal{G}_1$ and therefore 
	also no $\mu$-connecting walk with at most $k$ 
	colliders in $\mathcal{G}_2$, and Propositions 
	\ref{prop:singletonGraphIndep} and \ref{prop:kCollIk}
	give the result.
\end{proof}


\end{document}